\documentclass[12pt,svgnames]{article}
\usepackage{amsmath}
\usepackage{amssymb}
\usepackage{amsthm}
\usepackage{wasysym}
\usepackage{enumerate}
\usepackage{dsfont}
\usepackage{stmaryrd}
\usepackage{comment}
\usepackage{tikz}
\usepackage{tikz-cd}
\usepackage{extarrows}
\usepackage{mathrsfs}
\usepackage{enumitem}
\usepackage{fullpage}
\usepackage{hyperref}
\usepackage{accents}
\usepackage{titlesec}
\usepackage[all]{xy}
\hypersetup{
    linkbordercolor = {PaleTurquoise},
    citebordercolor = {PaleGreen},
}

\titleformat{\section}{\normalsize\bfseries}{\thesection}{1em}{}
\titleformat{\subsection}{\normalsize\bfseries}{\thesubsection}{1em}{}

\numberwithin{equation}{subsection}

\theoremstyle{definition}

\newtheorem{para}[subsection]{}
\newtheorem{ex}[subsection]{Example}
\newtheorem{rmk}[subsection]{Remark}

\newtheorem*{assumption*}{Assumption}

\newtheorem{defn_sub}[subsubsection]{Definition}
\newtheorem{para_sub}[subsubsection]{}

\newtheorem{rmk_sub}[subsubsection]{Remark}

\theoremstyle{plain}

\newtheorem{prop}[subsection]{Proposition}
\newtheorem{theo}[subsection]{Theorem}

\newtheorem*{claim*}{Claim}
\newtheorem*{just*}{Justification}
\newtheorem*{lem*}{Lemma}
\newtheorem*{prop*}{Proposition}

\newtheorem{prop_sub}[subsubsection]{Proposition}
\newtheorem{theo_sub}[subsubsection]{Theorem}
\newtheorem{lem_sub}[subsubsection]{Lemma}
\newtheorem{cor_sub}[subsubsection]{Corollary}

\newcommand{\J}{\mathscr{J}}
\newcommand{\V}{\mathscr{V}}

\newcommand{\ob}{\mathop{\mathsf{ob}}}

\newcommand{\tensor}{\otimes}
\newcommand{\Mnd}{\mathsf{Mnd}}

\newcommand{\Cat}{\text{-}\mathsf{Cat}}
\newcommand{\id}{\mathsf{id}}

\newcommand{\underJ}{\kern -0.5ex \mathscr{J}}

\newcommand{\Set}{\mathsf{Set}}
\newcommand{\ev}{\mathsf{ev}}
\newcommand{\op}{\mathsf{op}}
\newcommand{\Th}{\mathsf{Th}}
\newcommand{\T}{\mathbb{T}}
\newcommand{\C}{\mathscr{C}}

\newcommand{\A}{\mathscr{A}}

\newcommand{\Mod}{\text{-}\mathsf{Mod}}

\newcommand{\y}{\mathsf{y}}
\newcommand{\scrS}{\mathscr{S}}

\newcommand{\Alg}{\text{-}\mathsf{Alg}}

\newcommand{\N}{\mathbb{N}}

\newcommand{\Term}{\mathsf{Term}}

\newcommand{\scrT}{\mathscr{T}}

\newcommand{\X}{\mathscr{X}}

\newcommand{\E}{\mathcal{E}}

\newcommand{\bbC}{\mathbb{C}}

\newcommand{\calE}{\mathcal{E}}

\newcommand{\Var}{\mathsf{Var}}

\newcommand{\calT}{\mathcal{T}}

\newcommand{\Psh}{\text{-}\mathsf{Psh}}

\newcommand{\sfE}{\mathsf{E}}

\newcommand{\calS}{\mathcal{S}}

\newcommand{\Sbar}{\overline{S}}
\newcommand{\Indisc}{\mathop{\mathsf{Indisc}}}
\newcommand{\Disc}{\mathop{\mathsf{Disc}}}
\newcommand{\congr}{\left(\sim_S\right)_{S \in \calS}}

\newcommand{\calB}{\mathcal{B}}
\newcommand{\calC}{\mathcal{C}}

\newcommand{\CGTop}{\mathsf{CGTop}}
\newcommand{\Fib}{\mathsf{Fib}}

\newcommand{\calJ}{\mathcal{J}}
\newcommand{\calP}{\mathcal{P}}
\newcommand{\arity}{\mathsf{ar}}

\newcommand{\scrE}{\mathscr{E}}
\newcommand{\Top}{\mathsf{Top}}
\newcommand{\Meas}{\mathsf{Meas}}
\newcommand{\calO}{\mathcal{O}}
\newcommand{\Preord}{\mathsf{Preord}}
\newcommand{\PMet}{\mathsf{PMet}}

\newcommand{\Simp}{\mathsf{Simp}}
\newcommand{\Born}{\mathsf{Born}}
\newcommand{\calQ}{\mathcal{Q}}
\newcommand{\QuasiSp}{\mathsf{QuasiSp}}

\newcommand{\sfe}{\mathsf{e}}
\newcommand{\vbar}{\overline{v}}

\newcommand{\VCAT}{\V\text{-}\mathsf{CAT}}
\newcommand{\frakT}{\mathfrak{T}}

\begin{document}

\title{\Large \textbf{Free algebras of topologically enriched \\ multi-sorted equational theories}}
\author{Jason Parker\let\thefootnote\relax\thanks{We acknowledge the support of the Natural Sciences and Engineering Research Council of Canada (NSERC), [funding reference numbers RGPIN-2019-05274, RGPAS-2019-00087, DGECR-2019-00273].  Cette recherche a \'et\'e financ\'ee par le Conseil de recherches en sciences naturelles et en g\'enie du Canada (CRSNG), [num\'eros de r\'ef\'erence RGPIN-2019-05274, RGPAS-2019-00087, DGECR-2019-00273].} \medskip
\\
\small Brandon University, Brandon, Manitoba, Canada}
\date{}

\maketitle

\begin{abstract}
Classical \emph{multi-sorted equational theories} and their \emph{free algebras} have been fundamental in mathematics and computer science. In this paper, we present a generalization of multi-sorted equational theories from the classical (\emph{$\Set$-enriched}) context to the context of enrichment in a symmetric monoidal category $\V$ that is topological over $\Set$. Prominent examples of such categories include: various categories of topological and measurable spaces; the categories of models of \emph{relational Horn theories without equality}, including the categories of preordered sets and (extended) pseudo-metric spaces; and the categories of \emph{quasispaces} (a.k.a.~\emph{concrete sheaves}) on \emph{concrete sites}, which have recently attracted interest in the study of programming language semantics.

Given such a category $\V$, we define a notion of \emph{$\V$-enriched multi-sorted equational theory}. We show that every $\V$-enriched multi-sorted equational theory $\calT$ has an underlying classical multi-sorted equational theory $|\calT|$, and that free $\calT$-algebras may be obtained as suitable liftings of free $|\calT|$-algebras. We establish explicit and concrete descriptions of free $\calT$-algebras, which have a convenient \emph{inductive} character when $\V$ is cartesian closed. We provide several examples of $\V$-enriched multi-sorted equational theories, and we also discuss the close connection between these theories and the presentations of $\V$-enriched algebraic theories and monads studied in recent papers by the author and Lucyshyn-Wright.
\end{abstract}

\section{Introduction}

Classical \emph{multi-sorted equational theories} and their \emph{free algebras} have been fundamental in mathematics and computer science. For example, they have been prominently employed in studying algebraic specification and abstract algebraic datatypes \cite{Mitchell_foundations}, computational effects \cite{PlotkinPower}, and algebraic databases and data integration \cite{Alg_databases, Alg_int}. Given a set $\calS$ of \emph{sorts}, recall that a \emph{(classical) $\calS$-sorted signature} is a set $\Sigma$ of \emph{operation} (or \emph{function}) \emph{symbols} equipped with an assignment to each operation symbol $\sigma \in \Sigma$ of a finite tuple $(S_1, \ldots, S_n) \in \calS^*$ of \emph{input sorts} and an \emph{output sort} $S \in \calS$, in which case we write $\sigma : S_1 \times \ldots \times S_n \to S$. A \emph{$\Sigma$-algebra} $A$ is an $\calS$-sorted family of \emph{carrier sets} $A = \left(A_S\right)_{S \in \calS}$ equipped with, for each $\sigma : S_1 \times \ldots \times S_n \to S$ in $\Sigma$, a function
\[ \sigma^A : A_{S_1} \times \ldots \times A_{S_n} \to A_S. \] A \emph{(classical) $\calS$-sorted equational theory} is a pair $\calT = (\Sigma, \calE)$ consisting of a classical $\calS$-sorted signature $\Sigma$ and a set $\calE$ of \emph{syntactic equations} between \emph{terms} over the signature $\Sigma$, and a \emph{$\calT$-algebra} is a $\Sigma$-algebra that satisfies these equations. Writing $\calT\Alg$ for the category of $\calT$-algebras, the forgetful functor $U^\calT : \calT\Alg \to \Set^\calS$ that sends a $\calT$-algebra to its underlying $\calS$-sorted family of carrier sets has a left adjoint $F^\calT : \Set^\calS \to \calT\Alg$ with a well-known explicit and constructive description in terms of \emph{term algebras}. The \emph{initial $\calT$-algebra} can be obtained as the free $\calT$-algebra on the $\calS$-sorted family constant at the empty set, and this $\calT$-algebra is often viewed as the ``intended model'' of the multi-sorted equational theory $\calT$.

It is well known that classical $\calS$-sorted equational theories can be regarded as \emph{presentations} of \emph{(classical) $\calS$-sorted algebraic} (or \emph{Lawvere}) \emph{theories} \cite{Law:PhD, Benabou_structures, Algebraic_theories}, where the latter classify categories of algebras for $\calS$-sorted equational theories up to a suitable notion of isomorphism. It is also known that one can \emph{enrich} this notion of classical $\calS$-sorted algebraic theory by replacing the role of $\Set$ with that of a suitable symmetric monoidal closed category $\V$, thus yielding a notion of \emph{$\V$-enriched $\calS$-sorted algebraic theory}. More generally (and precisely), given a symmetric monoidal closed category $\V$ with equalizers and a \emph{subcategory of arities} (i.e.~a dense full sub-$\V$-category) $\J \hookrightarrow \C$ in a $\V$-category $\C$, a \emph{$\J$-theory} is a $\V$-category $\scrT$ equipped with an identity-on-objects $\V$-functor $\tau : \J^\op \to \scrT$ satisfying a certain condition. The notion of $\J$-theory was defined and studied in such generality in \cite{Struct}, generalizing and extending prior work in \cite{Law:PhD, Lintonequational, Dubucsemantics, BorceuxDay, PowerLawvere, NishizawaPower, LR, EAT, BourkeGarner}. Under suitable conditions, the category $\Th_{\underJ}(\C)$ of $\J$-theories is equivalent to the category $\Mnd_{\underJ}(\C)$ of \emph{$\J$-ary} (or \emph{$\J$-nervous}) $\V$-monads on $\C$ (see \cite[Theorem 4.13]{Struct}). Taking $\C$ to be the product $\V$-category $\V^\calS$, there is a specific subcategory of arities denoted $\N_\calS \hookrightarrow \V^\calS$ (defined in \S\ref{smc_section} below) for which $\N_\calS$-theories provide a notion of \emph{$\V$-enriched $\calS$-sorted algebraic theory}. In particular, with $\V = \Set$, these $\N_\calS$-theories are precisely the classical $\calS$-sorted algebraic theories, while $\N_\calS$-ary monads are precisely the finitary monads on $\Set^\calS$.  

Generalizing the situation for classical $\calS$-sorted algebraic theories, \emph{presentations} of $\J$-theories and $\J$-ary $\V$-monads by generalized signatures and equations were initially studied in certain specific settings in \cite{KellyPower, KellyLackstronglyfinitary}, and recently in more general settings in \cite{Pres}. In \cite{EP, Ax}, we then demonstrated that the presentations studied in \cite{KellyPower, KellyLackstronglyfinitary, Pres} can be concretely reformulated as \emph{diagrammatic $\J$-presentations} consisting of \emph{parameterized $\J$-ary operations} and \emph{diagrammatic equations}, closely mirroring the equational presentations of classical $\calS$-sorted algebraic theories recalled above. In particular, given a symmetric monoidal closed category $\V$ and a set $\calS$ of sorts, the notion of \emph{diagrammatic $\N_\calS$-presentation} (relative to the aforementioned subcategory of arities $\N_\calS \hookrightarrow \V^\calS$) provides a concrete notion of presentation for $\V$-enriched $\calS$-sorted algebraic theories.  

The purpose of the present work is to thoroughly study these presentations for $\V$-enriched $\calS$-sorted algebraic theories in the special and simplifying case where $\V$ is a symmetric monoidal category equipped with a \emph{topological (forgetful) functor} $|-| : \V \to \Set$ (in the sense of \cite[\S 21]{AHS}) that is strict symmetric monoidal (with respect to the cartesian monoidal structure on $\Set$). For such $\V$, we shall thereby develop a treatment of \emph{$\V$-enriched multi-sorted equational theories} that generalizes the treatment of classical multi-sorted equational theories. Prominent examples of such $\V$ include:
\begin{itemize}[leftmargin=*]
\item Various categories of topological spaces (including the category $\Top$ of all topological spaces and continuous functions), and the category $\Meas$ of measurable spaces.
\item The categories of models of \emph{relational Horn theories without equality} \cite{Monadsrelational, Extensivity}, including the category of preordered sets and monotone functions, and the category of (extended) pseudo-metric spaces and non-increasing functions.
\item The categories of \emph{quasispaces} (a.k.a.~\emph{concrete sheaves}) on \emph{concrete sites}, which have recently attracted interest in the study of programming language semantics \cite{Heunenprobability, Concrete_cats_recursion}, and include the categories of diffeological spaces, quasi-Borel spaces \cite{Heunenprobability}, bornological sets, (abstract) simplicial complexes, and convergence spaces \cite{Baezsmooth, Dubucquasitopoi}. Such categories are also known as \emph{concrete quasitoposes} \cite{Dubucquasitopoi}.
\end{itemize} 

Note that here we actually do \emph{not} assume that the symmetric monoidal structure of $\V$ is closed: this is so that we can accommodate examples such as $\Top$ and $\Meas$  with the (non-closed) cartesian monoidal structure. Thus, the (presentations of) $\V$-enriched $\calS$-sorted algebraic theories that we shall consider in this paper are, in the special case where $\V$ is topological over $\Set$, slightly more general than those studied in \cite{EP, Struct, Ax} (where $\V$ is always assumed to be closed).
  
In detail, given such a category $\V$ and a set $\calS$ of sorts, we first define a notion of \emph{$\V$-enriched $\calS$-sorted signature} $\Sigma$, which extends the notion of classical $\calS$-sorted signature by requiring that each operation symbol $\sigma \in \Sigma$ be equipped also with a \emph{parameter object} $P$ of $\V$. A \emph{$\Sigma$-algebra} $A$ is then an $\calS$-sorted family $A = \left(A_S\right)_{S \in \calS}$ of \emph{carrier objects of $\V$} equipped with, for each $\sigma : S_1 \times \ldots \times S_n \to S$ in $\Sigma$ with parameter object $P$, a \emph{$\V$-morphism}
\[ \sigma^A : P \tensor \left(A_{S_1} \times \ldots \times A_{S_n}\right) \to A_S. \] A \emph{$\V$-enriched $\calS$-sorted equational theory $\calT$} is a pair $\calT = (\Sigma, \calE)$ consisting of a $\V$-enriched $\calS$-sorted signature $\Sigma$ and a set $\calE$ of \emph{algebraic $\Sigma$-equations}, while a \emph{$\calT$-algebra} is a $\Sigma$-algebra that satisfies these equations. As we show in \S\ref{smc_section}, when the symmetric monoidal structure of $\V$ is closed, this definition is essentially an unpacking of the aforementioned notion of \emph{diagrammatic $\N_\calS$-presentation} \cite{EP, Ax} relative to the aforementioned subcategory of arities $\N_\calS \hookrightarrow \V^\calS$. The category $\calT\Alg$ of $\calT$-algebras is equipped with a forgetful functor $U^\calT : \calT\Alg \to \V^\calS$ that sends a $\calT$-algebra to its underlying $\calS$-sorted family of carrier objects of $\V$; when $\V$ is symmetric monoidal closed, the data of $U^\calT : \calT\Alg \to \V^\calS$ become $\V$-enriched. 

We show that every $\V$-enriched $\calS$-sorted equational theory $\calT = (\Sigma, \calE)$ has an underlying classical $\calS$-sorted equational theory $|\calT| = (|\Sigma|, |\calE|)$. In fact, crucially using the assumption that $\V$ is topological over $\Set$, we show that a $\V$-enriched $\calS$-sorted equational theory $\calT$ is equivalently given by a $\V$-enriched $\calS$-sorted signature $\Sigma$ and a set of syntactic equations between terms over the underlying classical $\calS$-sorted signature $|\Sigma|$. We establish that the forgetful functor $U^\calT$ has a left adjoint $F^\calT : \V^\calS \to \calT\Alg$, and that the resulting adjunction $F^\calT \dashv U^\calT : \calT\Alg \to \V^\calS$ is a \emph{(strict) lifting} of the adjunction $F^{|\calT|} \dashv U^{|\calT|} : |\calT|\Alg \to \Set^\calS$, where the latter result again relies on the assumption that $\V$ is topological over $\Set$. In particular, given an $\calS$-sorted family $X = \left(X_S\right)_{S \in \calS}$ of objects of $\V$, the free $\calT$-algebra on $X$ can be realized as the free $|\calT|$-algebra (constructed as a certain \emph{term algebra}) on the underlying $\calS$-sorted family of sets $|X| = \left(\left|X_S\right|\right)_{S \in \calS}$, equipped with an appropriate $\V$-structure. We use these results to establish explicit descriptions of free $\calT$-algebras, which have a more concrete and (\emph{countably}) \emph{inductive} character when $\V$ is cartesian closed.

Explicit descriptions of free algebras for certain examples of $\V$-enriched \emph{single-sorted} equational theories (with $\V$ topological over $\Set$) have been previously studied in the literature (early sources include \cite{Markoff, Malcev_free, Wyler_top_alg, Marxen, Taylor_top_alg, Nel_top_alg}). For example, \cite{Porst_existence} explicitly describes free topological groups, taking $\calT$ to be the (trivially $\V$-enriched) single-sorted equational theory for groups and $\V$ to be the category $\Top$ of topological spaces or the category $\CGTop$ of compactly generated topological spaces. \cite{Porst_cc} explicitly describes free $\calT$-algebras in $\V$ for an arbitrary (trivially $\V$-enriched) classical single-sorted equational theory $\calT$ and an arbitrary cartesian closed topological category $\V$ over $\Set$ (of which $\CGTop$ is an example). Given a ring object $R$ in a cartesian closed topological category $\V$ over $\Set$, \cite{Seal_modules} studies free $R$-modules in $\V$, which can be described as the free algebras for a certain $\V$-enriched single-sorted equational theory. In \cite{Battenfeld_comp}, free algebras of $\V$-enriched single-sorted equational theories with $\V = \Top, \CGTop$, and certain categories in topological domain theory are described.

We now outline the paper. After recalling some relevant background material on concrete and topological categories in \S\ref{background}, in \S\ref{first_section} we define and study $\V$-enriched multi-sorted signatures and their free algebras. These free algebras can be somewhat explicitly described under (just) our general assumptions on $\V$ (as we show in Theorems \ref{taut_sig_thm} and \ref{explicit_free_fibre_prop}), but they can be described even more concretely and constructively when $\V$ is cartesian closed (as we show in Theorem \ref{free_sig_alg_cc}). 

In \S\ref{theories_section} we define and study $\V$-enriched multi-sorted equational theories and their free algebras. We initially define such theories in terms of $\V$-enriched multi-sorted signatures $\Sigma$ and \emph{algebraic $\Sigma$-equations}, which are pairs of natural transformations between certain functors $\Sigma\Alg \to \V$. We then show in Theorem \ref{equiv_theory_prop} that (because $\V$ is topological over $\Set$) these theories can be given a more syntactic formulation that is similar to the syntactic formulation of classical multi-sorted equational theories. As with signatures, free algebras for $\V$-enriched multi-sorted equational theories can be somewhat explicitly described under (just) our general assumptions on $\V$ (as we show in Theorems \ref{taut_thm} and \ref{explicit_free_fibre_prop_2}), but they can be described even more concretely and constructively when $\V$ is cartesian closed (as we show in Theorems \ref{cc_free_alg_thm} and \ref{free_alg_cc}). \S\ref{examples} establishes several examples of $\V$-enriched multi-sorted equational theories, such as: the single-sorted theory of \emph{$H$-monoids} (with $\V = \Top$); what we call \emph{homotopy weakenings} of classical multi-sorted equational theories (again with $\V = \Top$); multi-sorted theories of \emph{covariant $\V$-presheaves}; and a certain subclass of the \emph{relational single-sorted algebraic theories} of \cite{Monadsrelational} (which generalize the \emph{quantitative single-sorted algebraic theories} of \cite{Quant_alg_reasoning} and the \emph{ordered single-sorted algebraic theories} of \cite{finitarymonadsposets}). 

In \S\ref{smc_section}, under the assumption that $\V$ is symmetric monoidal closed, we detail the previously discussed relationship between $\V$-enriched multi-sorted equational theories and the \emph{(diagrammatic) presentations} of certain $\V$-enriched algebraic theories and monads studied in \cite{Pres, EP, Struct, Ax}. Finally, \S\ref{appendix} is an Appendix that thoroughly unpacks, in many prominent examples of topological categories over $\Set$, the specific constructions that we use to explicitly and concretely describe free algebras in our main results. We intend to pursue applications of the present work to the development of topologically enriched treatments of algebraic specification, computational effects, and algebraic databases.   

\section{Background}
\label{background}

We begin by recalling some background material on concrete and topological categories, which can be found (e.g.) in \cite{AHS}.

\begin{para}[\textbf{Concrete categories and their fibres}]
\label{concrete_para}
Let $\X$ be a (locally small) category. A \emph{concrete category over $\X$} is a category $\V$ equipped with a faithful functor $|-| : \V \to \X$ (which we shall often not mention explicitly). We shall usually not distinguish notationally between a morphism $f : V \to W$ of $\V$ and the underlying morphism $f = |f| : |V| \to |W|$ of $\X$. Given objects $V$ and $W$ of $\V$ and a morphism $f : |V| \to |W|$ of $\X$, we say that $f$ is \emph{$\V$-admissible} if it lifts (necessarily uniquely) to a morphism $f : V \to W$ of $\V$; in this case, we may also say that $f : |V| \to |W|$ \emph{is a $\V$-admissible morphism $f : V \to W$}, or that $f : |V| \to |W|$ \emph{lifts to a $\V$-morphism $f : V \to W$}.

Given an object $X$ of $\X$, the \emph{fibre} $\Fib(X) = \Fib_\V(X)$ is the preordered class consisting of all the objects $V$ of $\V$ such that $|V| = X$, with $V \leq V'$ if the identity morphism $1_X : |V| = X \to X = \left|V'\right|$ of $X$ in $\X$ is $\V$-admissible. A concrete category $\V$ over $\X$ is \emph{amnestic} if each fibre $\Fib(X)$ ($X \in \ob\X$) is a partially ordered class.
\end{para} 

\begin{para}[\textbf{Initial sources, final sinks, and topological categories}]
\label{top_para}
A \emph{source} in an arbitrary category $\V$ is a (possibly large) family of morphisms $(f_i : V \to V_i)_{i \in I}$ in $\V$ with the same domain, while a \emph{sink} in $\V$ is a source in $\V^\op$, i.e.~a family of morphisms $(g_i : V_i \to V)_{i \in I}$ in $\V$ with the same codomain. A source $(f_i : V \to V_i)_{i \in I}$ in a concrete category $\V$ over a category $\X$ is \emph{initial} (or \emph{$|-|$-initial}) if for each object $W$ of $\V$ and each morphism $f : |W| \to |V|$ of $\X$, the $\X$-morphism $f$ is $\V$-admissible iff the composite $\X$-morphisms $f_i \circ f : |W| \to |V_i|$ are $\V$-admissible for all $i \in I$. In particular, a morphism of $\V$ is \emph{initial} (or \emph{$|-|$-initial}) if the source consisting of just that morphism is initial. Dually, a sink $(g_i : V_i \to V)_{i \in I}$ in a concrete category $\V$ over $\X$ is \emph{final} if for each object $W$ of $\V$ and each morphism $g : |V| \to |W|$ of $\X$, the $\X$-morphism $g$ is $\V$-admissible iff the composite $\X$-morphisms $g \circ g_i : |V_i| \to |W|$ are $\V$-admissible for all $i \in I$. In particular, a morphism of $\V$ is \emph{final} if the sink consisting of just that morphism is final. A morphism of $\V$ that is both final and epimorphic is called a \emph{quotient morphism}. 

Given a concrete category $\V$ over $\X$, a \emph{structured source (in $\X$)} is a (possibly large) family of $\X$-morphisms $\left(f_i : X \to |V_i|\right)_{i \in I}$ with $X \in \ob\X$ and $V_i \in \ob\V$ for all $i \in I$, while a \emph{structured sink (in $\X$)} is a (possibly large) family of $\X$-morphisms $\left(g_i : |V_i| \to X\right)_{i \in I}$ with $X \in \ob\X$ and $V_i \in \ob\V$ for all $i \in I$. A concrete category $\V$ over $\X$ is \emph{topological (over $\X$)} if it is amnestic and every structured source $\left(f_i : X \to |V_i|\right)_{i \in I}$ has an initial lift, meaning that there is an object $V$ of $\V$ with $|V| = X$ such that each $f_i : |V| = X \to |V_i|$ ($i \in I$) is $\V$-admissible, and the resulting source $(f_i : V \to V_i)_{i \in I}$ in $\V$ is initial; in this situation, we sometimes say that the object $X$ of $\X$ has been \emph{equipped with the initial structure $V$} induced by the morphisms $f_i$ ($i \in I$). We also say that the faithful functor $|-| : \V \to \X$ is \emph{topological}. A topological category over $\X$ also satisfies the dual condition that every structured sink $\left(g_i : \left|V_i\right| \to X\right)_{i \in I}$ has a final lift $\left(g_i : V_i \to V\right)_{i \in I}$ (see \cite[21.9]{AHS}); in this situation, we sometimes say that the object $X$ of $\X$ has been \emph{equipped with the final structure} $V$ induced by the morphisms $f_i$ ($i \in I$). 

A topological functor $|-| : \V \to \X$ uniquely lifts all limits (via initiality) and all colimits (via finality) that exist in $\X$, and strictly preserves all limits and colimits (see \cite[Proposition 21.15]{AHS}). In particular, every topological category over $\X = \Set$ is complete and cocomplete. 
\end{para}

\begin{para}[\textbf{Complete lattice fibres of topological categories}]
\label{fibre_para}
Let $\V$ be a topological category over a category $\X$. Each fibre $\Fib(X)$ ($X \in \ob\X$) is then a complete (possibly large) lattice (see \cite[Proposition 21.11]{AHS}). Given a (possibly large) family $\left(V_i\right)_{i \in I}$ of elements of $\Fib(X)$ ($X \in \ob\X$), the supremum $\bigvee_{i \in I} V_i$ in $\Fib(X)$ is the codomain of the final lift of the structured sink of identity functions $\left(1_X : \left|V_i\right| = X \to X\right)_{i \in I}$, while the infimum $\bigwedge_{i \in I} V_i$ in $\Fib(X)$ is the domain of the initial lift of the structured source of identity functions $\left(1_X : X \to X = \left|V_i\right|\right)_{i \in I}$. Given a structured source $\left(f_i : X \to \left|V_i\right|\right)_{i \in I}$, if for each $i \in I$ we write $V_{i, f}$ for the domain of the initial lift of the singleton structured source $\left(f_i : X \to |V_i|\right)$, then the domain of the initial lift of the structured source $\left(f_i : X \to \left|V_i\right|\right)_{i \in I}$ is equal to the infimum $\bigwedge_{i \in I} V_{i, f}$ in $\Fib(X)$.

In particular (for $I = \varnothing$), the smallest (resp.~largest) element $\Disc X$ (resp.~$\Indisc X$) of $\Fib(X)$ is the \emph{discrete} (resp.~\emph{indiscrete}) object of $\V$ on $X$: every $\X$-morphism $f : \left|\Disc X\right| = X \to |V|$ ($V \in \ob\V$) is $\V$-admissible, while every $\X$-morphism $f : |V| \to X = \left|\Indisc X\right|$ ($V \in \ob\V$) is $\V$-admissible (see \cite[Proposition 21.11]{AHS}). The topological functor $|-| : \V \to \X$ then has a fully faithful left adjoint section $\Disc : \X \to \V$ and a fully faithful right adjoint section $\Indisc : \X \to \V$ (see \cite[Proposition 21.12]{AHS}).    
\end{para}

\begin{para}[\textbf{$\calS$-sorted topological categories}]
\label{prod_top_para}
Let $\V$ be a concrete category over a category $\X$, let $\calS$ be a set, and consider the product categories $\V^\calS = \prod_{S \in \calS} \V$ and $\X^\calS = \prod_{S \in \calS} \X$. We shall typically refer to an object (resp.~morphism) of $\V^\calS$ as an \emph{$\calS$-sorted object} (resp.~an \emph{$\calS$-sorted morphism}) of $\V$, and similarly for $\X^\calS$. In particular, when $\X = \Set$, we shall typically refer to an object (resp.~morphism) of $\X^\calS = \Set^\calS$ as an \emph{$\calS$-sorted set} (resp.~an \emph{$\calS$-sorted function}). 

The category $\V^\calS$ is concrete over $\X^\calS$ via the faithful functor $|-|^\calS : \V^\calS \to \X^\calS$ that sends an $\calS$-sorted object $V = \left(V_S\right)_{S \in \calS}$ of $\V$ to the $\calS$-sorted object $|V|^\calS := \left(\left|V_S\right|\right)_{S \in \calS}$ of $\X$, and sends an $\calS$-sorted morphism $f = \left(f_S\right)_{S \in \calS} : V \to W$ of $\V$ to the $\calS$-sorted morphism $|f|^\calS := \left(|f_S|\right)_{S \in \calS} : |V|^\calS \to |W|^\calS$ of $\X$. We shall often omit the superscript from $|-|^\calS$ when this is unlikely to cause confusion. Note that the concrete category $\V^\calS$ is amnestic if the concrete category $\V$ is amnestic.  Given an $\calS$-sorted object $X$ of $\X$, the fibre $\Fib(X) = \Fib_{\V^\calS}(X)$ is isomorphic (as a preordered class) to the product preordered class $\prod_{S \in \calS} \Fib\left(X_S\right)$ of the fibres $\Fib\left(X_S\right) = \Fib_\V\left(X_S\right)$ ($S \in \calS$). 

If $\V$ is topological over $\X$, then $\V^\calS$ is topological over $\X^\calS$, with initial lifts (resp.~final lifts) of structured sources (resp.~structured sinks) given ``pointwise''. Explicitly, let $\left(f^i : X \to \left|V^i\right|^\calS\right)_{i \in I}$ be a structured source in $\X^\calS$. Then for each $S \in \calS$, we obtain the structured source $\left(f^i_S : X_S \to \left|V^i_S\right|\right)_{i \in I}$ in $\X$. Because $\V$ is topological over $\X$, for each $i \in I$ the structured source $\left(f^i_S : X_S \to \left|V^i_S\right|\right)_{i \in I}$ has an initial lift $\left(f^i_S : V_S \to V^i_S\right)_{i \in I}$ in $\V$. We then define the $\calS$-sorted object $V := \left(V_S\right)_{S \in \calS}$ of $\V$, and the resulting source $\left(f^i : V \to V^i\right)_{i \in I}$ in $\V^\calS$ is initial. In particular, if $\V$ is topological over $\X = \Set$, then $\V^\calS$ is topological over $\X^\calS = \Set^\calS$.   
\end{para}

\begin{para}[\textbf{Suitable symmetric monoidal (closed) structures on topological categories}]
\label{sym_mon_para}
Beginning in \S\ref{first_section}, we shall be primarily concerned with the following data: a topological category $\V$ over $\Set$ such that $\V = (\V, \tensor, I)$ is also a symmetric monoidal category (not necessarily \emph{closed}) and the topological functor $|-| : \V \to \Set$ is strict symmetric monoidal with respect to the cartesian symmetric monoidal structure on $\Set$. Since the topological functor $|-| : \V \to \Set$ strictly preserves products \eqref{top_para}, it is automatically strict symmetric monoidal when $\V$ is (also) equipped with the cartesian symmetric monoidal structure. 

A topological functor $|-| : \V \to \Set$ is also strict symmetric monoidal with respect to the \emph{canonical (non-cartesian) symmetric monoidal closed structure} on the topological category $\V$, which we now recall from \cite[\S 2.2]{Sato}. Let $X$ and $Y$ be objects of $\V$. For $x \in |X|$ and $y \in |Y|$, we write $\Gamma_x : |Y| \to |X| \times |Y|$ for the function given by $y \mapsto (x, y)$ and $\Gamma_y : |X| \to |X| \times |Y|$ for the function given by $x \mapsto (x, y)$. Then the monoidal product $X \tensor Y$ in $\V$ is the codomain of the final lift of the structured sink $\left(\Gamma_x : |Y| \to |X| \times |Y|\right)_{x \in |X|} \bigcup \left(\Gamma_y : |X| \to |X| \times |Y|\right)_{y \in |Y|}$. For each $x \in |X|$, we also write $\ev_x : \V(X, Y) \to |Y|$ for the evaluation function given by $f \mapsto f(x)$, where $\V(X, Y)$ is the hom-set from $X$ to $Y$ in the ordinary category $\V$. The internal hom $[X, Y]$ in $\V$ is then the domain of the initial lift of the structured source $\left(\ev_x : \V(X, Y) \to |Y|\right)_{x \in |X|}$. Finally, the unit object $I$ is the discrete object of $\V$ on a singleton set. For example, when $\V = \Top$ (see Examples \ref{examples_ex} and \ref{top_ex} below), the tensor product $X \tensor Y$ of topological spaces $X$ and $Y$ equips the set $|X| \times |Y|$ with the \emph{topology of separate continuity}, while the internal hom $[X, Y]$ equips the set $\Top(X, Y)$ of continuous functions $X \to Y$ with the \emph{topology of pointwise convergence} (see, e.g., \cite[7.1.6]{Borceux2}).      
\end{para}

\begin{ex}[\textbf{Examples of topological categories over $\Set$}]
\label{examples_ex}
We now list some prominent examples of topological categories over $\Set$. We provide a more comprehensive treatment of these examples in the Appendix (\S\ref{appendix}). We point out which examples are cartesian closed, since our strongest results will apply to such topological categories over $\Set$ (see \S\ref{sig_alg_subsection_cc} and \S\ref{cong_section}).
\begin{itemize}[leftmargin=*]
\item The category $\Set$, equipped with the identity functor $1_\Set : \Set \to \Set$, is topological over $\Set$, and is cartesian closed.

\item Perhaps the canonical example of a topological category over $\Set$ is the category $\Top$ of topological spaces and continuous functions, equipped with the forgetful functor $|-| : \Top \to \Set$ that sends a topological space to its underlying set. It is well known that $\Top$ is not cartesian closed, but it is symmetric monoidal closed when equipped with the canonical symmetric monoidal closed structure of \ref{sym_mon_para}.

\item The following example comes from \cite{EscardoCCC}. Let $\calC$ be a fixed class of topological spaces, called the \emph{generating spaces}. A topological space $X$ is \emph{$\calC$-generated} if $X$ has the final topology induced by all continuous functions into $X$ from spaces in $\calC$. We write $\Top_\calC \hookrightarrow \Top$ for the full subcategory of $\Top$ consisting of the $\calC$-generated spaces, which is topological over $\Set$ when equipped with the restricted forgetful functor $|-| : \Top_\calC \to \Set$. When the class $\calC$ of generating spaces is \emph{productive} in the sense of \cite[Definition 3.5]{EscardoCCC}, the category $\Top_\calC$ is cartesian closed by \cite[Theorem 3.6]{EscardoCCC}. By \cite[Definition 3.3]{EscardoCCC}, examples of $\Top_\calC$ with $\calC$ productive include the categories of compactly generated spaces, core compactly generated spaces, locally compactly generated spaces, and sequentially generated spaces.

\item The category $\Meas$ of measurable spaces and measurable functions is topological over $\Set$ when equipped with the forgetful functor $|-| : \Meas \to \Set$ that sends a measurable space to its underlying set (see e.g.~\cite[\S 2.1]{Sato}). It is known that $\Meas$ is not cartesian closed, but it is symmetric monoidal closed when equipped with the canonical symmetric monoidal closed structure of \ref{sym_mon_para} (which is specifically unpacked for $\Meas$ in \cite[\S 2]{Sato}).

\item Given a \emph{relational Horn theory $\T$ without equality} (in the precise sense of \cite[Definition 3.5]{Extensivity}), the category $\T\Mod$ of $\T$-models and their morphisms is topological over $\Set$. $\T\Mod$ is not cartesian closed in general, but it is cartesian closed under suitable conditions on $\T$ (see \cite[Theorem 6.15]{Exp_relational}). The canonical symmetric monoidal closed structure on $\T\Mod$ \eqref{sym_mon_para} was previously considered in \cite[Corollary 3.13]{Monadsrelational} (see also \cite[3.12]{Exp_relational}).Three prominent examples of $\T\Mod$ for a relational Horn theory $\T$ without equality are the category $\mathsf{Rel}$ of sets equipped with a binary relation, the category $\Preord$ of preordered sets and monotone mappings, and the category $\PMet$ of (extended) pseudo-metric spaces and non-expansive mappings. For a more comprehensive list of examples, see \cite[Example 3.5]{Monadsrelational} or \cite[Example 3.7]{Extensivity}.

\item The following example originates from \cite{Dubucquasitopoi}: the category of \emph{quasispaces} (or \emph{concrete sheaves}) on a \emph{concrete site} is topological over $\Set$ (such a category is also known as a \emph{concrete quasitopos}). Categories of quasispaces on concrete sites are always cartesian closed (even locally cartesian closed). Prominent examples of categories of quasispaces on concrete sites (from \cite{Dubucquasitopoi}) include the categories of convergence spaces, subsequential spaces, bornological sets, pseudotopological spaces, and quasitopological spaces. Other examples (from \cite{Baezsmooth}) include the categories of diffeological spaces, Chen spaces, and (abstract) simplicial complexes, while \cite{Heunenprobability} provides the example of quasi-Borel spaces. Some further examples (such as the example of \emph{quantum sets}) are considered in \cite{Concrete_cats_recursion}, while the example of \emph{$C$-spaces} is studied in \cite{Escardo_Xu}.
\end{itemize}
\end{ex}

\section{Free algebras of enriched multi-sorted signatures}
\label{first_section}

For the remainder of the paper (unless otherwise stated), we fix the following data:
\begin{itemize}[leftmargin=*]
\item A topological category $\V$ over $\Set$ such that $\V = (\V, \tensor, I)$ is also a symmetric monoidal category and the given topological functor $|-| : \V \to \Set$ is strict symmetric monoidal (with respect to the cartesian symmetric monoidal structure on $\Set$). In particular, this implies that $|X \tensor Y| = |X| \times |Y|$ for all $X, Y \in \ob\V$, and that $|I|$ is a singleton.    
 
\item A set $\calS$, whose elements we call \emph{sorts}. We write $\calS^*$ for the set of all finite (possibly empty) sequences of elements of $\calS$.
\end{itemize}

\subsection{Enriched multi-sorted signatures}
\label{sig_def_subsection}

We begin by defining the notion of a $\V$-enriched multi-sorted signature (for related notions of signature, see Remark \ref{sig_rmk} below).  

\begin{defn_sub}
\label{signature}
A \textbf{$\V$-enriched $\calS$-sorted signature} is a set $\Sigma$ of \emph{operation symbols} equipped with an assignment to each operation symbol $\sigma \in \Sigma$ of \emph{input sorts} $(S_1, \ldots, S_n) \in \calS^*$, an \emph{output sort} $S \in \calS$, and a \emph{parameter (object)} $P = P_\sigma \in \ob\V$. We say that an operation symbol $\sigma \in \Sigma$ has \emph{type} $((S_1, \ldots, S_n), S, P)$ if $\sigma$ has input sorts $(S_1, \ldots, S_n)$, output sort $S$, and parameter $P$. An operation symbol of $\Sigma$ is \textbf{ordinary} if its parameter object is the unit object $I$ of the symmetric monoidal category $\V$. A $\V$-enriched $\calS$-sorted signature is \textbf{ordinary} if its operation symbols are all ordinary.

When $\calS$ is a singleton, we refer to a $\V$-enriched $\calS$-sorted signature as a \textbf{$\V$-enriched single-sorted signature}. Given a $\V$-enriched single-sorted signature $\Sigma$, we say that an operation symbol $\sigma \in \Sigma$ has \emph{arity} $n \geq 0$ if the length of its tuple of input sorts is $n$.

In the context $\V = \Set$, we say that a \textbf{classical} \textbf{$\calS$-sorted signature} is a ordinary $\Set$-enriched $\calS$-sorted signature.       
\end{defn_sub}

\begin{defn_sub}
\label{sig_algebra}
Let $\Sigma$ be a $\V$-enriched $\calS$-sorted signature. A \textbf{$\Sigma$-algebra} $A$ is an object $A = \left(A_S\right)_{S \in \calS}$ of $\V^\calS$ equipped with, for each $\sigma \in \Sigma$ of type $((S_1, \ldots, S_n), S, P)$, a $\V$-morphism \[ \sigma^A : P \tensor \left(A_{S_1} \times \ldots \times A_{S_n}\right) \to A_S. \] When $\sigma$ is ordinary (so that $P = I$), we simply write
\[ \sigma^A : A_{S_1} \times \ldots \times A_{S_n} \to A_S. \] 
Note that when $\V$ is symmetric monoidal \emph{closed} (with internal hom $[-, -]$), we may equivalently write
\[ \sigma^A : P \to \left[A_{S_1} \times \ldots \times A_{S_n}, A_S\right]. \]
We often say that the $\calS$-sorted object $A$ of $\V$ is the \emph{carrier (object)} of the $\Sigma$-algebra $A$. 

Given $\Sigma$-algebras $A$ and $B$, a \textbf{morphism of $\Sigma$-algebras} $f : A \to B$ is a morphism $f = \left(f_S\right)_{S \in \calS} : A = \left(A_S\right)_{S \in \calS} \to \left(B_S\right)_{S \in \calS} = B$ of $\V^\calS$ that makes the following diagram commute for each $\sigma \in \Sigma$ of type $((S_1, \ldots, S_n), S, P)$: 
\[\begin{tikzcd}
	{P \tensor \left(A_{S_1} \times \ldots \times A_{S_n}\right)} &&& {P \tensor \left(B_{S_1} \times \ldots \times B_{S_n}\right)} \\
	\\
	{A_S} &&& {B_S}.
	\arrow["{1_P \tensor \left(f_{S_1} \times \ldots \times f_{S_n}\right)}", from=1-1, to=1-4]
	\arrow["{\sigma^A}"', from=1-1, to=3-1]
	\arrow["{\sigma^B}", from=1-4, to=3-4]
	\arrow["{f_S}"', from=3-1, to=3-4]
\end{tikzcd}\]
We have a category $\Sigma\Alg$ of $\Sigma$-algebras and their morphisms, which is equipped with a forgetful functor $U^\Sigma : \Sigma\Alg \to \V^\calS$ that sends a $\Sigma$-algebra $A$ to the underlying $\calS$-sorted object $A$ of $\V$.  
\end{defn_sub}

\begin{rmk_sub}
\label{sig_rmk}
When $\V$ is symmetric monoidal closed, the $\V$-enriched $\calS$-sorted signatures of Definition \ref{signature} accord closely with the \emph{free-form signatures (for a subcategory of arities)} studied in \cite{EP, Ax}, as we show in Propositions \ref{equiv_presentations_prop1} and \ref{equiv_presentations_prop2}. Definition \ref{signature} is also an $\calS$-sorted generalization of \cite[Definition 3.1]{Battenfeld_comp}\footnote{Except that \cite[Definition 3.1]{Battenfeld_comp} requires the set of operation symbols to be \emph{countable}.} (wherein a $\V$-enriched single-sorted signature is referred to as \emph{a signature for a (finitary) parameterized algebraic theory for $\V$}). 
\end{rmk_sub}

\begin{rmk_sub}
\label{enriched_rmk}
When $\V$ is symmetric monoidal closed (and $|-| : \V \to \Set$ is represented by the unit object $I$ of $\V$), so that $\V$ and $\V^\calS$ can be regarded as $\V$-categories (for the latter, see \ref{Vcat_para} below), the category $\Sigma\Alg$ underlies a $\V$-category (also denoted $\Sigma\Alg$) and the functor $U^\Sigma : \Sigma\Alg \to \V$ underlies a $\V$-functor (also denoted $U^\Sigma$). In detail, for each $\Sbar = (S_1, \ldots, S_n) \in \calS^*$, we write $U^{\Sbar} : \V^\calS \to \V$ for the $\V$-functor given by $A \mapsto A_{S_1} \times \ldots \times A_{S_n}$ ($A \in \V^\calS$). In particular, for a single sort $S \in \calS$, we write $U^S := U^{(S)} : \V^\calS \to \V$, which is given by $A \mapsto A_S$ ($A \in \V^\calS$). For each $P \in \ob\V$ and each $\Sbar \in \calS^*$ we then have the $\V$-functor $P \tensor U^{\Sbar} : \V^\calS \to \V$ defined as the composite $\V^\calS \xrightarrow{U^{\Sbar}} \V \xrightarrow{P \tensor (-)} \V$, which is given by $A \mapsto P \tensor \left(A_{S_1} \times \ldots \times A_{S_n}\right)$ ($A \in \V^\calS$). Now each hom-object $\Sigma\Alg(A, B)$ ($A, B \in \Sigma\Alg$) is defined as the \emph{pairwise equalizer} \cite[2.1]{Commutants} of the following $\Sigma$-indexed family of parallel pairs in $\V$ (one for each $\sigma \in \Sigma$ of type $\left(\Sbar = (S_1, \ldots, S_n), S, P\right)$):
\[ \V^\calS(A, B) \xrightarrow{\left(P \tensor U^{\Sbar}\right)_{AB}} \V\left(P \tensor \left(A_{S_1} \times \ldots \times A_{S_n}\right), P \tensor \left(B_{S_1} \times \ldots \times B_{S_n}\right)\right) \] \[ \xrightarrow{\V\left(1, \sigma^B\right)} \V\left(P \tensor \left(A_{S_1} \times \ldots \times A_{S_n}\right), B_S\right), \]
\[ \V^\calS(A, B) \xrightarrow{U^S_{AB}} \V\left(A_S, B_S\right) \xrightarrow{\V\left(\sigma^A, 1\right)} \V\left(P \tensor \left(A_{S_1} \times \ldots \times A_{S_n}\right), B_S\right). \]
Composition in $\Sigma\Alg$ is then defined in the unique manner that allows the resulting subobjects $U^\Sigma_{AB} : \Sigma\Alg(A, B) \rightarrowtail \V^\calS(A, B)$ ($A, B \in \Sigma\Alg$) to serve as the structural morphisms of a faithful $\V$-functor $U^\Sigma : \Sigma\Alg \to \V^\calS$ that sends each $\Sigma$-algebra $A$ to its underlying $\calS$-sorted carrier object $A \in \ob\V^\calS$.

However, as our focus in this section is primarily on the explicit description of free $\Sigma$-algebras (i.e.~of the objects in the image of a left adjoint to $U^\Sigma$), we shall not pay much attention to the $\V$-enriched structure of $\Sigma\Alg$ until \S\ref{smc_section}.
\end{rmk_sub}

\noindent For the remainder of \S\ref{first_section}, we focus on establishing explicit descriptions of free algebras for $\V$-enriched $\calS$-sorted signatures, i.e.~of the left adjoint to $U^\Sigma : \Sigma\Alg \to \V^\calS$ for a $\V$-enriched $\calS$-sorted signature $\Sigma$ (in Theorems \ref{taut_sig_thm}, \ref{explicit_free_fibre_prop}, and \ref{free_sig_alg_cc}). For this purpose, we shall require the following material. 

\begin{defn_sub}
\label{ord_sig}
Let $\Sigma$ be a $\V$-enriched $\calS$-sorted signature. The \textbf{underlying classical $\calS$-sorted signature} is the classical $\calS$-sorted signature $|\Sigma|$ \eqref{signature} defined as follows. Given an operation symbol $\sigma \in \Sigma$ of type $((S_1, \ldots, S_n), S, P)$, we consider for each $p \in |P|$ an ordinary operation symbol $\sigma_p$ of type $((S_1, \ldots, S_n), S, 1)$. We then write $|\Sigma|$ for the classical $\calS$-sorted signature whose operation symbols are the ordinary operation symbols $\sigma_p$ for all $\sigma \in \Sigma$ and $p \in \left|P_\sigma\right|$ (with the types just described).         
\end{defn_sub}

\begin{para_sub}[\textbf{The underlying $|\Sigma|$-algebra $|A|$ of a $\Sigma$-algebra $A$}]
\label{ord_sig_para}
Let $\Sigma$ be a $\V$-enriched $\calS$-sorted signature with underlying classical $\calS$-sorted signature $|\Sigma|$ \eqref{ord_sig}. Since $|-| : \V \to \Set$ is faithful and strict monoidal, it readily follows that a $\Sigma$-algebra $A$ is equivalently given by an object $A = \left(A_S\right)_{S \in \calS}$ of $\V^\calS$ and a $|\Sigma|$-algebra $|A|$ with carrier $|A| = \left(\left|A_S\right|\right)_{S \in \calS}$ satisfying the condition that for each $\sigma \in \Sigma$ of type $((S_1, \ldots, S_n), S, P)$, the function
\begin{equation}
\label{ord_sig_eq} 
\sigma^A : \left|P \tensor \left(A_{S_1} \times \ldots \times A_{S_n}\right)\right| = |P| \times \left|A_{S_1}\right| \times \ldots \times \left|A_{S_n}\right| \to \left|A_S\right|
\end{equation}
\[ (p, a_1, \ldots, a_n) \mapsto \sigma_p^{|A|}(a_1, \ldots, a_n) \] 
lifts to a $\V$-admissible morphism 
\[ \sigma^A : P \tensor \left(A_{S_1} \times \ldots \times A_{S_n}\right) \to A_S. \] In particular, every $\Sigma$-algebra $A$ has an \textbf{underlying $|\Sigma|$-algebra} $|A|$ with carrier $|A| = \left(\left|A_S\right|\right)_{S \in \calS}$ and $\sigma_p^{|A|} := \sigma^A(p, -)$ for all $\sigma \in \Sigma$ and $p \in |P_\sigma|$. We readily obtain a faithful functor $|-| = |-|^\Sigma : \Sigma\Alg \to |\Sigma|\Alg$ that sends a $\Sigma$-algebra to its underlying $|\Sigma|$-algebra and makes the following square commute:
\begin{equation}
\label{sig_square}
\begin{tikzcd}
	\Sigma\Alg && {\V^\calS} \\
	\\
	{|\Sigma|\Alg} && {\Set^\calS}.
	\arrow["{U^\Sigma}", from=1-1, to=1-3]
	\arrow["{|-|^\calS}", from=1-3, to=3-3]
	\arrow["{|-|^\Sigma}"', from=1-1, to=3-1]
	\arrow["{U^{|\Sigma|}}"', from=3-1, to=3-3]
\end{tikzcd}
\end{equation}
So we may regard $\Sigma\Alg$ as a concrete category over $|\Sigma|\Alg$, and this concrete category is amnestic \eqref{concrete_para} because the concrete category $\V$ (over $\Set$) is amnestic.
\end{para_sub}

We now show that the commutative square \eqref{sig_square} satisfies the hypotheses of Wyler's \emph{taut lift theorem} (see \cite[Theorem 21.28]{AHS}), which will ultimately allow us to provide our first descriptions of free algebras of $\V$-enriched $\calS$-sorted signatures in Theorems \ref{taut_sig_thm} and \ref{explicit_free_fibre_prop}.

\begin{prop_sub}
\label{taut_lift_hyp}
Let $\Sigma$ be a $\V$-enriched $\calS$-sorted signature with underlying classical $\calS$-sorted signature $|\Sigma|$ \eqref{ord_sig}. The faithful functor $|-|^\Sigma : \Sigma\Alg \to |\Sigma|\Alg$ of \eqref{sig_square} is topological, and $U^\Sigma : \Sigma\Alg \to \V^\calS$ sends $|-|^\Sigma$-initial sources to $|-|^\calS$-initial sources \eqref{top_para}.
\end{prop_sub}

\begin{proof}
To show that $|-|^\Sigma : \Sigma\Alg \to |\Sigma|\Alg$ is topological, let $\left(f_i : A \to \left|B_i\right|^\Sigma\right)_{i \in I}$ be a structured source in $|\Sigma|\Alg$. Since $|-|^\calS : \V^\calS \to \Set^\calS$ is topological \eqref{prod_top_para}, we can equip the carrier object $A = \left(A_S\right)_{S \in \calS}$ of $\Set^\calS$ with the initial structure $A^* = \left(A_S^*\right)_{S \in \calS}$ in $\V^\calS$ induced by the $\calS$-sorted functions $f_i : A \to \left|B_i\right|^\calS$ ($i \in I$), so that the resulting source $\left(f_i : A^* \to B_i\right)_{i \in I}$ in $\V^\calS$ is $|-|^\calS$-initial. Now let $\sigma \in \Sigma$ have type $((S_1, \ldots, S_n), S, P)$, and let us show that the function \[ \sigma^{A^*} : \left|P \tensor \left(A_{S_1}^* \times \ldots \times A_{S_n}^*\right)\right| = |P| \times A_{S_1} \times \ldots \times A_{S_n} \to A_S = \left|A_S^*\right| \] 
\[ (p, a_1, \ldots, a_n) \mapsto \sigma_p^{A}(a_1, \ldots, a_n) \]
lifts to a $\V$-morphism \[ \sigma^{A^*} : P \tensor \left(A_{S_1}^* \times \ldots \times A_{S_n}^*\right) \to A_S^*. \] The source $\left(f_{i, S} : A_S^* \to B_{i, S}\right)_{i \in I}$ in $\V$ is $|-|$-initial (by \ref{prod_top_para}), so it suffices to show for each $i \in I$ that the function $f_{i, S} \circ \sigma^{A^*} : \left|P \tensor \left(A_{S_1}^* \times \ldots \times A_{S_n}^*\right)\right| \to \left|B_{i, S}\right|$ is $\V$-admissible. One immediately verifies that the latter function is equal to the function $\sigma^{B_i} \circ \left(1_P \times \left(f_{i, S_1} \times \ldots \times f_{i, S_n}\right)\right)$ (because $f_i : A \to |B_i|^\Sigma$ is a morphism of $|\Sigma|$-algebras), which is $\V$-admissible because $B_i$ is a $\Sigma$-algebra and each $f_{i, S_j} : A_{S_j}^* \to B_{i, S_j}$ ($1 \leq j \leq n$) is a $\V$-morphism. So (in view of \ref{ord_sig_para}) we have a $\Sigma$-algebra $A^*$ whose underlying $|\Sigma|$-algebra is $A$. 

To show that the source $\left(f_i : A^* \to B_i\right)_{i \in I}$ in $\Sigma\Alg$ is $|-|^\Sigma$-initial, let $B$ be a $\Sigma$-algebra and let $g : |B|^\Sigma \to A$ be a morphism of $|\Sigma|$-algebras with the property that each morphism of $|\Sigma|$-algebras $f_i \circ g : |B|^\Sigma \to |B_i|^\Sigma$ ($i \in I$) lifts to a morphism of $\Sigma$-algebras $f_i \circ g : B \to B_i$. We must show that the morphism of $|\Sigma|$-algebras $g : |B|^\Sigma \to A$ lifts to a morphism of $\Sigma$-algebras $g : B \to A^*$. It is equivalent to show for each $S \in \calS$ that the function $g_S : |B_S| \to A_S$ lifts to a $\V$-morphism $g_S : B_S \to A_S^*$, and this is true because the source $\left(f_{i, S} : A_S^* \to B_{i, S}\right)_{i \in I}$ in $\V$ is $|-|$-initial and (from the assumption on $g$) each function $f_{i, S} \circ g_S : |B_S| \to |B_{i, S}|$ ($i \in I$) lifts to a $\V$-morphism $f_{i, S} \circ g_S : B_S \to B_{i, S}$. This proves that $|-|^\Sigma : \Sigma\Alg \to |\Sigma|\Alg$ is topological. 

To that $U^\Sigma : \Sigma\Alg \to \V^\calS$ sends $|-|^\Sigma$-initial sources to $|-|^\calS$-initial sources, let $(f_i : A \to B_i)_{i \in I}$ be a $|-|^\Sigma$-initial source in $\Sigma\Alg$, and let us show that the source $(f_i : A \to B_i)_{i \in I}$ in $\V^\calS$ is $|-|^\calS$-initial. The source $(f_i : A \to B_i)_{i \in I}$ is a $|-|^\Sigma$-initial lift of the source $\left(f_i : |A|^\Sigma \to |B_i|^\Sigma\right)_{i \in I}$ in $|\Sigma|\Alg$, and so because such initial lifts are unique (because the concrete category $\Sigma\Alg$ over $|\Sigma|\Alg$ is amnestic), it follows that the source $(f_i : A \to B_i)_{i \in I}$ is equal to the source $(f_i : \left(|A|^\Sigma\right)^* \to B_i)_{i \in I}$ constructed in the first part of the proof. So the source $(f_i : A \to B_i)_{i \in I}$ is equal to the source $\left(f_i : \left(|A|^\calS\right)^* \to B_i\right)_{i \in I}$, and the latter source is $|-|^\calS$-initial.                     
\end{proof}

\begin{para_sub}[\textbf{Free algebras of classical $\calS$-sorted signatures}]
\label{free_sig_alg_Set}
Proposition \ref{taut_lift_hyp} shows that the commutative square \eqref{sig_square} satisfies two of the three central assumptions of Wyler's taut lift theorem \cite[Theorem 21.28]{AHS}, the third assumption being that the functor $U^{|\Sigma|} : |\Sigma|\Alg \to \Set$ has a left adjoint. Given an arbitrary classical $\calS$-sorted signature $\Sigma$ (not necessarily of the form $\left|\Sigma'\right|$ for a $\V$-enriched $\calS$-sorted signature $\Sigma'$), we now recall the well-known explicit description of the left adjoint $F^\Sigma : \Set^\calS \to \Sigma\Alg$ for $U^\Sigma : \Sigma\Alg \to \Set^\calS$ (see e.g.~\cite[Remark 14.18]{Algebraic_theories}).

Let $X = \left(X_S\right)_{S \in \calS}$ be an $\calS$-sorted set. We first recursively define, for each sort $S \in \calS$, the set $\Term_\Sigma(X)_S$ of \emph{ground $\Sigma$-terms of sort $S$ with constants from $X$}, by the following clauses:
\begin{enumerate}[leftmargin=*]
\item For each sort $S \in \calS$ we have $X_S \subseteq \Term_\Sigma(X)_S$. 

\item Given $\sigma \in \Sigma$ of type $((S_1, \ldots, S_n), S, 1)$ and $t_i \in \Term_\Sigma(X)_{S_i}$ for all $1 \leq i \leq n$, we have $\sigma(t_1, \ldots, t_n) \in \Term_\Sigma(X)_S$. 
\end{enumerate}
We then have the $\calS$-sorted set $\Term_\Sigma(X) := \left(\Term_\Sigma(X)_S\right)_{S \in \calS}$ of \emph{ground $\Sigma$-terms with constants from $X$}, which is the carrier of the free $\Sigma$-algebra $F^\Sigma X$ on the $\calS$-sorted set $X$. For each $\sigma \in \Sigma$ of type $((S_1, \ldots, S_n), S, 1)$, we define the function \[ \sigma^{F^\Sigma X} : \Term_\Sigma(X)_{S_1} \times \ldots \times \Term_\Sigma(X)_{S_n} \to \Term_\Sigma(X)_S \]
\[ (t_1, \ldots, t_n) \mapsto \sigma(t_1, \ldots, t_n). \] These data yield a $\Sigma$-algebra $F^\Sigma X$, which is equipped with a canonical $\calS$-sorted function $\eta = \eta_X^\Sigma : X \to U^\Sigma F^\Sigma X = \Term_\Sigma(X)$ given at each sort $S \in \calS$ by the inclusion function $\eta_S : X_S \hookrightarrow \Term_\Sigma(X)_S$, and $\eta$ exhibits $F^\Sigma X$ as the free $\Sigma$-algebra on $X$.             
\end{para_sub}

\subsection{Free algebras of enriched multi-sorted signatures: the general case}
\label{sig_alg_subsection_gen}

For each $\V$-enriched $\calS$-sorted signature $\Sigma$ and each $\calS$-sorted object $X$ of $\V$, we write $X\downarrow U^\Sigma$ for the family consisting of all pairs $\left(A, f : X \to U^\Sigma A\right)$ consisting of a $\Sigma$-algebra $A$ and a morphism $f : X \to U^\Sigma A = A$ of $\V^\calS$. Given such a pair $\left(A, f : X \to U^\Sigma A\right)$, we write \[ f^\sharp : F^{|\Sigma|} |X| \to |A|^\Sigma \] for the unique $|\Sigma|$-algebra morphism corresponding to the $\calS$-sorted function $f : |X| \to \left|U^\Sigma A\right| = U^{|\Sigma|}|A|^{\Sigma}$ under the adjunction $F^{|\Sigma|} \dashv U^{|\Sigma|} : |\Sigma|\Alg \to \Set^\calS$ of \ref{free_sig_alg_Set}. Note that by \ref{free_sig_alg_Set}, we may also regard $f^\sharp$ as (just) an $\calS$-sorted function $f^\sharp : \Term_{|\Sigma|}(|X|) \to |A|$.     

From Proposition \ref{taut_lift_hyp}, \ref{free_sig_alg_Set}, and Wyler's taut lift theorem \cite[Theorem 21.28]{AHS}, we now immediately obtain the following result, which was previously (only) established in the special case where $\Sigma$ is ordinary and single-sorted (see e.g.~\cite[Theorem 2.3]{Porst_existence}). 

\begin{theo_sub}
\label{taut_sig_thm}
Let $\Sigma$ be a $\V$-enriched $\calS$-sorted signature with underlying classical $\calS$-sorted signature $|\Sigma|$ \eqref{ord_sig}. Then $U^\Sigma : \Sigma\Alg \to \V^\calS$ has a left adjoint $F^\Sigma : \V^\calS \to \Sigma\Alg$, and the resulting adjunction $F^\Sigma \dashv U^\Sigma : \Sigma\Alg \to \V^\calS$ is a lifting of the adjunction\footnote{In the precise sense of \cite[Definition 21.26]{AHS}.} $F^{|\Sigma|} \dashv U^{|\Sigma|} : |\Sigma|\Alg \to \Set^\calS$. In particular, the following square strictly commutes:
\begin{equation}
\label{sig_adjoint_square}
\begin{tikzcd}
	\V^\calS && {\Sigma\Alg} \\
	\\
	{\Set^\calS} && {|\Sigma|\Alg}.
	\arrow["{F^\Sigma}", from=1-1, to=1-3]
	\arrow["{|-|^\Sigma}", from=1-3, to=3-3]
	\arrow["{|-|^\calS}"', from=1-1, to=3-1]
	\arrow["{F^{|\Sigma|}}"', from=3-1, to=3-3]
\end{tikzcd}
\end{equation}
For each $\calS$-sorted object $X$ of $\V$, the free $\Sigma$-algebra $F^\Sigma X$ on $X$ is the domain of the initial lift of the $|-|^\Sigma$-structured source
\begin{equation}\label{free_sig_eqn} 
\left(f^\sharp : F^{|\Sigma|}|X| \to |A|^\Sigma\right)_{(A, f) \in X \downarrow U^\Sigma}.
\end{equation}
In particular, the carrier object $U^\Sigma F^\Sigma X$ of $\V^\calS$ is the domain of the initial lift of the $|-|^\calS$-structured source
\begin{equation}\label{free_sig_eqn_2} 
\left(f^\sharp : \Term_{|\Sigma|}(|X|) \to |A|^\calS\right)_{(A, f) \in X \downarrow U^\Sigma}.
\end{equation} 
\end{theo_sub}

\begin{para_sub}
\label{free_sig_para}
In the setting of Theorem \ref{taut_sig_thm}, note that the free $\Sigma$-algebra $F^\Sigma X$ on an object $X$ of $\V^\calS$ lies in the fibre of $\Sigma\Alg$ over the free $|\Sigma|$-algebra $F^{|\Sigma|}(|X|)$ on the underlying object $|X| = |X|^\calS$ of $\Set^\calS$. Recalling the explicit description of $F^{|\Sigma|}(|X|)$ from \ref{free_sig_alg_Set}, this means that the carrier object $U^\Sigma F^\Sigma X$ lies in the fibre of $\V^\calS$ over the $\calS$-sorted set $\Term_{|\Sigma|}(|X|)$ of ground $|\Sigma|$-terms with constants from $|X|$. Moreover, in view of \ref{ord_sig_para}, the following hold:
\begin{enumerate}[leftmargin=*]
\item The $\calS$-sorted inclusion function 
\[ \eta = \eta_{|X|}^{|\Sigma|} : |X| \to \Term_{|\Sigma|}(|X|) = \left|U^\Sigma F^\Sigma X\right| \] of \ref{free_sig_alg_Set} is a $\V^\calS$-admissible morphism $\eta = \eta_X^\Sigma : X \to U^\Sigma F^\Sigma X$.  
\item For each $\sigma \in \Sigma$ of type $((S_1, \ldots, S_n), S, P)$, the function
\[ \widehat{\sigma} : |P| \times \Term_{|\Sigma|}(|X|)_{S_1} \times \ldots \times \Term_{|\Sigma|}(|X|)_{S_n} \to \Term_{|\Sigma|}(|X|)_S \]
\[ (p, t_1, \ldots, t_n) \mapsto \sigma_p(t_1, \ldots, t_n) \] 
which may be equivalently written as
\[ \widehat{\sigma} : \left|P \tensor \left(\left(U^\Sigma F^\Sigma X\right)_{S_1} \times \ldots \times \left(U^\Sigma F^\Sigma X\right)_{S_n} \right)\right| = |P| \times \left|\left(U^\Sigma F^\Sigma X\right)_{S_1} \right| \times \ldots \times \left|\left(U^\Sigma F^\Sigma X\right)_{S_n} \right| \] \[ \to \left|\left(U^\Sigma F^\Sigma X\right)_{S} \right|, \] is a $\V$-admissible morphism \[ \widehat{\sigma} : P \tensor \left(\left(U^\Sigma F^\Sigma X\right)_{S_1}  \times \ldots \times \left(U^\Sigma F^\Sigma X\right)_{S_n} \right) \to \left(U^\Sigma F^\Sigma X\right)_{S}, \] and $\sigma^{F^\Sigma X} = \widehat{\sigma}$. 
\end{enumerate}
So $F^\Sigma X$ is determined by its carrier object $U^\Sigma F^\Sigma X$. While Theorem \ref{taut_sig_thm} provides a description of this carrier object, we shall provide more explicit and concrete descriptions of this carrier object in Theorem \ref{explicit_free_fibre_prop} and (especially) Theorem \ref{free_sig_alg_cc}. 
\end{para_sub}

\begin{defn_sub}
\label{Sigma_compatible}
Let $\Sigma$ be a $\V$-enriched $\calS$-sorted signature with underlying classical $\calS$-sorted signature $|\Sigma|$ \eqref{ord_sig}, and let $X = \left(X_S\right)_{S \in \calS}$ be an $\calS$-sorted object of $\V$. Given an $\calS$-sorted object $B$ of $\V$ that lies in the fibre over the $\calS$-sorted set $\Term_{|\Sigma|}(|X|)$ of ground $|\Sigma|$-terms with constants from $|X|$ \eqref{free_sig_alg_Set}, we say that $B$ is \textbf{$\Sigma$-compatible with $X$} if the following conditions are satisfied:
\begin{enumerate}[leftmargin=*]
\item The $\calS$-sorted inclusion function 
\[ \eta = \eta_{|X|}^{|\Sigma|} : |X| \to \Term_{|\Sigma|}(|X|) = |B| \] is a $\V^\calS$-admissible morphism $\eta : X \to B$. In other words, for each sort $S \in \calS$, the inclusion function $\eta_S : \left|X_S\right| \hookrightarrow \Term_{|\Sigma|}(|X|)_S = \left|B_S\right|$ is a $\V$-admissible morphism $\eta_S : X_S \to B_S$.  
\item For each $\sigma \in \Sigma$ of type $((S_1, \ldots, S_n), S, P)$, the function
\[ \widehat{\sigma} : |P| \times \Term_{|\Sigma|}(|X|)_{S_1} \times \ldots \times \Term_{|\Sigma|}(|X|)_{S_n} \to \Term_{|\Sigma|}(|X|)_S \]
\[ (p, t_1, \ldots, t_n) \mapsto \sigma_p(t_1, \ldots, t_n), \] 
which may be equivalently written as
\[ \widehat{\sigma} : \left|P \tensor \left(B_{S_1} \times \ldots \times B_{S_n}\right)\right| = |P| \times \left|B_{S_1}\right| \times \ldots \times \left|B_{S_n}\right| \to \left|B_{S}\right|, \] is a $\V$-admissible morphism $\widehat{\sigma} : P \tensor \left(B_{S_1} \times \ldots \times B_{S_n}\right) \to B_S$. 
\end{enumerate} 
If $B$ is $\Sigma$-compatible with $X$, then $B$ can be equipped with the structure of a $\Sigma$-algebra, where $\sigma^B = \widehat{\sigma}$ for each $\sigma \in \Sigma$. 
\end{defn_sub}

\begin{lem_sub}
\label{Sigma_compatible_lem}
Let $\Sigma$ be a $\V$-enriched $\calS$-sorted signature with underlying classical $\calS$-sorted signature $|\Sigma|$ \eqref{ord_sig}, and let $X = \left(X_S\right)_{S \in \calS}$ be an $\calS$-sorted object of $\V$. Given a pair $\left(A, f : X \to U^\Sigma A = A\right)$ in $X \downarrow U^\Sigma$, we write $A_f = \left(A_{f, S}\right)_{S \in \calS}$ for the domain of the initial lift of the $|-|^\calS$-structured morphism $f^\sharp : \Term_{|\Sigma|}(|X|) \to |A|^\calS$, so that $A_f$ is an element of the fibre over $\Term_{|\Sigma|}(|X|)$ and $f^\sharp : A_f \to A$ is an initial morphism of $\V^\calS$. Then $A_f$ is $\Sigma$-compatible with $X$.
\end{lem_sub}

\begin{proof}
The $\calS$-sorted inclusion function $\eta = \eta_{|X|}^{|\Sigma|} : |X| \to \Term_{|\Sigma|}(|X|) = \left|A_f\right|$ is a $\V^\calS$-admissible morphism $\eta : X \to A_f$ because $f^\sharp : A_f \to A$ is $|-|^\calS$-initial and $f^\sharp \circ \eta = f : |X|^\calS \to |A|^\calS$, which is $\V^\calS$-admissible. Given $\sigma \in \Sigma$ of type $((S_1, \ldots, S_n), S, P)$, we must show that the function
\[ \widehat{\sigma} : \left|P \tensor \left(A_{f, S_1} \times \ldots \times A_{f, S_n}\right)\right| = |P| \times \left|A_{f, S_1}\right| \times \ldots \times \left|A_{f, S_n}\right| \to \left|A_{f, S}\right| \]
\[ (p, t_1, \ldots, t_n) \mapsto \sigma_p(t_1, \ldots, t_n) \] is a $\V$-admissible morphism $\widehat{\sigma} : P \tensor \left(A_{f, S_1} \times \ldots \times A_{f, S_n}\right) \to A_{f, S}$. The $\V$-morphism $f^\sharp_S : A_{f, S} \to A_S$ is $|-|$-initial (by \ref{prod_top_para}) and $f^\sharp_S \circ \widehat{\sigma} = f^\sharp_S \circ \sigma^{F^\Sigma X} = \sigma^A \circ \left(1_P \tensor \left(f^\sharp_{S_1} \times \ldots \times f^\sharp_{S_n}\right)\right)$ by \ref{free_sig_para} and the fact that $f^\sharp : F^\Sigma X \to A$ is a morphism of $\Sigma$-algebras (since it is a morphism $f^\sharp : \left|F^\Sigma X\right|^\Sigma = F^{|\Sigma|}|X| \to |A|^\Sigma$ of the underlying $|\Sigma|$-algebras), whence the desired result follows because the rightmost composite is $\V$-admissible.  
\end{proof}

We now provide the following slightly more explicit description (compared to that given in Theorem \ref{taut_sig_thm}) of the left adjoint $F^\Sigma : \V^\calS \to \Sigma\Alg$ for a $\V$-enriched $\calS$-sorted signature $\Sigma$. The special case where $\Sigma$ is ordinary and single-sorted appears (for example) on \cite[Page 439]{Porst_cc}, while the special case where $\Sigma$ is single-sorted (but not necessarily ordinary) and $\V = \Top$ (see Examples \ref{examples_ex} and \ref{top_ex})  appears in \cite[Lemma 4.4]{Battenfeld_comp}.

\begin{theo_sub}
\label{explicit_free_fibre_prop}
Let $\Sigma$ be a $\V$-enriched $\calS$-sorted signature with underlying classical $\calS$-sorted signature $|\Sigma|$ \eqref{ord_sig}, and let $X$ be an $\calS$-sorted object of $\V$. Then the $\calS$-sorted carrier object $U^\Sigma F^\Sigma X \in \ob\V^\calS$ of the free $\Sigma$-algebra $F^\Sigma X$ on $X$ is the infimum in the fibre over the $\calS$-sorted set $U^{|\Sigma|} F^{|\Sigma|} |X| = \Term_{|\Sigma|}(|X|)$ \eqref{free_sig_alg_Set} of all the elements that are $\Sigma$-compatible with $X$ \eqref{Sigma_compatible}:
\[ U^\Sigma F^\Sigma X = \bigwedge\left\{B \in \Fib\left(\Term_{|\Sigma|}(|X|)\right) \mid B \text{ is } \Sigma\text{-compatible with } X\right\}. \]   
\end{theo_sub}

\begin{proof}
We know from Theorem \ref{taut_sig_thm} that the carrier object $U^\Sigma F^\Sigma X$ of $\V^\calS$ is the domain of the initial lift of the $|-|^\calS$-structured source
\[ \left(f^\sharp : \Term_{|\Sigma|}(|X|) \to |A|^\calS\right)_{(A, f) \in X \downarrow U^\Sigma}. \] Given a pair $\left(A, f : X \to U^\Sigma A = A\right)$ in $X \downarrow U^\Sigma$, recall from \ref{Sigma_compatible_lem} that we write $A_f$ for the domain of the initial lift of the $|-|^\calS$-structured morphism $f^\sharp : \Term_{|\Sigma|}(|X|) \to |A|^\calS$, so that $A_f$ is an element of the fibre over $\Term_{|\Sigma|}(|X|)$. Then (by \ref{fibre_para}) the carrier object $U^\Sigma F^\Sigma X$ of $\V^\calS$ is equal to the infimum $\bigwedge_{(A, f) \in X\downarrow U^\Sigma} A_f$ in the fibre over $\Term_{|\Sigma|}(|X|)$, so we are reduced to showing that
\[ \bigwedge_{(A, f) \in X\downarrow U^\Sigma} A_f = \bigwedge\left\{B \in \Fib\left(\Term_{|\Sigma|}(|X|)\right) \mid B \text{ is } \Sigma\text{-compatible with } X\right\} \] in $\Fib\left(\Term_{|\Sigma|}(|X|)\right)$. Given $(A, f)$ in $X \downarrow U^\Sigma$, we know from Lemma \ref{Sigma_compatible_lem} that $A_f$ is $\Sigma$-compatible with $X$, which yields the $\geq$ inequality. To show the $\leq$ inequality, it suffices to show for each $B \in \Fib\left(\Term_{|\Sigma|}(|X|)\right)$ that is $\Sigma$-compatible with $X$ that there is some pair $\left(C, f : X \to U^\Sigma C = C\right)$ in $X \downarrow U^\Sigma$ such that $C_f \leq B$ in $\Fib\left(\Term_{|\Sigma|}(|X|)\right)$. Since $B$ is $\Sigma$-compatible with $X$, it follows from \ref{Sigma_compatible} that $B$ can be equipped with the structure of a $\Sigma$-algebra, where $\sigma^B = \widehat{\sigma}$ for each $\sigma \in \Sigma$, and $\eta : X \to B = U^\Sigma B$ is $\V^\calS$-admissible. Then we have $(B, \eta : X \to B)$ in $X \downarrow U^\Sigma$ with $B_\eta \leq B$ in $\Fib\left(\Term_{|\Sigma|}(|X|)\right)$, because $B_\eta$ is the domain of the initial lift of $\eta^\sharp = 1 : \Term_{|\Sigma|}(|X|) \to \Term_{|\Sigma|}(|X|) = |B|^\calS$.   
\end{proof}

\noindent The following corollary is the single-sorted version of Theorem \ref{explicit_free_fibre_prop}.

\begin{cor_sub}
\label{explicit_free_fibre_prop_SS}
Let $\Sigma$ be a $\V$-enriched single-sorted signature with underlying classical single-sorted signature $|\Sigma|$ \eqref{ord_sig}, and let $X$ be an object of $\V$. Then the carrier object $U^\Sigma F^\Sigma X \in \ob\V$ of the free $\Sigma$-algebra $F^\Sigma X$ on $X$ is the infimum in the fibre over the set $U^{|\Sigma|} F^{|\Sigma|} |X| = \Term_{|\Sigma|}(|X|)$ of all the elements that are $\Sigma$-compatible with $X$ \eqref{Sigma_compatible}, i.e.~all the elements $B \in \Fib\left(\Term_{|\Sigma|}(|X|)\right)$ that satisfy the following conditions:
\begin{enumerate}[leftmargin=*]
\item The inclusion function 
\[ \eta = \eta_{|X|}^{|\Sigma|} : |X| \to \Term_{|\Sigma|}(|X|) = |B| \] is a $\V$-admissible morphism $\eta : X \to B$. 
\item For each $\sigma \in \Sigma$ of arity $n \geq 0$ \eqref{signature} and parameter object $P$, the function
\[ \widehat{\sigma} : |P| \times \Term_{|\Sigma|}(|X|)^n \to \Term_{|\Sigma|}(|X|) \]
\[ (p, t_1, \ldots, t_n) \mapsto \sigma_p(t_1, \ldots, t_n), \]
which may be equivalently written as
\[ \widehat{\sigma} : \left|P \tensor B^n\right| = |P| \times |B|^n \to \left|B\right|, \] is a $\V$-admissible morphism $\widehat{\sigma} : P \tensor B^n \to B$. \qed 
\end{enumerate}  
\end{cor_sub}

\subsection{Free algebras of enriched multi-sorted signatures: the case where $\V$ is cartesian closed}
\label{sig_alg_subsection_cc}

Given a $\V$-enriched $\calS$-sorted signature $\Sigma$ and an $\calS$-sorted object $X$ of $\V$, the descriptions of the free $\Sigma$-algebra $F^\Sigma X$ on $X$ provided by Theorems \ref{taut_sig_thm} and \ref{explicit_free_fibre_prop} are not as concrete or intrinsic as one might ideally want, since they ultimately refer not only to $X$ and $\Sigma$, but also (in some way) to the entire category $\Sigma\Alg$. Under the assumption that $\V$ is equipped with the \emph{cartesian} symmetric monoidal structure and is moreover cartesian closed, we shall provide a more concrete and intrinsic \emph{(countably) inductive} description of $F^\Sigma X$ in Theorem \ref{free_sig_alg_cc} below. (But we do not make the blanket assumption within this subsection that $\V$ is cartesian closed.)

\begin{para_sub}[\textbf{Final epi-sinks, coproducts, and quotient morphisms}]
\label{epi_coprod_para}
Let $\left(g_i : V_i \to V\right)_{i \in I}$ be a small sink in $\V$. We say that $\left(g_i : V_i \to V\right)_{i \in I}$ is an \emph{epi-sink} if the underlying sink $\left(g_i : |V_i| \to |V|\right)_{i \in I}$ in $\Set$ is jointly surjective, and we say that $\left(g_i : V_i \to V\right)_{i \in I}$ is a \emph{final epi-sink} if it is both final and an epi-sink. Since $|-| : \V \to \Set$ strictly preserves small coproducts \eqref{top_para}, one readily sees that $\left(g_i : V_i \to V\right)_{i \in I}$ is a final epi-sink iff the induced morphism $\left[g_i\right]_i : \coprod_i V_i \to V$ from the small coproduct $\coprod_i V_i$ is a final epimorphism, i.e.~a quotient morphism (noting that the epimorphisms of $\V$ are precisely the surjective morphisms). Recall from \cite[Proposition 21.13]{AHS} that the quotient morphisms of $\V$ are precisely the regular epimorphisms of $\V$.
\end{para_sub}

\begin{lem_sub}
\label{cc_pres_joins}
Let $P$ be an object of $\V$, and suppose that the product endofunctor $P \times (-) : \V \to \V$ preserves small colimits. Then for each set $X$, the monotone function 
\[ P \times (-) : \Fib(X) \to \Fib(|P| \times X) \] 
\[ A \mapsto P \times A \]
preserves small suprema.  
\end{lem_sub}

\begin{proof}
Let $(A_i)_{i \in I}$ be a small family of objects in $\Fib(X)$. Recall from \ref{fibre_para} that the supremum $\bigvee_i A_i$ in $\Fib(X)$ is the codomain of the final lift of the $|-|$-structured sink of identity functions $\left(\id_i : |A_i| = X \to X\right)_{i \in I}$. So $\left(\id_i : A_i \to \bigvee_i A_i\right)_{i \in I}$ is a final epi-sink in $\V$, and hence (by \ref{epi_coprod_para}) the induced morphism $\left[\id_i\right]_i : \coprod_i A_i \to \bigvee_i A_i$ is a quotient morphism. The endofunctor $P \times (-) : \V \to \V$ preserves regular epimorphisms and thus quotient morphisms, so that $P \times \left[\id_i\right]_i : P \times \coprod_i A_i \to P \times \bigvee_i A_i$ is a quotient morphism. The endofunctor $P \times (-) : \V \to \V$ also preserves small coproducts, so that the morphism $\left[1_P \times \id_i\right]_i : \coprod_i P \times A_i \to P \times \bigvee_i A_i$ is a quotient morphism, i.e.~a final epimorphism. This implies (by \ref{epi_coprod_para}) that $\left(1_P \times \id_i : P \times A_i \to P \times \bigvee_i A_i\right)_{i \in I}$ is a final epi-sink, which readily entails (by \ref{fibre_para}) that $P \times \bigvee_i A_i = \bigvee_i P \times A_i$ in $\Fib(|P| \times X)$, as desired.   
\end{proof}

\begin{lem_sub}
\label{join_ineq_lem}
Suppose that $\V$ is cartesian closed. Let $P$ be an object of $\V$, let $n \geq 1$, and let $X_1, \ldots, X_n$ be sets. For each $1 \leq i \leq n$, let $\left(A^m_i\right)_{m \geq 0} \in \Fib(X_i)$ be a countable increasing sequence (i.e.~$A_i^m \leq A_i^{m+1}$ for all $m \geq 0$). Then 
\begin{equation}
\label{join_eq}
P \times \left(\bigvee_{m \geq 0} A_1^m\right) \times \ldots \times \left(\bigvee_{m \geq 0} A^m_n\right) = \bigvee_{m \geq 0} \left(P \times A_1^m \times \ldots \times A^m_n\right)
\end{equation}
in $\Fib(|P| \times X_1 \times \ldots \times X_n)$. 
\end{lem_sub}

\begin{proof}
The cartesian closure of $\V$ means that for each object $X$ of $\V$, the product endofunctor $X \times (-) : \V \to \V$ is a left adjoint and thus preserves small colimits. We prove the non-trivial inequality $\leq$ in \eqref{join_eq} by induction on $n \geq 1$. For $n = 1$, this (in)equality holds because $P \times (-) : \V \to \V$ preserves small suprema by Lemma \ref{cc_pres_joins}. Now suppose that the equality \eqref{join_eq} holds for $n \geq 1$, and let us show that  
\[ P \times \left(\bigvee_{m \geq 0} A_1^m\right) \times \ldots \times \left(\bigvee_{m \geq 0} A^m_n\right) \times \left(\bigvee_{m \geq 0} A^m_{n+1}\right) \leq \bigvee_{m \geq 0} \left(P \times A_1^m \times \ldots \times A^m_n \times A^m_{n+1}\right). \] Given the induction hypothesis, it is equivalent to show that
\[ \bigvee_{m \geq 0} \left(P \times A_1^m \times \ldots \times A^m_n\right) \times \left(\bigvee_{m \geq 0} A^m_{n+1}\right) \leq \bigvee_{m \geq 0} \left(P \times A_1^m \times \ldots \times A^m_n \times A^m_{n+1}\right). \] By Lemma \ref{cc_pres_joins} applied to $\bigvee_{m \geq 0} \left(P \times A_1^m \times \ldots \times A^m_n\right)$, we have
\[ \bigvee_{m \geq 0} \left(P \times A_1^m \times \ldots \times A^m_n\right) \times \left(\bigvee_{m \geq 0} A^m_{n+1}\right) = \bigvee_{m \geq 0} \left(P \times A_1^m \times \ldots \times A^m_n \times \left(\bigvee_{j \geq 0} A^j_{n+1}\right)\right). \] And then by Lemma \ref{cc_pres_joins} applied to $P \times A_1^m \times \ldots \times A^m_n$ ($m \geq 0$), we deduce that
\[ \bigvee_{m \geq 0} \left(P \times A_1^m \times \ldots \times A^m_n\right) \times \left(\bigvee_{m \geq 0} A^m_{n+1}\right) = \bigvee_{m \geq 0}\left(\bigvee_{j \geq 0} \left(P \times A_1^m \times \ldots \times A_n^m \times A_{n+1}^j\right)\right). \] So it is finally equivalent to show that
\[ \bigvee_{m \geq 0}\left(\bigvee_{j \geq 0} \left(P \times A_1^m \times \ldots \times A_n^m \times A_{n+1}^j\right)\right) \leq \bigvee_{m \geq 0} \left(P \times A_1^m \times \ldots \times A^m_n \times A^m_{n+1}\right), \] i.e.~that for all $m, j \geq 0$ we have 
\begin{equation}
\label{join_eq_2}
P \times A_1^m \times \ldots \times A_n^m \times A_{n+1}^j \leq \bigvee_{m \geq 0} \left(P \times A_1^m \times \ldots \times A^m_n \times A^m_{n+1}\right).
\end{equation} 
Taking $p := \max\{j, m\}$, we have $A_i^m \leq A_i^p$ for each $1 \leq i \leq n$ and $A_{n+1}^j \leq A_{n+1}^p$, so that
\[ P \times A_1^m \times \ldots \times A_n^m \times A_{n+1}^j \leq P \times A_1^p \times \ldots \times A^p_n \times A^p_{n+1}, \] which yields the desired inequality \eqref{join_eq_2}. 
\end{proof}

\noindent We now have the following theorem. In the Appendix (\S\ref{appendix}), we provide explicit descriptions of the main ingredients used in this result (suprema and final structures) for all of the examples of (cartesian closed) topological categories over $\Set$ that we considered in Example \ref{examples_ex}.

\begin{theo_sub}
\label{free_sig_alg_cc}
Suppose that $(\V, \tensor, I) = (\V, \times, 1)$ and that $\V$ is cartesian closed. Let $\Sigma$ be a $\V$-enriched $\calS$-sorted signature  with underlying classical $\calS$-sorted signature $|\Sigma|$ \eqref{ord_sig}, and let $X$ be an $\calS$-sorted object of $\V$. Then the $\calS$-sorted carrier object $U^\Sigma F^\Sigma X \in \ob\V^\calS$ of the free $\Sigma$-algebra $F^\Sigma X$ on $X$ is the supremum \[ U^\Sigma F^\Sigma X = \bigvee_n \Omega^n(X) \] in the fibre over the $\calS$-sorted set $\Term_{|\Sigma|}(|X|)$ of the following $\calS$-sorted objects $\Omega^n(X)$ of $\V$ ($n \geq 0$):
\begin{enumerate}[leftmargin=*]
\item $\Omega^0(X)$ is the final $\V^\calS$-structure on $\Term_{|\Sigma|}(|X|)$ induced by the $\calS$-sorted inclusion function $\eta = \eta_{|X|}^{|\Sigma|} : |X| \to \Term_{|\Sigma|}(|X|)$. 

\item Given $\Omega^n(X) \in \Fib\left(\Term_{|\Sigma|}(|X|)\right)$ ($n \geq 0$), we first define, for each sort $S \in \calS$, an element $\Omega^n(X, S) \in \Fib\left(\Term_{|\Sigma|}(|X|)_S\right)$ as the final $\V$-structure on the set $\Term_{|\Sigma|}(|X|)_S$ induced by the functions
\[ \widehat{\sigma} : \left|P \times \Omega^n(X)_{S_1} \times \ldots \times \Omega^n(X)_{S_m}\right| = |P| \times \left|\Omega^n(X)_{S_1}\right| \times \ldots \times \left|\Omega^n(X)_{S_m}\right| \to \Term_{|\Sigma|}(|X|)_S \]
\[ (p, t_1, \ldots, t_m) \mapsto \sigma_p(t_1, \ldots, t_m) \]
 for all $\sigma \in \Sigma$ of type $((S_1, \ldots, S_m), S, P)$. We thus obtain an $\calS$-sorted object $\left(\Omega^n(X, S)\right)_{S \in \calS}$ of $\V$ in the fibre over the $\calS$-sorted set $\Term_{|\Sigma|}(|X|)$, and we then define \[ \Omega^{n+1}(X) := \Omega^n(X) \vee \left(\Omega^n(X, S)\right)_{S \in \calS} \] in the fibre over the $\calS$-sorted set $\Term_{|\Sigma|}(|X|)$.   
\end{enumerate}     
\end{theo_sub}

\begin{proof}
By Theorem \ref{explicit_free_fibre_prop}, it is equivalent to show that 
\begin{equation}
\label{explicit_eq}
\bigvee_n \Omega^n(X) = \bigwedge\left\{B \in \Fib\left(\Term_{|\Sigma|}(|X|)\right) \mid B \text{ is } \Sigma\text{-compatible with } X\right\}
\end{equation} 
in $\Fib\left(\Term_{|\Sigma|}(|X|)\right)$. To show the $\leq$ inequality in \eqref{explicit_eq}, let $B = \left(B_S\right)_{S \in \calS} \in \Fib\left(\Term_{|\Sigma|}(|X|)\right)$ be $\Sigma$-compatible with $X$, and let us show for each $n \geq 0$ that $\Omega^n(X) \leq B$. We have $\Omega^0(X) \leq B$ immediately from the definition of $\Omega^0(X)$ and the fact that $B$ is $\Sigma$-compatible with $X$. Now supposing for $n \geq 0$ that $\Omega^n(X) \leq B$, let us show that $\Omega^{n+1}(X) \leq B$, for which it suffices to show that $\left(\Omega^n(X, S)\right)_{S \in \calS} \leq B$. For each $S \in \calS$, we must show (in view of \ref{prod_top_para}) that $\Omega^n(X, S) \leq B_S$ in the fibre over the set $\Term_{|\Sigma|}(|X|)_S$. To prove this, it is equivalent (by the definition of $\Omega^n(X, S)$) to show for each $\sigma \in \Sigma$ of type $((S_1, \ldots, S_m), S, P)$ that the function
\[ \widehat{\sigma} : \left|P \times \Omega^n(X)_{S_1} \times \ldots \times \Omega^n(X)_{S_m}\right| = |P| \times \left|\Omega^n(X)_{S_1}\right| \times \ldots \times \left|\Omega^n(X)_{S_m}\right| \to \Term_{|\Sigma|}(|X|)_S = \left|B_S\right| \]
\[ (p, t_1, \ldots, t_m) \mapsto \sigma_p(t_1, \ldots, t_m) \]
 is a $\V$-admissible morphism
\[ \widehat{\sigma} : P \times \Omega^n(X)_{S_1} \times \ldots \times \Omega^n(X)_{S_m} \to B_S. \]
But this is true because $\Omega^n(X)_{S_i} \leq B_{S_i}$ in the fibre over the set $\Term_{|\Sigma|}(|X|)_{S_i}$ for each $1 \leq i \leq m$ (by the induction hypothesis) and because
the function
\[ \widehat{\sigma} : \left|P \times B_{S_1} \times \ldots \times B_{S_m}\right| = |P| \times \left|B_{S_1}\right| \times \ldots \times \left|B_{S_m}\right| \to \Term_{|\Sigma|}(|X|)_S = \left|B_S\right| \]
is a $\V$-admissible morphism
\[ \widehat{\sigma} : P \times B_{S_1} \times \ldots \times B_{S_m} \to B_S \] (since $B$ is $\Sigma$-compatible with $X$). This establishes the $\leq$ inequality in \eqref{explicit_eq}.        

To show the converse inequality $\geq$ in \eqref{explicit_eq}, it suffices to show that $\bigvee_n \Omega^n(X)$ is $\Sigma$-compatible with $X$. That $\bigvee_n \Omega^n(X)$ satisfies condition (1) of Definition \ref{Sigma_compatible} is because $\Omega^0(X)$ satisfies condition (1) and $\Omega^0(X) \leq \bigvee_n \Omega^n(X)$. To show that $\bigvee_n \Omega^n(X)$ satisfies condition (2) of Definition \ref{Sigma_compatible}, let $\sigma \in \Sigma$ have type $((S_1, \ldots, S_m), S, P)$ with $m \geq 1$ (we consider the case $m = 0$ after), and let us show that the function
\[ \widehat{\sigma} : |P| \times \left|\left(\bigvee_n \Omega^n(X)\right)_{S_1}\right| \times \ldots \times \left|\left(\bigvee_n \Omega^n(X)\right)_{S_m}\right| \to \left|\left(\bigvee_n \Omega^n(X)\right)_S\right| \]
\[ (p, t_1, \ldots, t_m) \mapsto \sigma_p(t_1, \ldots, t_m) \] is a $\V$-admissible morphism
\[ \widehat{\sigma} : P \times \left(\bigvee_n \Omega^n(X)\right)_{S_1} \times \ldots \times \left(\bigvee_n \Omega^n(X)\right)_{S_m} \to \left(\bigvee_n \Omega^n(X)\right)_S. \]
For each $n \geq 0$, the function
\[ \widehat{\sigma} : |P| \times \left|\Omega^n(X)_{S_1}\right| \times \ldots \times \left|\Omega^n(X)_{S_m}\right| \to \left|\Omega^{n+1}(X)_{S}\right| \] (note the difference in superscripts) is a $\V$-admissible morphism
\[ \widehat{\sigma} : P \times \Omega^n(X)_{S_1} \times \ldots \times \Omega^n(X)_{S_m} \to \Omega^{n+1}(X)_{S} \] by the definition of $\Omega^{n+1}(X)$. Recalling that suprema in $\V^\calS$-fibres are formed pointwise \eqref{prod_top_para}, it then readily follows that
\[ \widehat{\sigma} : \bigvee_n \left(P \times \Omega^n(X)_{S_1} \times \ldots \times \Omega^n(X)_{S_m}\right) \to \left(\bigvee_n \Omega^n(X)\right)_S = \bigvee_n \Omega^n(X)_S \] is $\V$-admissible. Now  
\[ P \times \left(\bigvee_n \Omega^n(X)\right)_{S_1} \times \ldots \times \left(\bigvee_n \Omega^n(X)\right)_{S_m} = P \times \bigvee_n \Omega^n(X)_{S_1} \times \ldots \times \bigvee_n \Omega^n(X)_{S_m}, \] and 
\[ P \times \bigvee_n \Omega^n(X)_{S_1} \times \ldots \times \bigvee_n \Omega^n(X)_{S_m} = \bigvee_n \left(P \times \Omega^n(X)_{S_1} \times \ldots \times \Omega^n(X)_{S_m}\right) \] by Lemma \ref{join_ineq_lem}, since $\V$ is cartesian closed and each sequence $\left(\Omega^n(X)_{S_i}\right)_{n \geq 0} \in \Fib\left(\Term_{|\Sigma|}(|X|)_{S_i}\right)$ ($1 \leq i \leq m$) is increasing. This proves the desired result for $m \geq 1$. For $m = 0$, we must show that the function
\[ \widehat{\sigma} : |P| \to \left|\left(\bigvee_n \Omega^n(X)\right)_S\right| \] \[ p \mapsto \sigma_p \] is $\V$-admissible; but this follows because $\Omega^1(X)_S \leq \left(\bigvee_n \Omega^n(X)\right)_S$ and the function $\widehat{\sigma} : |P| \to \left|\Omega^1(X)_S\right|$ is $\V$-admissible by the definition of $\Omega^1(X)$. This completes the proof of \eqref{explicit_eq}.    
\end{proof}

\noindent The following corollary is the single-sorted version of Theorem \ref{free_sig_alg_cc}. The special case where $\Sigma$ is ordinary and single-sorted essentially appears as \cite[Theorem 1.2]{Porst_cc} (see also \cite[Theorem 3.3]{Porst_existence} for the special case where $\V$ is the cartesian closed category  $\CGTop$ of compactly generated topological spaces, \ref{examples_ex}).

\begin{cor_sub}
\label{free_sig_alg_cc_SS}
Suppose that $(\V, \tensor, I) = (\V, \times, 1)$ and that $\V$ is cartesian closed. Let $\Sigma$ be a $\V$-enriched single-sorted signature with underlying classical single-sorted signature $|\Sigma|$ \eqref{ord_sig}, and let $X$ be an object of $\V$. Then the carrier object $U^\Sigma F^\Sigma X \in \ob\V$ of the free $\Sigma$-algebra $F^\Sigma X$ on $X$ is the supremum \[ U^\Sigma F^\Sigma X = \bigvee_n \Omega^n(X) \] in the fibre over the set $\Term_{|\Sigma|}(|X|)$ of the following objects $\Omega^n(X)$ of $\V$ ($n \geq 0$):
\begin{enumerate}[leftmargin=*]
\item $\Omega^0(X)$ is the final $\V$-structure on the set $\Term_{|\Sigma|}(|X|)$ induced by the inclusion function $\eta = \eta_{|X|}^{|\Sigma|} : |X| \to \Term_{|\Sigma|}(|X|)$. 

\item Given $\Omega^n(X) \in \Fib\left(\Term_{|\Sigma|}(|X|)\right)$ ($n \geq 0$), we first define $\Omega^n(X, \Sigma)$ to be the final $\V$-structure on the set $\Term_{|\Sigma|}(|X|)$ induced by the functions
\[ \widehat{\sigma} : |P| \times \left|\Omega^n(X)\right|^m \to \Term_{|\Sigma|}(|X|) \] 
\[ (p, t_1, \ldots, t_m) \mapsto \sigma_p(t_1, \ldots, t_m) \]
for all $\sigma \in \Sigma$ of arity $m \geq 0$ and parameter $P$. We then define \[ \Omega^{n+1}(X) := \Omega^n(X) \vee \Omega^n(X, \Sigma) \] in the fibre over the set $\Term_{|\Sigma|}(|X|)$.  \qed 
\end{enumerate}     
\end{cor_sub}

\section{Free algebras of enriched multi-sorted equational theories}
\label{theories_section}

As in \S\ref{first_section}, we continue to fix a set $\calS$ of sorts, and a topological category $\V$ over $\Set$ such that $\V = (\V, \tensor, I)$ is also a symmetric monoidal category and the given topological functor $|-| : \V \to \Set$ is strict symmetric monoidal (with respect to the cartesian symmetric monoidal structure on $\Set$). 

We begin this section by defining the notion of \emph{$\V$-enriched $\calS$-sorted equational theory} in Definition \ref{eqn_theory}; we mention related notions of enriched (equational) theory in Remark \ref{alt_th_rmk}. Each $\V$-enriched $\calS$-sorted equational theory $\calT$ has an underlying \emph{classical} $\calS$-sorted equational theory $|\calT|$ (Definition \ref{underlying_trad_theory}), and we show in Theorem \ref{taut_thm} that the adjunction $F^\calT \dashv U^\calT : \calT\Alg \to \V^\calS$ is a lifting of the adjunction $F^{|\calT|} \dashv U^{|\calT|} : |\calT|\Alg \to \Set^\calS$. We study free algebras of $\V$-enriched $\calS$-sorted equational theories in the general case in \S\ref{init_alg_theory_gen_sub}, and in the cartesian closed case in \S\ref{cong_section}.

\subsection{Enriched multi-sorted equational theories}
\label{th_def_subsection}      

\begin{para_sub}[\textbf{The functor $P \tensor U^{\Sbar} : \Sigma\Alg \to \V$}]
\label{forgetful_sort_functor}
Let $\Sigma$ be a $\V$-enriched $\calS$-sorted signature. For each $\Sbar = (S_1, \ldots, S_n) \in \calS^*$, we write $U^{\Sbar} : \Sigma\Alg \to \V$ for the functor defined on objects by $A \mapsto A_{S_1} \times \ldots \times A_{S_n}$ and on morphisms by $f \mapsto f_{S_1} \times \ldots \times f_{S_n}$ (cf.~Remark \ref{enriched_rmk}). In particular, for a single sort $S \in \calS$, we write $U^S := U^{(S)} : \Sigma\Alg \to \V$, which is defined on objects by $A \mapsto A_S$ and on morphisms by $f \mapsto f_S$. For each $P \in \ob\V$ and each $\Sbar \in \calS^*$ we then have the functor $P \tensor U^{\Sbar} : \Sigma\Alg \to \V$ defined as the composite $\Sigma\Alg \xrightarrow{U^{\Sbar}} \V \xrightarrow{P \tensor (-)} \V$, which is therefore given on objects by $A \mapsto P \tensor \left(A_{S_1} \times \ldots \times A_{S_n}\right)$ and on morphisms by $f \mapsto 1_P \tensor \left(f_{S_1} \times \ldots \times f_{S_n}\right)$.
\end{para_sub}   

\begin{defn_sub}
\label{alg_operation}
Let $\Sigma$ be a $\V$-enriched $\calS$-sorted signature. An \textbf{algebraic $\Sigma$-operation} is a natural transformation $\omega : P \tensor U^{\Sbar} \Longrightarrow U^S : \Sigma\Alg \to \V$ for some $P \in \ob\V$, some $\Sbar \in \calS^*$, and some $S \in \calS$. Explicitly, an algebraic $\Sigma$-operation $\omega$ consists of $\V$-morphisms $\omega_A : P \tensor \left(A_{S_1} \times \ldots \times A_{S_n}\right) \to A_S$ ($A \in \ob\Sigma\Alg$) such that the following diagram commutes for each morphism $f : A \to B$ of $\Sigma\Alg$:
\begin{equation}\label{alg_sigma_op_eq}
\begin{tikzcd}
	{P \tensor \left(A_{S_1} \times \ldots \times A_{S_n}\right)} &&& {P \tensor \left(B_{S_1} \times \ldots \times B_{S_n}\right)} \\
	\\
	{A_S} &&& {B_S}.
	\arrow["{1_P \tensor \left(f_{S_1} \times \ldots \times f_{S_n}\right)}", from=1-1, to=1-4]
	\arrow["{\omega_A}"', from=1-1, to=3-1]
	\arrow["{\omega_B}", from=1-4, to=3-4]
	\arrow["{f_S}"', from=3-1, to=3-4]
\end{tikzcd}
\end{equation}
We say that $P$ is the \emph{parameter (object)} of $\omega$, that $\omega$ has \emph{input sorts} $\Sbar$, that $S$ is the \emph{output sort} of $\omega$, and that $\omega$ has \emph{type} $\left(\Sbar, S, P\right)$. An algebraic $\Sigma$-operation $\omega$ is \textbf{ordinary} if its parameter object is the unit object $I$ of the symmetric monoidal category $\V$, in which case we simply write $\omega : U^{\Sbar} \Longrightarrow U^S$ and $\omega_A : A_{S_1} \times \ldots \times A_{S_n} \to A_S$ for each $\Sigma$-algebra $A$.
\end{defn_sub}

\begin{defn_sub}
\label{equation}
Let $\Sigma$ be a $\V$-enriched $\calS$-sorted signature. An \textbf{algebraic $\Sigma$-equation}, written $\omega \doteq \nu$, is a pair of algebraic $\Sigma$-operations $\omega, \nu : P \tensor U^{\Sbar} \Longrightarrow U^S$ of the same type $\left(\Sbar, S, P\right)$. We say that $P$ is the \emph{parameter (object)} of $\omega \doteq \nu$, that $\omega \doteq \nu$ has \emph{input sorts} $\Sbar$, that $S$ is the \emph{output sort} of $\omega \doteq \nu$, and that $\omega \doteq \nu$ has \emph{type} $\left(\Sbar, S, P\right)$. We say that $\omega \doteq \nu$ is \textbf{ordinary} if its parameter object is the unit object $I$ of the symmetric monoidal category $\V$ (i.e.~if $\omega$ and $\nu$ are ordinary). A $\Sigma$-algebra $A$ \textbf{satisfies} an algebraic $\Sigma$-equation $\omega \doteq \nu$ of type $\left(\Sbar, S, P\right)$ if \[ \omega_A = \nu_A : P \tensor \left(A_{S_1} \times \ldots \times A_{S_n}\right) \to A_S. \] 
\end{defn_sub}

\begin{defn_sub}
\label{eqn_theory}
A \textbf{$\V$-enriched $\calS$-sorted equational theory} is a pair $\calT = (\Sigma, \E)$ consisting of a $\V$-enriched $\calS$-sorted signature $\Sigma$ and a set $\E$ of algebraic $\Sigma$-equations. A $\Sigma$-algebra $A$ is a \textbf{$\calT$-model} or \textbf{$\calT$-algebra} or \textbf{model of $\calT$} if it satisfies each algebraic $\Sigma$-equation in $\E$. We write $\calT\Alg$ for the full subcategory of $\Sigma\Alg$ consisting of the $\calT$-algebras, which is equipped with the faithful functor $U^\calT : \calT\Alg \to \V^\calS$ obtained by restricting $U^\Sigma : \Sigma\Alg \to \V^\calS$. A $\V$-enriched $\calS$-sorted equational theory $\calT = (\Sigma, \calE)$ is \textbf{ordinary} if $\Sigma$ is ordinary \eqref{signature} and each algebraic $\Sigma$-equation in $\calE$ is ordinary \eqref{equation}. Given a fixed $\V$-enriched $\calS$-sorted signature $\Sigma$, a \textbf{$\V$-enriched $\calS$-sorted equational theory over $\Sigma$} is a $\V$-enriched $\calS$-sorted equational theory $\calT = (\Sigma, \calE)$ whose first component is $\Sigma$. 

When $\calS$ is a singleton, we refer to a $\V$-enriched $\calS$-sorted equational theory (over $\Sigma$) as a \textbf{$\V$-enriched single-sorted equational theory (over $\Sigma$)}.  
\end{defn_sub}

\begin{rmk_sub}
\label{alt_th_rmk}
In Theorem \ref{equiv_theory_prop} below, we shall prove that there is an equivalent, more \emph{syntactic}, way of describing $\V$-enriched $\calS$-sorted equational theories. The definition of $\V$-enriched $\calS$-sorted equational theory that we have given in Definition \ref{eqn_theory} will allow us to more readily establish in \S\ref{smc_section} that, when the symmetric monoidal structure of $\V$ is \emph{closed}, $\V$-enriched $\calS$-sorted equational theories are presentations (in a suitable sense) of certain $\V$-enriched monads on $\V^\calS$ (see Theorem \ref{mnd_pres_thm} and Remark \ref{amenable_rmk}). Definition \ref{eqn_theory} is also an $\calS$-sorted generalization of \cite[Definition 3.3]{Battenfeld_comp} (wherein $\V$-enriched single-sorted equational theories are referred to as \emph{(finitary) parameterized equational theories for $\V$}). 
\end{rmk_sub}

\begin{rmk_sub}
\label{enriched_rmk_2}
As in Remark \ref{enriched_rmk}, when the symmetric monoidal structure of $\V$ is \emph{closed} and $|-| : \V \to \Set$ is represented by the unit object $I$ of $\V$, the category $\calT\Alg$ underlies a $\V$-category (also denoted $\calT\Alg$) and the functor $U^\calT : \calT\Alg \to \V^\calS$ underlies a $\V$-functor (also denoted $U^\calT$). Indeed, we simply take the $\V$-category $\calT\Alg$ to be the full sub-$\V$-category of the $\V$-category $\Sigma\Alg$ \eqref{enriched_rmk} consisting of the $\calT$-algebras, and we take the $\V$-functor $U^\calT : \calT\Alg \to \V^\calS$ to be the restriction of the $\V$-functor $U^\Sigma : \Sigma\Alg \to \V^\calS$. However, as our focus in this section is primarily on the explicit description of free $\calT$-algebras (i.e.~of the objects in the image of a left adjoint to $U^\calT$), we shall not pay much attention to the $\V$-enriched structure of $\calT\Alg$ until \S\ref{smc_section}.  
\end{rmk_sub}

\noindent Our next objective is to ultimately show (in Theorem \ref{equiv_theory_prop} below) that $\V$-enriched $\calS$-sorted equational theories can be given a more traditional \emph{syntactic} formulation, with the algebraic $\Sigma$-equations of Definition \ref{equation} replaced by syntactic equations between \emph{terms} over the underlying classical $\calS$-sorted signature $|\Sigma|$ \eqref{ord_sig}. For this purpose, we shall require the following material.

\begin{para_sub}[\textbf{$\Indisc : \Set \to \V$ is monoidal}]
\label{indisc_para}
Recall from \ref{fibre_para} that the topological functor $|-| : \V \to \Set$ has a fully faithful right adjoint section $\Indisc : \Set \to \V$ that sends a set $S$ to the indiscrete object $\Indisc S$ over $S$ (the top element of the fibre $\Fib(S)$). Since $|-| : \V \to \Set$ is strict monoidal, it follows by \emph{doctrinal adjunction} \cite[1.5]{Kellydoctrinal} that the adjunction $|-| \dashv \Indisc : \Set \to \V$ is \emph{monoidal}, i.e.~that the right adjoint $\Indisc : \Set \to \V$ is a \emph{(lax) monoidal} functor (this is also easy to deduce directly from the definition of $\Indisc$). As a right adjoint, the functor $\Indisc : \Set \to \V$ also (strictly) preserves products.
\end{para_sub} 

\begin{defn_sub}[\textbf{The indiscrete $\Sigma$-algebra determined by a $|\Sigma|$-algebra}]
\label{indiscrete_algebra}
Let $\Sigma$ be a $\V$-enriched $\calS$-sorted signature with underlying classical $\calS$-sorted signature $|\Sigma|$ \eqref{ord_sig}, and let $A$ be a $|\Sigma|$-algebra. Since $\Indisc : \Set \to \V$ is monoidal and strictly preserves finite products \eqref{indisc_para}, we obtain an associated $\Sigma$-algebra $A^{\Indisc}$, which we may call \emph{the indiscrete $\Sigma$-algebra determined by the $|\Sigma|$-algebra $A$}, as follows. The $\calS$-sorted carrier object of $\V$ is $A^{\Indisc} := \left(\Indisc A_S\right)_{S \in \calS}$. For each $\sigma \in \Sigma$ of type $((S_1, \ldots, S_n), S, P)$, we first write
\[ \sigma^A : |P| \times A_{S_1} \times \ldots \times A_{S_n} \to A_S \] for the function
\[ (p, a_1, \ldots, a_n) \to \sigma_p^A(a_1, \ldots, a_n). \]  
We then define
\[ \sigma^{A^{\Indisc}} : P \tensor \left(\Indisc A_{S_1} \times \ldots \times \Indisc A_{S_n}\right) \to \Indisc A_S \] to be the following (two-line) composite $\V$-morphism, whose components $\alpha$ and $\beta$ are defined subsequently:
\[ P \tensor \left(\Indisc A_{S_1} \times \ldots \times \Indisc A_{S_n}\right) \xrightarrow{\alpha} \Indisc|P| \tensor \left(\Indisc A_{S_1} \times \ldots \times \Indisc A_{S_n}\right) \]
\[ = \Indisc|P| \tensor \Indisc\left(A_{S_1} \times \ldots \times A_{S_n}\right) \xrightarrow{\beta} \Indisc\left(|P| \times \left(A_{S_1} \times \ldots \times A_{S_n}\right)\right) \xrightarrow{\Indisc \sigma^A} \Indisc A_S. \]
The morphism $\alpha$ arises from the inequality $P \leq \Indisc|P|$ in $\Fib(|P|)$, while the morphism $\beta$ arises from the fact that $\Indisc$ is monoidal (the underlying functions of both $\alpha$ and $\beta$ are identities). This completes the definition of the $\Sigma$-algebra $A^{\Indisc}$.

The assignment $A \mapsto A^{\Indisc}$ ($A \in \ob|\Sigma|\Alg$) readily extends to a functor \[ (-)^{\Indisc} : |\Sigma|\Alg \to \Sigma\Alg \] given on a morphism $f : A \to B$ of $|\Sigma|$-algebras by
\[ f^{\Indisc} := \left(\Indisc f_S\right)_{S \in \calS} = \left(f_S\right)_{S \in \calS}, \] recalling that $\Indisc$ is a section of $|-| : \V \to \Set$.   
\end{defn_sub}

\begin{defn_sub}[\textbf{The ordinary algebraic $|\Sigma|$-operations determined by an algebraic $\Sigma$-operation}]
\label{image_alg_operation}
Let $\Sigma$ be a $\V$-enriched $\calS$-sorted signature with underlying classical $\calS$-sorted signature $|\Sigma|$ \eqref{ord_sig}, and let $\omega$ be an algebraic $\Sigma$-operation of type $\left((S_1, \ldots, S_n), S, P\right)$. For each $p \in |P|$, we define a ordinary algebraic $|\Sigma|$-operation $|\omega_p|$ with input sorts $(S_1, \ldots, S_n)$ and output sort $S$ as follows. For each $|\Sigma|$-algebra $A$, we shall define a function \[ |\omega_p|_A : A_{S_1} \times \ldots \times A_{S_n} \to A_S. \] By \ref{indiscrete_algebra}, we have an associated $\Sigma$-algebra $A^{\Indisc}$, and so we have the $\V$-morphism \[ \omega_{A^{\Indisc}} : P \tensor \left(\Indisc A_{S_1} \times \ldots \times \Indisc A_{S_n}\right) \to \Indisc A_S. \] 
Since $\Indisc : \Set \to \V$ is a section of $|-| : \V \to \Set$ and $|-|$ strictly preserves products, we then define $|\omega_p|_A$ to be the function \[ |\omega_p|_A := \left(A_{S_1} \times \ldots \times A_{S_n} = \left|\Indisc A_{S_1} \times \ldots \times \Indisc A_{S_n}\right| \xrightarrow{\omega_{A^{\Indisc}}(p, -)} |\Indisc A_{S}| = A_S\right). \] 
It then readily follows that $|\omega_p| := \left(|\omega_p|_A\right)_{A \in |\Sigma|\Alg}$ is indeed an algebraic $|\Sigma|$-operation.         
\end{defn_sub}  

\begin{defn_sub}[\textbf{The ordinary algebraic $|\Sigma|$-equations determined by an algebraic $\Sigma$-equation}]
\label{ord_eqn}
Let $\Sigma$ be a $\V$-enriched $\calS$-sorted signature with underlying classical $\calS$-sorted signature $|\Sigma|$ \eqref{ord_sig}. Given an algebraic $\Sigma$-equation $\sfe \equiv \left(\omega \doteq \nu\right)$ of type $((S_1, \ldots, S_n), S, P)$, for each $p \in |P|$ we define (in view of \ref{image_alg_operation}) a ordinary algebraic $|\Sigma|$-equation 
\begin{equation}
\label{e_p}
\sfe_p := \left(|\omega_p| \doteq |\nu_p|\right)
\end{equation} 
with input sorts $(S_1, \ldots, S_n)$ and output sort $S$. 
\end{defn_sub}  

\begin{prop_sub}
\label{ord_eqn_prop}
Let $\Sigma$ be a $\V$-enriched $\calS$-sorted signature with underlying classical $\calS$-sorted signature $|\Sigma|$ \eqref{ord_sig}, let $\sfe$ be an algebraic $\Sigma$-equation with parameter $P$, and let $A$ be a $\Sigma$-algebra. Then $A$ satisfies $\sfe$ iff the underlying $|\Sigma|$-algebra $|A|$ \eqref{ord_sig_para} satisfies the ordinary algebraic $|\Sigma|$-equations $\sfe_p$ \eqref{e_p} for all $p \in |P|$.  
\end{prop_sub}

\begin{proof}
Supposing that $\sfe \equiv \left(\omega \doteq \nu\right)$ is of type $((S_1, \ldots, S_n), S, P)$, we need to show that \[ \omega_A = \nu_A : P \tensor \left(A_{S_1} \times \ldots \times A_{S_n}\right) \to A_S \] iff for all $p \in |P|$ we have \[ |\omega_p|_{|A|} = |\nu_p|_{|A|} : \left|A_{S_1}\right| \times \ldots \times \left|A_{S_n}\right| \to \left|A_S\right|, \] where $|\omega_p|_{|A|} = \omega_{|A|^{\Indisc}}(p, -)$ and $|\nu_p|_{|A|} = \nu_{|A|^{\Indisc}}(p, -)$ (as functions in $\Set$) by \ref{image_alg_operation}. So it suffices to show that \[ \omega_A = \omega_{|A|^{\Indisc}} : |P| \times \left|A_{S_1}\right| \times \ldots \times \left|A_{S_n}\right| \to \left|A_S\right| \] and $\nu_A = \nu_{|A|^{\Indisc}}$ as functions in $\Set$. We have the morphism \[ \left(1_{\left|A_S\right|} : A_S \to \Indisc \left|A_S\right|\right)_{S \in \calS} : A \to |A|^{\Indisc} \] of $\Sigma\Alg$, and hence the following diagram commutes because $\omega$ is an algebraic $\Sigma$-operation, which yields the result for $\omega$ (the proof for $\nu$ being identical):
\[\begin{tikzcd}
	{P \tensor \left(A_{S_1} \times \ldots \times A_{S_n}\right)} &&& {P \tensor \left(\Indisc\left|A_{S_1}\right| \times \ldots \times \Indisc\left|A_{S_n}\right|\right)} \\
	\\
	{A_S} &&& {\Indisc\left|A_S\right|}.
	\arrow["{1_P \tensor \left(1_{\left|A_{S_1}\right|} \times \ldots \times 1_{\left|A_{S_n}\right|}\right)}", from=1-1, to=1-4]
	\arrow["{\omega_A}"', from=1-1, to=3-1]
	\arrow["{\omega_{|A|^{\Indisc}}}", from=1-4, to=3-4]
	\arrow["{1_{\left|A_S\right|}}"', from=3-1, to=3-4]
\end{tikzcd}\] 
\end{proof}

\noindent Proposition \ref{ord_eqn_prop} will provide the basis for showing, in Theorem \ref{equiv_theory_prop} below, that $\V$-enriched $\calS$-sorted equational theories can be given a more \emph{syntactic} formulation. Building up to Theorem \ref{equiv_theory_prop} will require some review of the syntax and semantics of $\Set$-enriched $\calS$-sorted equational theories, which we cover in \ref{syntax_para} and \ref{nat_syntax}.

\begin{para_sub}[\textbf{$\Sigma$-terms in context and their interpretations}]
\label{syntax_para}
Let $\Sigma$ be a classical $\calS$-sorted signature \eqref{signature}, and let $\Var_S$ ($S \in \calS$) be pairwise disjoint countably infinite sets of variable symbols. An \emph{$\calS$-sorted variable}\footnote{When $\calS$ is a singleton, so that $\Sigma$ is a classical \emph{single-sorted} signature, we can omit explicit mention of the unique sort in the concepts that we discuss here.} is an expression of the form $v : S$, where $S \in \calS$ is a sort and $v \in \Var_S$. An \emph{$\calS$-sorted variable context}, or simply a \emph{context}, is a finite (possibly empty) list of $\calS$-sorted variables, typically written $v_1 : S_1, \ldots, v_n : S_n$. Given a context $\vbar \equiv v_1 : S_1, \ldots, v_n : S_n$, we recursively define, for each sort $S \in \calS$, the \emph{$\Sigma$-terms of sort $S$ in context $\vbar$}, by the following clauses:
\begin{enumerate}[leftmargin=*]
\item For each $1 \leq i \leq n$, the expression $[\vbar \vdash v_i : S_i]$ is a $\Sigma$-term of sort $S_i$ in context $\vbar$.
\item For each operation symbol $\sigma \in \Sigma$ with input sorts $T_1, \ldots, T_m$ and output sort $T$ and all $\Sigma$-terms $[\vbar \vdash t_j : T_j]$ ($1 \leq j \leq m$) of sort $T_j$ in context $\vbar$, the expression $[\vbar \vdash \sigma(t_1, \ldots, t_m) : T]$ is a $\Sigma$-term of sort $T$ in context $\vbar$.    
\end{enumerate} 
For each sort $S \in \calS$, we write $\Term(\Sigma; \vbar)_S$ for the set of $\Sigma$-terms of sort $S$ in context $\vbar$, and we thus obtain the $\calS$-sorted set $\Term(\Sigma; \vbar) := \left(\Term(\Sigma; \vbar)_S\right)_{S \in \calS}$ of \emph{$\Sigma$-terms in context $\vbar$}. 

Given a context $\vbar$ and a sort $S \in \calS$, a \emph{syntactic $\Sigma$-equation of sort $S$ in context $\vbar$} is an ordered pair of $\Sigma$-terms of sort $S$ in context $\vbar$, and we write such an equation as $[\vbar \vdash s \doteq t : S]$. A \emph{syntactic $\Sigma$-equation in context $\vbar$} is a syntactic $\Sigma$-equation of sort $S$ in context $\vbar$ for some sort $S \in \calS$, while a \emph{syntactic $\Sigma$-equation (in context)} is a syntactic $\Sigma$-equation in context $\vbar$ for some context $\vbar$.  

Let $A$ be a $\Sigma$-algebra. Given a context $\vbar \equiv v_1 : S_1, \ldots, v_n : S_n$, we recursively define, for each sort $S \in \calS$, the \emph{interpretation in $A$} of each $\Sigma$-term $[\vbar \vdash t : S]$ of sort $S$ in context $\vbar$, which shall be a function
\[ \left[\vbar \vdash t : S\right]^A : A_{S_1} \times \ldots \times A_{S_n} \to A_S, \] by the following clauses:
\begin{enumerate}[leftmargin=*]
\item For each $1 \leq i \leq n$, we define \[ \left[\vbar \vdash v_i : S_i\right]^A : A_{S_1} \times \ldots \times A_{S_n} \to A_{S_i} \] to be the product projection onto $A_{S_i}$.

\item For each operation symbol $\sigma \in \Sigma$ with input sorts $T_1, \ldots, T_m$ and output sort $T$ and all $\Sigma$-terms $[\vbar \vdash t_j : T_j]$ ($1 \leq j \leq m$) of sort $T_j$ in context $\vbar$, we define 
\[ \left[\vbar \vdash \sigma(t_1, \ldots, t_m) : T\right]^A : A_{S_1} \times \ldots \times A_{S_n} \to A_T \] to be the following composite function:
\[ A_{S_1} \times \ldots \times A_{S_n} \xrightarrow{\left\langle \left[\vbar \vdash t_1 : T_1\right]^A, \ldots, \left[\vbar \vdash t_m : T_m\right]^A\right\rangle} A_{T_1} \times \ldots \times A_{T_m} \xrightarrow{\sigma^A} A_T. \]
\end{enumerate} 
We say that a $\Sigma$-algebra $A$ \emph{satisfies} a syntactic $\Sigma$-equation $[\vbar \vdash s \doteq t : S]$ of sort $S$ in context $\vbar$ if
\begin{equation}
\label{sat_synt}
\left[\vbar \vdash s : S\right]^A = \left[\vbar \vdash t : S\right]^A : A_{S_1} \times \ldots \times A_{S_n} \to A_S.
\end{equation} 

Given a morphism of $\Sigma$-algebras $f : A \to B$ and a sort $S \in \calS$, one readily verifies by induction on the $\Sigma$-term $[\vbar \vdash t : S]$ of sort $S$ in context $\vbar$ that the following diagram commutes:
\begin{equation}
\label{synt_to_nat_square}
\begin{tikzcd}
	{A_{S_1} \times \ldots \times A_{S_n}} &&& {B_{S_1} \times \ldots \times B_{S_n}} \\
	\\
	{A_S} &&& {B_S}.
	\arrow["{f_{S_1} \times \ldots \times f_{S_n}}", from=1-1, to=1-4]
	\arrow["{\left[\vbar \vdash t : S\right]^A}"', from=1-1, to=3-1]
	\arrow["{\left[\vbar \vdash t : S\right]^B}", from=1-4, to=3-4]
	\arrow["{f_S}"', from=3-1, to=3-4]
\end{tikzcd}
\end{equation}
Thus, we have an ordinary algebraic $\Sigma$-operation
\begin{equation}
\label{synt_to_nat_def}
\omega_t := \left(\left[\vbar \vdash t : S\right]^A\right)_{A \in \Sigma\Alg}
\end{equation} 
with input sorts $(S_1, \ldots, S_n)$ and output sort $S$. 
\end{para_sub}

\begin{lem_sub}[\textbf{The ordinary algebraic $\Sigma$-operation determined by a $|\Sigma|$-term in context}]
\label{synt_to_nat_lem}
Let $\Sigma$ be a $\V$-enriched $\calS$-sorted signature with underlying classical $\calS$-sorted signature $|\Sigma|$ \eqref{ord_sig}. Let $S \in \calS$ be a sort, let $\vbar \equiv v_1 : S_1, \ldots, v_n : S_n$ be a context, and let $[\vbar \vdash t : S]$ be a $|\Sigma|$-term of sort $S$ in context $\vbar$, so that we have the ordinary algebraic $|\Sigma|$-operation $\omega_t$ \eqref{synt_to_nat_def} with input sorts $(S_1, \ldots, S_n)$ and output sort $S$.
Then there is an ordinary algebraic $\Sigma$-operation $\widehat{\omega}_t$ with input sorts $(S_1, \ldots, S_n)$ and output sort $S$ such that
\begin{equation}
\label{disc_to_disc}
\left(\widehat{\omega}_t\right)_A = \left(\omega_t\right)_{|A|} = \left[\vbar \vdash t : S\right]^{|A|} : \left|A_{S_1}\right| \times \ldots \times \left|A_{S_n}\right| \to \left|A_{S}\right|
\end{equation} 
for each $\Sigma$-algebra $A$, where $|A|$ is the underlying $|\Sigma|$-algebra of $A$ \eqref{ord_sig_para}. Thus, we may write
\begin{equation}
\label{disc_to_disc_2}
\left(\widehat{\omega}_t\right)_A = \left[\vbar \vdash t : S\right]^{A} : A_{S_1} \times \ldots \times A_{S_n} \to A_{S}.
\end{equation}   
\end{lem_sub}

\begin{proof}
We simply take \eqref{disc_to_disc} as the definition of $\left(\widehat{\omega}_t\right)_A$, which is well-defined, since it readily follows by induction on $[\vbar \vdash t : S]$ that $\left(\omega_t\right)_{|A|} = \left[\vbar \vdash t : S\right]^{|A|}$ lifts to a $\V$-morphism $\left(\widehat{\omega}_t\right)_A : A_{S_1} \times \ldots \times A_{S_n} \to A_S$. Then $\widehat{\omega}_t$ is an algebraic $\Sigma$-operation, because $\omega_t$ is an algebraic $|\Sigma|$-operation and every morphism of $\Sigma$-algebras is also a morphism of the underlying $|\Sigma|$-algebras by \eqref{ord_sig_para}. 
\end{proof}

\begin{para_sub}[\textbf{Correspondence between ordinary algebraic $\Sigma$-operations and $\Sigma$-terms in context}]
\label{nat_syntax}
Let $\Sigma$ be a classical $\calS$-sorted signature \eqref{signature}. Given a context $\vbar \equiv v_1 : S_1, \ldots, v_n : S_n$ and a sort $S \in \calS$, it is known that ordinary algebraic $\Sigma$-operations $\omega$ with input sorts $(S_1, \ldots, S_n)$ and output sort $S$ are in bijective correspondence with $\Sigma$-terms $[\vbar \vdash t : S]$ of sort $S$ in context $\vbar$, such that when $\omega$ and $[\vbar \vdash t : S]$ are related under this correspondence, then 
\[ \omega_A = \left[\vbar \vdash t : S\right]^A : A_{S_1} \times \ldots \times A_{S_n} \to A_S \] for each $\Sigma$-algebra $A$. To recall the details of this correspondence, first note that the context $\vbar$ canonically determines an $\calS$-sorted set $X^{\vbar} = \left(X^{\vbar}_S\right)_{S \in \calS}$, where for each $1 \leq i \leq n$, the set $X^{\vbar}_{S_i}$ consists of the $\Sigma$-terms in context $\vbar$ of the form $\left[\vbar \vdash v : S_i\right]$ such that $v : S_i$ occurs in $\vbar$, and $X^{\vbar}_S := \varnothing$ for $S \notin \{S_1, \ldots, S_n\}$. We then write $F^\Sigma \vbar$ for the free $\Sigma$-algebra on the $\calS$-sorted set $X^{\vbar}$ determined by $\vbar$. For each sort $S \in \calS$, the carrier set of $F^\Sigma \vbar$ at the sort $S$ is (by \ref{free_sig_alg_Set}) the set $\Term_\Sigma\left(X^{\vbar}\right)_S$ of ground $\Sigma$-terms of sort $S$ with constants from $X^{\vbar}$, which we can clearly identify with the set $\Term(\Sigma; \vbar)_S$ of $\Sigma$-terms of sort $S$ in context $\vbar$ \eqref{syntax_para}.

Now, given an ordinary algebraic $\Sigma$-operation $\omega$ with input sorts $(S_1, \ldots, S_n)$ and output sort $S$, we have the component function
\[ \omega_{F^\Sigma \vbar} : \Term(\Sigma; \vbar)_{S_1} \times \ldots \times \Term(\Sigma; \vbar)_{S_n} \to \Term(\Sigma; \vbar)_S. \] The $\Sigma$-term $\left[\vbar \vdash t_\omega : S\right]$ of sort $S$ in context $\vbar$ that corresponds to $\omega$ is then
\begin{equation}
\label{t_omega_def} 
\left[\vbar \vdash t_\omega : S\right] := \omega_{F^\Sigma \vbar}\left([\vbar \vdash v_1 : S_1], \ldots, [\vbar \vdash v_n : S_n]\right) \in \Term(\Sigma; \vbar)_S.
\end{equation} 
For each $\Sigma$-algebra $A$ we have
\begin{equation}
\label{nat_to_synt_eq}
\omega_A = \left[\vbar \vdash t_\omega : S\right]^A : A_{S_1} \times \ldots \times A_{S_n} \to A_S
\end{equation} 
by the following argument. Given $a_1 \in A_{S_1}, \ldots, a_n \in A_{S_n}$, we must show that 
\begin{equation}
\label{nat_to_synt_eq2}
\omega_A(a_1, \ldots, a_n) = \left[\vbar \vdash t_\omega : S\right]^A(a_1, \ldots, a_n) \in A_S.
\end{equation} 
Let $h : F^\Sigma \vbar \to A$ be the unique morphism of $\Sigma$-algebras with the property that $h_{S_i}\left([\vbar \vdash v_i : S_i]\right) = a_i$ for each $1 \leq i \leq n$. Since $\omega$ is an algebraic $\Sigma$-operation, the following diagram commutes:
\begin{equation}\label{nat_to_synt_diag}
\begin{tikzcd}
	{\Term(\Sigma; \vbar)_{S_1} \times \ldots \times \Term(\Sigma; \vbar)_{S_n}} &&& {A_{S_1} \times \ldots \times A_{S_n}} \\
	\\
	{\Term(\Sigma; \vbar)_{S}} &&& {A_S}.
	\arrow["{h_{S_1} \times \ldots \times h_{S_n}}", from=1-1, to=1-4]
	\arrow["{\omega_{F^\Sigma \vbar}}"', from=1-1, to=3-1]
	\arrow["{\omega_A}", from=1-4, to=3-4]
	\arrow["{h_S}"', from=3-1, to=3-4]
\end{tikzcd}
\end{equation}
Moreover, for each sort $T \in \calS$ and each $\Sigma$-term $[\vbar \vdash t : T]$ of sort $T$ in context $\vbar$, one readily verifies that
\begin{equation}
\label{h_eq}
h_T([\vbar \vdash t : T]) = [\vbar \vdash t : T]^A(a_1, \ldots, a_n) \in A_T.
\end{equation} 
Using \eqref{nat_to_synt_diag}, \eqref{h_eq}, and the definition of $h$, we then have
\[ \omega_A(a_1, \ldots, a_n) = \omega_A\left(h_{S_1}\left([\vbar \vdash v_1 : S_1]\right), \ldots, h_{S_n}\left([\vbar \vdash v_n : S_n]\right)\right) \] \[ = h_S\left(\omega_{F^\Sigma \vbar}\left([\vbar \vdash v_1 : S_1], \ldots, [\vbar \vdash v_n : S_n]\right)\right) = h_S\left([\vbar \vdash t_\omega : S]\right) = \left[\vbar \vdash t_\omega : S\right]^A(a_1, \ldots, a_n), \] thereby proving \eqref{nat_to_synt_eq2} and thus \eqref{nat_to_synt_eq}. The injectivity of the assignment $\omega \mapsto \left[\vbar \vdash t_\omega : S\right]$ follows immediately from \eqref{nat_to_synt_eq}.

Conversely, given a $\Sigma$-term $[\vbar \vdash t : S]$ of sort $S$ in context $\vbar$, in \eqref{synt_to_nat_def} we defined a corresponding ordinary algebraic $\Sigma$-operation $\omega_t$ with input sorts $(S_1, \ldots, S_n)$ and output sort $S$ with the desired property that
\begin{equation}
\label{synt_to_nat_eq}
\left(\omega_t\right)_A = \left[\vbar \vdash t : S\right]^A : A_{S_1} \times \ldots \times A_{S_n} \to A_S
\end{equation}
for each $\Sigma$-algebra $A$. Using \eqref{t_omega_def} and \eqref{synt_to_nat_eq}, we have $\left[\vbar \vdash t_{\omega_t} : S\right] = \left[\vbar \vdash t : S\right]$ because 
\[ \left[\vbar \vdash t_{\omega_t} : S\right] = \left(\omega_t\right)_{F^\Sigma \vbar}\left([\vbar \vdash v_1 : S_1], \ldots, [\vbar \vdash v_n : S_n]\right) \] \[ = \left[\vbar \vdash t : S\right]^{F^\Sigma \vbar}\left([\vbar \vdash v_1 : S_1], \ldots, [\vbar \vdash v_n : S_n]\right) = \left[\vbar \vdash t : S\right], \]
where the last equality readily follows by induction on $[\vbar \vdash t : S]$. 
\end{para_sub}

\begin{prop_sub}
\label{equiv_eqn_prop}
Let $\Sigma$ be a $\V$-enriched $\calS$-sorted signature with underlying classical $\calS$-sorted signature $|\Sigma|$ \eqref{ord_sig}. Then the following hold:
\begin{enumerate}[leftmargin=*]
\item For each algebraic $\Sigma$-equation $\sfe$ with input sorts $(S_1, \ldots, S_n)$ and output sort $S$, there is a set $\calE_\sfe$ of ordinary algebraic $|\Sigma|$-equations, each with input sorts $(S_1, \ldots, S_n)$ and output sort $S$, which has the property that a $\Sigma$-algebra $A$ satisfies $\sfe$ iff the underlying $|\Sigma|$-algebra $|A|$ \eqref{ord_sig_para} satisfies each element of $\calE_\sfe$. 

\item Given a finite tuple $(S_1, \ldots, S_n)$ of sorts, we write $\vbar$ for the variable context $\vbar \equiv v_1 : S_1, \ldots, v_n : S_n$, whose elements are assumed to be pairwise distinct. For each ordinary algebraic $|\Sigma|$-equation $\sfe$ with input sorts $(S_1, \ldots, S_n)$ and output sort $S$, there is a syntactic $|\Sigma|$-equation $[\vbar \vdash s_\sfe \doteq t_\sfe : S]$ of sort $S$ in context $\vbar$ \eqref{syntax_para} with the property that a $|\Sigma|$-algebra satisfies $\sfe$ iff it satisfies $[\vbar \vdash s_\sfe \doteq t_\sfe : S]$ \eqref{sat_synt}. 

\item For each syntactic $|\Sigma|$-equation $[\vbar \vdash s \doteq t : S]$ of sort $S$ in context $\vbar \equiv v_1 : S_1, \ldots, v_n : S_n$ \eqref{syntax_para}, there is an ordinary algebraic $\Sigma$-equation $\widehat{\omega}_s \doteq \widehat{\omega}_t$ with input sorts $(S_1, \ldots, S_n)$ and output sort $S$ such that a $\Sigma$-algebra $A$ satisfies $\widehat{\omega}_s \doteq \widehat{\omega}_t$ iff the underlying $|\Sigma|$-algebra $|A|$ \eqref{ord_sig_para} satisfies $[\vbar \vdash s \doteq t : S]$ \eqref{sat_synt}.      
\end{enumerate}
\end{prop_sub}

\begin{proof}
For (1), writing $P$ for the parameter object of the algebraic $\Sigma$-equation $\sfe$, we take $\calE_\sfe := \left\{\sfe_p \mid p \in |P|\right\}$ \eqref{e_p}, and $\calE_\sfe$ has the desired property by Proposition \ref{ord_eqn_prop}. For (2), let $\sfe \equiv (\omega \doteq \nu)$ for ordinary algebraic $|\Sigma|$-operations $\omega$ and $\nu$. By \ref{nat_syntax} and \eqref{nat_to_synt_eq} (applied to the classical $\calS$-sorted signature $|\Sigma|$) there are $|\Sigma|$-terms $[\vbar \vdash t_\omega : S]$ and $[\vbar \vdash t_\nu : S]$ of sort $S$ in context $\vbar$ such that $\omega_A = \left[\vbar \vdash t_\omega : S\right]^A$ and $\nu_A = \left[\vbar \vdash t_\nu : S\right]^A$ for every $|\Sigma|$-algebra $A$, so the syntactic $|\Sigma|$-equation $[\vbar \vdash t_\omega \doteq t_\nu : S]$ has the desired property. 

For (3), by Lemma \ref{synt_to_nat_lem} there are ordinary algebraic $\Sigma$-operations $\widehat{\omega}_s$ and $\widehat{\omega}_t$ with input sorts $(S_1, \ldots, S_n)$ and output sort $S$ such that $\left(\widehat{\omega}_s\right)_A = \left[\vbar \vdash s : S\right]^{|A|}$ and $\left(\widehat{\omega}_t\right)_A = \left[\vbar \vdash t : S\right]^{|A|}$ for every $\Sigma$-algebra $A$. The algebraic $\Sigma$-equation $\widehat{\omega}_s \doteq \widehat{\omega}_t$ then has the desired property.            
\end{proof}

\begin{defn_sub}
\label{trad_theory}
A \textbf{classical $\calS$-sorted equational theory} is a pair $\calT = (\Sigma, \calE)$ consisting of a classical $\calS$-sorted signature $\Sigma$ \eqref{signature} and a set $\calE$ of syntactic $\Sigma$-equations in context \eqref{syntax_para}. A $\Sigma$-algebra $A$ is a \textbf{$\calT$-algebra} or \textbf{$\calT$-model} or \textbf{model of $\calT$} if it satisfies each syntactic $\Sigma$-equation in context in $\calE$ \eqref{sat_synt}. 
Given a fixed classical $\calS$-sorted signature $\Sigma$, a \textbf{classical $\calS$-sorted equational theory over $\Sigma$} is a classical $\calS$-sorted equational theory $\calT = (\Sigma, \calE)$ whose first component is $\Sigma$. When $\calS$ is a singleton, we speak instead of a \textbf{classical single-sorted equational theory (over $\Sigma$)}. 
\end{defn_sub}

\begin{theo_sub}
\label{equiv_theory_prop}
A $\V$-enriched $\calS$-sorted equational theory is equivalently given by a $\V$-enriched $\calS$-sorted signature $\Sigma$ and a classical $\calS$-sorted equational theory \eqref{trad_theory} over the underlying classical $\calS$-sorted signature $|\Sigma|$ \eqref{ord_sig}. More precisely: 
\begin{enumerate}[leftmargin=*]
\item Given a $\V$-enriched $\calS$-sorted equational theory $\calT = (\Sigma, \calE)$, there is a classical $\calS$-sorted equational theory $|\calT| = \left(|\Sigma|, |\calE|\right)$ over $|\Sigma|$ such that a $\Sigma$-algebra $A$ is a $\calT$-algebra iff the underlying $|\Sigma|$-algebra $|A|$ \eqref{ord_sig_para} is a $|\calT|$-algebra. 

\item Conversely, given a $\V$-enriched $\calS$-sorted signature $\Sigma$ and a classical $\calS$-sorted equational theory $\calT = (|\Sigma|, \calE)$ over $|\Sigma|$, there is a $\V$-enriched $\calS$-sorted equational theory $\widehat{\calT} = \left(\Sigma, \widehat{\calE}\right)$ over $\Sigma$ such that a $\Sigma$-algebra $A$ is a $\widehat{\calT}$-algebra iff the underlying $|\Sigma|$-algebra $|A|$ \eqref{ord_sig_para} is a $\calT$-algebra.
\end{enumerate}
\end{theo_sub}

\begin{proof}
To prove (1), for each algebraic $\Sigma$-equation $\sfe \in \calE$, we have by (1) and (2) of Proposition \ref{equiv_eqn_prop} a set $\left|\calE_\sfe\right|$ of syntactic $|\Sigma|$-equations in context with the property that a $\Sigma$-algebra $A$ satisfies $\sfe$ iff the underlying $|\Sigma|$-algebra $|A|$ satisfies each element of $\left|\calE_\sfe\right|$. We then set $|\calE| := \bigcup_{\sfe \in \calE} \left|\calE_\sfe\right|$ and $|\calT| := \left(|\Sigma|, |\calE|\right)$, which proves (1). 

To prove (2), for each syntactic $|\Sigma|$-equation $\sfe \in \calE$, there is by Proposition \ref{equiv_eqn_prop}(3) an (ordinary) algebraic $\Sigma$-equation $\widehat{\sfe}$ with the property that a $\Sigma$-algebra $A$ satisfies $\widehat{\sfe}$ iff the underlying $|\Sigma|$-algebra $|A|$ satisfies $\sfe$. We then set $\widehat{\calE} := \left\{\widehat{\sfe} \mid \sfe \in \calE\right\}$ and $\widehat{\calT} := \left(\Sigma, \widehat{\calE}\right)$, which proves (2). 
\end{proof}

\begin{defn_sub}
\label{underlying_trad_theory}
Let $\calT = (\Sigma, \calE)$ be a $\V$-enriched $\calS$-sorted equational theory. We call the classical $\calS$-sorted equational theory $|\calT| = (|\Sigma|, |\calE|)$ of Theorem \ref{equiv_theory_prop}(1) the \textbf{underlying classical $\calS$-sorted equational theory (of $\calT$)}. 
\end{defn_sub}

\noindent For the remainder of \S\ref{theories_section}, we now focus on establishing explicit descriptions of free algebras for $\V$-enriched multi-sorted equational theories, i.e.~of the left adjoint to $U^\calT : \calT\Alg \to \V^\calS$ for a $\V$-enriched $\calS$-sorted equational theory $\calT$ (in Theorems \ref{taut_thm}, \ref{explicit_free_fibre_prop_2}, \ref{cc_free_alg_thm}, and \ref{free_alg_cc}).

\begin{para_sub}
\label{ord_theory_para}
Let $\calT = (\Sigma, \calE)$ be a $\V$-enriched $\calS$-sorted equational theory, with underlying classical $\calS$-sorted equational theory $|\calT| = (|\Sigma|, |\calE|)$ \eqref{underlying_trad_theory}. In view of Theorem \ref{equiv_theory_prop}(1), the faithful functor $|-|^\Sigma : \Sigma\Alg \to |\Sigma|\Alg$ of \ref{ord_sig_para} restricts to a faithful functor $|-|^\calT : \calT\Alg \to |\calT|\Alg$, yielding the following commutative diagram:
\begin{equation}
\label{theory_square}
\begin{tikzcd}
	\calT\Alg && \Sigma\Alg && {\V^\calS} \\
	\\
	{|\calT|\Alg} && {|\Sigma|\Alg} && {\Set^\calS}.
	\arrow[hook, from=1-1, to=1-3]
	\arrow["{U^\Sigma}", from=1-3, to=1-5]
	\arrow["{|-|^\calT}"', from=1-1, to=3-1]
	\arrow["{|-|^\Sigma}"', from=1-3, to=3-3]
	\arrow["{|-|^\calS}", from=1-5, to=3-5]
	\arrow[hook, from=3-1, to=3-3]
	\arrow["{U^{|\Sigma|}}"', from=3-3, to=3-5]
\end{tikzcd}
\end{equation}
Note that the top composite in \eqref{theory_square} is precisely the forgetful functor $U^\calT : \calT\Alg \to \V^\calS$ of Definition \ref{eqn_theory}, while the bottom composite is the forgetful functor $U^{|\calT|} : |\calT|\Alg \to \Set^\calS$. 
\end{para_sub}

\begin{prop_sub}
\label{taut_lift_hyp_2}
Let $\calT = (\Sigma, \calE)$ be a $\V$-enriched $\calS$-sorted equational theory with underlying classical $\calS$-sorted equational theory $|\calT| = (|\Sigma|, |\calE|)$ \eqref{underlying_trad_theory}. Then the lefthand square of the commutative diagram \eqref{theory_square} is a pullback (in $\mathsf{CAT}$), the faithful functor $|-|^\calT : \calT\Alg \to |\calT|\Alg$ is topological, and the forgetful functor $U^\calT : \calT\Alg \to \V^\calS$ sends $|-|^\calT$-initial sources to $|-|^\calS$-initial sources \eqref{top_para}.
\end{prop_sub}

\begin{proof}
The first assertion follows immediately from Theorem \ref{equiv_theory_prop}(1). Since $|-|^\Sigma : \Sigma\Alg \to |\Sigma|\Alg$ is topological by Proposition \ref{taut_lift_hyp} and the lower left edge of \eqref{theory_square} is fully faithful, it readily follows that $|-|^\calT : \calT\Alg \to |\calT|\Alg$ is topological, and that the upper left edge sends $|-|^\calT$-initial sources to $|-|^\Sigma$-initial sources. Since $U^\Sigma : \Sigma\Alg \to \V^\calS$ sends $|-|^\Sigma$-initial sources to $|-|^\calS$-initial sources by Proposition \ref{taut_lift_hyp}, it follows that $U^\calT : \calT\Alg \to \V^\calS$ sends $|-|^\calT$-initial sources to $|-|^\calS$-initial sources. 
\end{proof}

\subsection{Free algebras of classical multi-sorted equational theories}
\label{free_theory_alg_Set}

Proposition \ref{taut_lift_hyp_2} shows that the outer commutative rectangle in \eqref{theory_square} satisfies two of the three central assumptions of Wyler's taut lift theorem \cite[Theorem 21.28]{AHS}, the third assumption being that the functor $U^{|\calT|} : |\calT|\Alg \to \Set^\calS$ has a left adjoint. Given an arbitrary classical $\calS$-sorted equational theory $\calT = (\Sigma, \calE)$ (not necessarily of the form $\left|\calT'\right|$ for a $\V$-enriched $\calS$-sorted equational theory $\calT'$), in this subsection we recall the well-known explicit description of the left adjoint $F^\calT : \Set^\calS \to \calT\Alg$ for $U^\calT : \calT\Alg \to \Set^\calS$.

\begin{defn_sub}
\label{congruence}
Let $\Sigma$ be a classical $\calS$-sorted signature \eqref{signature}, and let $A$ be a $\Sigma$-algebra. A \textbf{$\Sigma$-congruence} on $A$ is an $\calS$-sorted family $\left(\sim_S\right)_{S \in \calS}$ with each $\sim_S$ ($S \in \calS$) an equivalence relation on the carrier set $A_S$, such that for each $\sigma \in \Sigma$ with input sorts $(S_1, \ldots, S_n)$ and output sort $S$ and all $a_i, b_i \in A_{S_i}$ ($1 \leq i \leq n$), we have
\[ a_i \sim_{S_i} b_i \ \ \forall 1 \leq i \leq n \ \ \Longrightarrow \ \ \sigma^A(a_1, \ldots, a_n) \sim_S \sigma^A(b_1, \ldots, b_n). \]
For each $S \in \calS$, we write $A_S/{\sim_S}$ for the set of equivalence classes of $\sim_S$, and for each $a \in A_S$, we write $[a]_S$ or just $[a]$ for its equivalence class. We write $q_S : A_S \to A_S/{\sim_S}$ for the surjective function given by $a \mapsto [a]$ $(a \in A_S)$.  
\end{defn_sub}

\begin{para_sub}[\textbf{The $\Sigma$-algebra determined by a $\Sigma$-congruence}]
\label{quot_alg_para}
Let $\Sigma$ be a classical $\calS$-sorted signature \eqref{signature}, let $A$ be a $\Sigma$-algebra, and let $\sim \ = \congr$ be a $\Sigma$-congruence on $A$. Then the $\calS$-sorted set \[ A/{\sim} := \left(A_S\middle/{\sim_S}\right)_{S \in \calS} \] can be equipped with a $\Sigma$-algebra structure \[ A/{\sim}\] so that the $\calS$-sorted function \[ q := \left(q_S\right)_{S \in \calS} : A =  \left(A_S\right)_{S \in \calS} \to \left(A_S/{\sim_S}\right)_{S \in \calS} = A/{\sim} \] becomes a morphism of $\Sigma$-algebras \[ q : A \to A/{\sim}. \] Specifically, let $\sigma \in \Sigma$ have input sorts $(S_1, \ldots, S_n)$ and output sort $S$, and let us write $\sim_i \ := \ \sim_{S_i}$ for each $1 \leq i \leq n$. We define 
\[ \sigma^{A/{\sim}} : \left(A_{S_1}/{\sim_1}\right) \times \ldots \times \left(A_{S_n}/{\sim_n}\right) \to A_S/{\sim_S} \] 
\[ \left([a_1], \ldots, [a_n]\right) \mapsto \left[\sigma^A(a_1, \ldots, a_n)\right], \] which is well-defined because $\sim$ is a $\Sigma$-congruence on $A$.
\end{para_sub}

\begin{defn_sub}
\label{cong_gen}
Let $\calT = (\Sigma, \calE)$ be a classical $\calS$-sorted equational theory \eqref{trad_theory}, and let $A$ be a $\Sigma$-algebra. \textbf{The $\Sigma$-congruence on $A$ generated by $\calE$} is the smallest $\Sigma$-congruence $\sim^\calE \ = \left(\sim^\calE_S\right)_{S \in \calS}$ on $A$ with the property that for each syntactic $\Sigma$-equation $[\vbar \vdash s \doteq t : S]$ of sort $S \in \calS$ in context $\vbar \equiv v_1 : S_1, \ldots, v_n : S_n$ in $\calE$ and all $a_i \in A_{S_i}$ ($1 \leq i \leq n$), we have \[ s^A(a_1, \ldots, a_n) \ \sim_S^\calE \ t^A(a_1, \ldots, a_n), \] where $s^A, t^A : A_{S_1} \times \ldots \times A_{S_n} \rightrightarrows A_S$ are the interpretations of $[\vbar \vdash s : S]$ and $[\vbar \vdash t : S]$ in $A$ \eqref{syntax_para}. It is then straightforward to verify that the resulting $\Sigma$-algebra $A/{\sim^\calE}$ \eqref{quot_alg_para} is a $\calT$-algebra.
\end{defn_sub}

\begin{para_sub}[\textbf{Explicit description of free algebras for classical multi-sorted equational theories}]
\label{init_alg_theory_para}
Let $\calT = (\Sigma, \calE)$ be a classical $\calS$-sorted equational theory \eqref{trad_theory}. Let $F^\Sigma : \Set^\calS \to \Sigma\Alg$ be the left adjoint of $U^\Sigma : \Sigma\Alg \to \Set^\calS$, whose explicit description we recalled in \ref{free_sig_alg_Set}. It is (well) known that\footnote{We shall also recover this result from the $\V = \Set$ case of Theorem \ref{cc_free_alg_thm} below.} for each $\calS$-sorted set $X = \left(X_S\right)_{S \in \calS}$, the free $\calT$-algebra $F^\calT X$ on $X$ is the $\calT$-algebra $\left.F^\Sigma X\middle/{\sim^\calE}\right.$, where $\sim^\calE$ is the $\Sigma$-congruence on $F^\Sigma X$ generated by $\calE$ \eqref{cong_gen}. So for each sort $S \in \calS$, the carrier set of $F^\calT X$ at $S$ is the set $\left.\Term_\Sigma(X)_S\middle/{\sim^\calE_S}\right.$ of equivalence classes (with respect to $\sim^\calE_S$) of ground terms of sort $S$ with constants from $X$. For each $\sigma \in \Sigma$ with input sorts $(S_1, \ldots, S_n)$ and output sort $S$ and all $\left.[t_i] \in \Term_{\Sigma}(X)_{S_i}\middle/{\sim^\calE_{S_i}}\right.$ ($1 \leq i \leq n$), we have
\[ \left. \sigma^{F^\calT X}([t_1], \ldots, [t_n]) = [\sigma(t_1, \ldots, t_n)] \in \Term_{\Sigma}(X)_{S}\middle/{\sim^\calE_{S}}. \right. \]
The canonical $\calS$-sorted function $\left.\eta = \eta_X^\calT : X \to U^\calT F^\calT X = \Term_{\Sigma}(X)\middle/{\sim^\calE}\right.$ is given at each sort $S \in \calS$ by the composite function
\[ \left.X_S \hookrightarrow \Term_\Sigma(X)_S \xrightarrow{q_S} \Term_\Sigma(X)_S\middle/{\sim^\calE_S}.\right. \]
\end{para_sub}

\subsection{Free algebras of enriched multi-sorted equational theories: the general case}
\label{init_alg_theory_gen_sub}  

Let $\calT = (\Sigma, \calE)$ be a $\V$-enriched $\calS$-sorted equational theory. For each $\calS$-sorted object $X$ of $\V$, we write $X \downarrow U^\calT$ for the family consisting of all pairs $\left(A, f : X \to U^\calT A\right)$ consisting of a $\calT$-algebra $A$ and a morphism $f : X \to U^\calT A = A$ of $\V^\calS$. Given such a pair $\left(A, f : X \to U^\calT A\right)$, we write \[ f^\sharp : F^{|\calT|} |X|^\calS \to |A|^\calT \] for the unique $|\calT|$-algebra morphism corresponding to the $\calS$-sorted function $f : |X|^\calS \to \left|U^\calT A\right|^\calS = U^{|\calT|}|A|^{\calT}$ under the adjunction $F^{|\calT|} \dashv U^{|\calT|} : |\calT|\Alg \to \Set^\calS$ of \ref{init_alg_theory_para}. Note that by \ref{init_alg_theory_para}, we may also regard $f^\sharp$ as (just) an $\calS$-sorted function $\left.f^\sharp : \Term_{|\Sigma|}(|X|)\middle/{\sim^{|\calE|}} \to |A|^\calS\right.$, where $\sim^{|\calE|}$ is the $|\Sigma|$-congruence on the $|\Sigma|$-algebra $F^{|\Sigma|}(|X|)$ generated by $|\calE|$ \eqref{cong_gen}.      

From Proposition \ref{taut_lift_hyp_2}, \ref{init_alg_theory_para}, and Wyler's taut lift theorem \cite[Theorem 21.28]{AHS}, we now obtain the following result, which (like Proposition \ref{taut_sig_thm}) was previously only established in the special case where $\calT$ is ordinary and single-sorted (see e.g.~\cite[Theorem 2.3]{Porst_existence}).

\begin{theo_sub}
\label{taut_thm}
Let $\calT = (\Sigma, \calE)$ be a $\V$-enriched $\calS$-sorted equational theory with underlying classical $\calS$-sorted equational theory $|\calT| = (|\Sigma|, |\calE|)$ \eqref{underlying_trad_theory}. Then $U^\calT : \calT\Alg \to \V^\calS$ has a left adjoint $F^\calT : \V^\calS \to \calT\Alg$, and the resulting adjunction $F^\calT \dashv U^\calT : \calT\Alg \to \V^\calS$ is a lifting of the adjunction $F^{|\calT|} \dashv U^{|\calT|} : |\calT|\Alg \to \Set^\calS$. In particular, the following diagram strictly commutes:
\begin{equation}
\label{theory_square_free}
\begin{tikzcd}
	\V^\calS && {\calT\Alg} \\
	\\
	{\Set^\calS} && {|\calT|\Alg}.
	\arrow["{F^\calT}", from=1-1, to=1-3]
	\arrow["{|-|^\calT}", from=1-3, to=3-3]
	\arrow["{|-|^\calS}"', from=1-1, to=3-1]
	\arrow["{F^{|\calT|}}"', from=3-1, to=3-3]
\end{tikzcd}
\end{equation}
For each $\calS$-sorted object $X$ of $\V$, the free $\calT$-algebra $F^\calT X$ on $X$ is the domain of the initial lift of the $|-|^\calT$-structured source
\begin{equation}\label{free_th_eqn} 
\left(f^\sharp : F^{|\calT|}|X| \to |A|^\calT\right)_{(A, f) \in X \downarrow U^\calT}.
\end{equation}
In particular, the carrier object $U^\calT F^\calT X$ of $\V^\calS$ is the domain of the initial lift of the $|-|^\calS$-structured source
\begin{equation}\label{free_th_eqn_2} 
\left.\left(f^\sharp : \Term_{|\Sigma|}(|X|)\middle/{\sim^{|\calE|}} \to |A|^\calS\right)_{(A, f) \in X \downarrow U^\calT}, \right.
\end{equation}
where $\sim^{|\calE|}$ is the $|\Sigma|$-congruence on the $|\Sigma|$-algebra $F^{|\Sigma|}(|X|)$ generated by $|\calE|$ \eqref{cong_gen}. 
\end{theo_sub}

\begin{para_sub}
\label{free_th_para}
In the setting of Theorem \ref{taut_thm}, note that the free $\calT$-algebra $F^\calT X$ on an object $X$ of $\V^\calS$ lies in the fibre of $\calT\Alg$ over the free $|\calT|$-algebra $F^{|\calT|}(|X|)$ on the underlying object $|X|$ of $\Set^\calS$. Recalling the explicit description of $F^{|\calT|}(|X|)$ from \ref{init_alg_theory_para}, this means that the carrier object $U^\calT F^\calT X$ lies in the fibre of $\V^\calS$ over the $\calS$-sorted set $\left.\Term_{|\Sigma|}(|X|)\middle/{\sim^{|\calE|}}\right.$ of equivalence classes of ground $|\Sigma|$-terms with constants from $|X|$ modulo the $|\Sigma|$-congruence $\sim^{|\calE|}$ generated by $|\calE|$ \eqref{cong_gen}. Moreover, in view of \ref{ord_sig_para}, the following hold:
\begin{enumerate}[leftmargin=*]
\item The canonical $\calS$-sorted function 
\[ \left.\eta = \eta_{|X|}^{|\calT|} : |X| \to \Term_{|\Sigma|}(|X|)\middle/{\sim^{|\calE|}} = \left|U^\calT F^\calT X\right|\right. \] of \ref{init_alg_theory_para} is a $\V^\calS$-admissible morphism $\eta = \eta_X^\calT : X \to U^\calT F^\calT X$.  
\item For each $\sigma \in \Sigma$ of type $((S_1, \ldots, S_n), S, P)$, the function
\[ \left.\widetilde{\sigma} : |P| \times \left(\Term_{|\Sigma|}(|X|)_{S_1}\middle/{\sim^{|\calE|}_{S_1}}\right) \times \ldots \times \left(\Term_{|\Sigma|}(|X|)_{S_n}\middle/{\sim^{|\calE|}_{S_n}}\right) \to \Term_{|\Sigma|}(|X|)_{S}\middle/{\sim^{|\calE|}_{S}}\right. \]
\[ (p, [t_1], \ldots, [t_n]) \mapsto \left[\sigma_p(t_1, \ldots, t_n)\right], \] 
which may be equivalently written as
\[ \widetilde{\sigma} : \left|P \tensor \left(\left(U^\calT F^\calT X\right)_{S_1} \times \ldots \times \left(U^\calT F^\calT X\right)_{S_n} \right)\right| = |P| \times \left|\left(U^\calT F^\calT X\right)_{S_1} \right| \times \ldots \times \left|\left(U^\calT F^\calT X\right)_{S_n} \right| \] \[ \to \left|\left(U^\calT F^\calT X\right)_{S} \right|, \] is a $\V$-admissible morphism \[ \widetilde{\sigma} : P \tensor \left(\left(U^\calT F^\calT X\right)_{S_1}  \times \ldots \times \left(U^\calT F^\calT X\right)_{S_n} \right) \to \left(U^\calT F^\calT X\right)_{S}, \] and $\sigma^{F^\calT X} = \widetilde{\sigma}$. 
\end{enumerate}
So $F^\calT X$ is completely determined by its carrier object $U^\calT F^\calT X$. While Theorem \ref{taut_thm} provides an explicit description of this carrier object, we shall now endeavour to provide even more explicit and concrete descriptions of this carrier object (in Theorems \ref{explicit_free_fibre_prop_2} and \ref{free_alg_cc}). 
\end{para_sub}

\begin{defn_sub}
\label{T_compatible}
Let $\calT = (\Sigma, \calE)$ be a $\V$-enriched $\calS$-sorted equational theory with underlying classical $\calS$-sorted equational theory $|\calT| = (|\Sigma|, |\calE|)$ \eqref{underlying_trad_theory}, and let $X$ be an $\calS$-sorted object of $\V$. Given an $\calS$-sorted object $B$ of $\V$ that lies in the fibre over the $\calS$-sorted set $\left.\Term_{|\Sigma|}(|X|)\middle/{\sim^{|\calE|}}\right.$ of congruence classes of ground $|\Sigma|$-terms with constants from $|X|$ modulo the $|\Sigma|$-congruence $\sim^{|\calE|}$ \eqref{init_alg_theory_para}, we say that $B$ is \textbf{$\calT$-compatible with $X$} if the following conditions are satisfied:
\begin{enumerate}[leftmargin=*]
\item The $\calS$-sorted unit function 
\[ \left.\eta = \eta_{|X|} : |X| \to \Term_{|\Sigma|}(|X|)/{\sim^{|\calE|}} = |B|\right. \] is a $\V^\calS$-admissible morphism $\eta : X \to B$.  
\item For each $\sigma \in \Sigma$ of type $((S_1, \ldots, S_n), S, P)$, the function
\[ \left.\widetilde{\sigma} : |P| \times \left(\Term_{|\Sigma|}(|X|)_{S_1}\middle/{\sim^{|\calE|}_{S_1}}\right) \times \ldots \times \left(\Term_{|\Sigma|}(|X|)_{S_n}\middle/{\sim^{|\calE|}_{S_n}}\right) \to \Term_{|\Sigma|}(|X|)_{S}\middle/{\sim^{|\calE|}_{S}}\right. \]
\[ (p, [t_1], \ldots, [t_n]) \mapsto \left[\sigma_p(t_1, \ldots, t_n)\right], \] 
which may be equivalently written as
\[ \widetilde{\sigma} : \left|P \tensor \left(B_{S_1} \times \ldots \times B_{S_n} \right)\right| = |P| \times \left|B_{S_1} \right| \times \ldots \times \left|B_{S_n} \right| \to \left|B_{S} \right|, \] is a $\V$-admissible morphism \[ \widetilde{\sigma} : P \tensor \left(B_{S_1}  \times \ldots \times B_{S_n} \right) \to B_{S}. \] 
\end{enumerate} 
If $B$ is $\calT$-compatible with $X$, then $B$ can be equipped with the structure of a $\Sigma$-algebra, where $\sigma^B = \widetilde{\sigma}$ for each $\sigma \in \Sigma$, and $B$ is moreover a $\calT$-algebra by Proposition \ref{equiv_eqn_prop}(1), because the underlying $|\Sigma|$-algebra of $B$ is the free $|\calT|$-algebra $F^{|\calT|}(|X|)$.
\end{defn_sub}

We now have the following slightly more explicit description (compared to that given in Theorem \ref{taut_thm}) of the left adjoint $F^\calT : \V^\calS \to \calT\Alg$ for a $\V$-enriched $\calS$-sorted equational theory $\calT$, whose proof is virtually identical to that of Theorem \ref{explicit_free_fibre_prop}. The special case where $\calT$ is ordinary and single-sorted appears on \cite[Page 439]{Porst_cc}, and the special case where $\calT$ is a $\V$-enriched single-sorted equational theory with $\V = \Top$ appears as \cite[Lemma 4.4]{Battenfeld_comp}.

\begin{theo_sub}
\label{explicit_free_fibre_prop_2}
Let $\calT = (\Sigma, \calE)$ be a $\V$-enriched $\calS$-sorted equational theory with underlying classical $\calS$-sorted equational theory $|\calT| = (|\Sigma|, |\calE|)$ \eqref{underlying_trad_theory}, and let $X$ be an $\calS$-sorted object of $\V$. Then the $\calS$-sorted carrier object $U^\calT F^\calT X \in \ob\V^\calS$ of the free $\calT$-algebra $F^\calT X$ on $X$ is the infimum in the fibre over the $\calS$-sorted set $\left.U^{|\calT|} F^{|\calT|} |X| = \Term_{|\Sigma|}(|X|)\middle/{\sim^{|\calE|}}\right.$ \eqref{init_alg_theory_para} of all the elements that are $\calT$-compatible with $X$ \eqref{T_compatible}:
\[ \left.U^\calT F^\calT X = \bigwedge\left\{B \in \Fib\left(\Term_{|\Sigma|}(|X|)\middle/{\sim^{|\calE|}}\right) \mid B \text{ is } \calT\text{-compatible with } X\right\}.\right. \]   
\end{theo_sub}

\subsection{Free algebras of enriched multi-sorted equational theories: the case where $\V$ is cartesian closed}
\label{cong_section}

Analogous to the discussion at the beginning of \S\ref{sig_alg_subsection_cc}, given a $\V$-enriched $\calS$-sorted equational theory $\calT$ and an $\calS$-sorted object $X$ of $\V$, the descriptions of the free $\calT$-algebra $F^\calT X$ on $X$ provided by Theorems \ref{taut_thm} and \ref{explicit_free_fibre_prop_2} are not as concrete or intrinsic as one might hope, since they ultimately refer not only to $X$ and $\calT$, but also to the entire category $\calT\Alg$. Under the assumption (which we make throughout this subsection) that $\V$ is equipped with the \emph{cartesian} symmetric monoidal structure and is moreover cartesian closed, we shall provide a more concrete and intrinsic description of $F^\calT X$ in Theorems \ref{cc_free_alg_thm} and \ref{free_alg_cc} below. This subsection essentially generalizes \cite[\S 3.3]{Battenfeld_comp}, which treats the special case where $\calT$ is single-sorted and $\V$ is the cartesian closed category $\CGTop$ of compactly generated topological spaces (see Examples \ref{examples_ex} and \ref{topc_ex}). 

\begin{defn_sub}
\label{congruence_2}
Let $\Sigma$ be a $\V$-enriched $\calS$-sorted signature with underlying classical $\calS$-sorted signature $|\Sigma|$ \eqref{ord_sig}, and let $A$ be a $\Sigma$-algebra. A \textbf{$\Sigma$-congruence on $A$} is a $|\Sigma|$-congruence \eqref{congruence} on the underlying $|\Sigma|$-algebra $|A|$ \eqref{ord_sig_para}. 

Explicitly, a $\Sigma$-congruence on $A$ is an $\calS$-sorted family $\left(\sim_S\right)_{S \in \calS}$ with each $\sim_S$ ($S \in \calS$) an equivalence relation on the set $\left|A_S\right|$, such that for each $\sigma \in \Sigma$ of type $((S_1, \ldots, S_n), S, P)$, each $p \in |P|$, and all $a_i, b_i \in \left|A_{S_i}\right|$ ($1 \leq i \leq n$), we have
\[ a_i \sim_{S_i} b_i \ \ \forall 1 \leq i \leq n \ \ \Longrightarrow \ \ \sigma^A(p, a_1, \ldots, a_n) \ \sim_S \ \sigma^A(p, b_1, \ldots, b_n). \]    
\end{defn_sub}

\begin{defn_sub}
\label{quot_obj}
Let $\Sigma$ be a $\V$-enriched $\calS$-sorted signature, let $A$ be a $\Sigma$-algebra, and let $\sim \ = \congr$ be a $\Sigma$-congruence on $A$. For each $S \in \calS$ we have the surjective function $q_S : |A_S| \to |A_S|/{\sim_S}$ given by $a \mapsto [a]$ ($a \in A_S$). We define the object \[ A_S/{\sim_S} \] of $\V$ to be the set $|A_S|/{\sim_S}$ equipped with the final structure induced by the surjective function $q_S : |A_S| \to |A_S|/{\sim_S}$, so that \[ q_S : A_S \to A_S/{\sim_S} \] is a quotient morphism of $\V$ \eqref{top_para}. We thus have an object 
\begin{equation}
\label{S_sorted_quot}
A/{\sim} := \left(A_S/{\sim_S}\right)_{S \in \calS}
\end{equation} 
of $\V^\calS$ equipped with a quotient morphism 
\begin{equation}
\label{S_sorted_quot_mor}
q := \left(q_S\right)_{S \in \calS} : A \to A/{\sim}.
\end{equation}    
\end{defn_sub}

\begin{lem_sub}[\textbf{The $\Sigma$-algebra determined by a $\Sigma$-congruence}]
\label{quot_alg_lem}
Let $\Sigma$ be a $\V$-enriched $\calS$-sorted signature, let $A$ be a $\Sigma$-algebra, and let $\sim \ = \congr$ be a $\Sigma$-congruence on $A$. Then the object $A/{\sim}$ of $\V^\calS$ \eqref{S_sorted_quot} can be equipped with a $\Sigma$-algebra structure 
\begin{equation}
\label{quot_alg_eq}
A/{\sim}
\end{equation}
so that the quotient morphism $q : A \to A/{\sim}$ \eqref{S_sorted_quot_mor} becomes a morphism of $\Sigma$-algebras $q : A \to A/{\sim}$. The underlying $|\Sigma|$-algebra of $A/{\sim}$ is $|A|/{\sim}$ \eqref{quot_alg_para}. 
\end{lem_sub}

\begin{proof}
Let $\sigma \in \Sigma$ have type $((S_1, \ldots, S_n), S, P)$, and let us write $\sim_i \ := \ \sim_{S_i}$ and $q_i := q_{S_i}$ for each $1 \leq i \leq n$. We define a $\V$-morphism \[ \sigma^{A/{\sim}} : P \times \left(A_{S_1}/{\sim_1}\right) \times \ldots \times \left(A_{S_n}/{\sim_n}\right) \to A_S/{\sim_S} \] (recall that $\tensor = \times$) as follows. For all $p \in |P|$ and $a_i \in \left|A_{S_i}\right|$ ($1 \leq i \leq n$), we define
\begin{equation}
\label{quot_sigma} 
\sigma^{A/{\sim}}\left(p, [a_1], \ldots, [a_n]\right) := \left[\sigma^A(p, a_1, \ldots, a_n)\right],
\end{equation}
which is well-defined because $\sim$ is a $\Sigma$-congruence on $A$. We must now show that the function \[ \sigma^{A/{\sim}} : |P| \times \left|A_{S_1}/{\sim_1}\right| \times \ldots \times \left|A_{S_n}/{\sim_n}\right| \to \left|A_S/{\sim_S}\right| \] defined by \eqref{quot_sigma} is $\V$-admissible. Because $\V$ is cartesian closed, the $\V$-morphism \[ 1_P \times q_1 \times \ldots \times q_n : P \times A_{S_1} \times \ldots \times A_{S_n} \to P \times \left(A_{S_1}/{\sim_1}\right) \times \ldots \times \left(A_{S_n}/{\sim_n}\right) \] is a quotient morphism. So it suffices to show that the composite function \[ \sigma^{A/{\sim}} \circ \left(1_P \times q_1 \times \ldots \times q_n\right) : \left|P \times A_{S_1} \times \ldots \times A_{S_n}\right| \to \left|A_S\right|/{\sim_S} \] is $\V$-admissible; but this is true because it is equal (by the definition \eqref{quot_sigma} of $\sigma^{A/{\sim}}$) to the $\V$-admissible function $q_S \circ \sigma^A$. So we have a well-defined $\Sigma$-algebra $A/{\sim}$, and it immediately follows from \eqref{quot_sigma} that the quotient morphism $q : A \to A/{\sim}$ becomes a morphism of $\Sigma$-algebras $q : A \to A/{\sim}$. The final assertion is immediate from the relevant definitions.          
\end{proof}

\begin{defn_sub}
\label{kernel_cong}
Let $\Sigma$ be a $\V$-enriched $\calS$-sorted signature, and let $f : A \to B$ be a morphism of $\Sigma$-algebras. The \textbf{kernel $\Sigma$-congruence of $f$} is the $\Sigma$-congruence \[ \sim^f \ = \left(\sim^f_S\right)_{S \in \calS} \] on $A$ defined as follows: for all $S \in \calS$ and $a, a' \in \left|A_S\right|$, we have $a \sim_S^f a'$ if $f_S(a) = f_S\left(a'\right) \in \left|B_S\right|$. It is immediate that $\sim^f$ is indeed a $\Sigma$-congruence on $A$.  
\end{defn_sub}

\begin{lem_sub}
\label{kernel_lem}
Let $\Sigma$ be a $\V$-enriched $\calS$-sorted signature, let $f : A \to B$ be a morphism of $\Sigma$-algebras, and let $\sim$ be any $\Sigma$-congruence on $A$ that is contained in the kernel $\Sigma$-congruence $\sim^f$ of $f$ (i.e.~$\sim_S \ \subseteq \ \sim^f_S$ for each $S \in \calS$). Then $f : A \to B$ factors in $\Sigma\Alg$ as
\begin{equation}
\label{kern_fact_eq}
\left.f = \left(A \xrightarrow{q} \left(A\middle/{\sim}\right) \xrightarrow{f^*} B\right)\right.
\end{equation}
for a unique $\Sigma$-algebra morphism $\left.f^* : \left(A\middle/{\sim}\right) \to B\right.$. 
\end{lem_sub}

\begin{proof}
For each $S \in \calS$, we define the function \[ \left.f^*_S : \left|A_S\middle/{\sim_S}\right| = \left(\left|A_S\right|\middle/{\sim_S}\right) \to \left|B_S\right|\right.\]
\[ [a] \mapsto f_S(a), \] which is well-defined by the definition of $\sim^f$ and the assumption that $\sim$ is contained in $\sim^f$. The function $f_S^*$ is $\V$-admissible because $f_S^* \circ q_S = f_S$ and $\left.q_S : A_S \to A_S\middle/{\sim_S}\right.$ is a quotient morphism and $f_S : \left|A_S\right| \to \left|B_S\right|$ is $\V$-admissible. Then $f^*$ is a $\Sigma$-algebra morphism because $f$ is a $\Sigma$-algebra morphism, and we clearly have $f^* \circ q = f$. The uniqueness of $f^*$ follows from the surjectivity of each $q_S$ ($S \in \calS$).      
\end{proof}

\begin{defn_sub}
\label{cong_gen_2}
Let $\calT = (\Sigma, \calE)$ be a $\V$-enriched $\calS$-sorted equational theory with underlying classical $\calS$-sorted equational theory $|\calT| = \left(|\Sigma|, |\calE|\right)$ \eqref{underlying_trad_theory}, and let $A$ be a $\Sigma$-algebra. The \textbf{$\Sigma$-congruence $\sim^\calE$ on $A$ generated by $\calE$} is the $|\Sigma|$-congruence $\sim^{|\calE|}$ on the underlying $|\Sigma|$-algebra $|A|$ generated by $|\calE|$ \eqref{cong_gen}. We then have the $\Sigma$-algebra $\left.A\middle/{\sim^\calE}\right.$ \eqref{quot_alg_eq}, whose underlying $|\Sigma|$-algebra is $\left.|A|\middle/{\sim^{|\calE|}}\right.$ \eqref{quot_alg_lem}.     
\end{defn_sub}

\begin{lem_sub}
\label{cong_gen_lem}
Let $\calT = (\Sigma, \calE)$ be a $\V$-enriched $\calS$-sorted equational theory, and let $A$ be a $\Sigma$-algebra. Then the $\Sigma$-algebra $\left.A\middle/{\sim^\calE}\right.$ \eqref{quot_alg_eq} is a $\calT$-algebra, and for any $\Sigma$-algebra morphism $f : A \to B$ such that $B$ is a $\calT$-algebra, $\sim^\calE$ is contained in the kernel $\Sigma$-congruence $\sim^f$ of $f$.  
\end{lem_sub}

\begin{proof}
The underlying $|\Sigma|$-algebra of $\left.A\middle/{\sim^\calE}\right.$ is $\left.|A|\middle/{\sim^{|\calE|}}\right.$, which is a $|\calT|$-algebra by \ref{cong_gen}, so that $\left.A\middle/{\sim^\calE}\right.$ is a $\calT$-algebra by Theorem \ref{equiv_theory_prop}(1). Now let $f : A \to B$ be a morphism of $\Sigma$-algebras to a $\calT$-algebra $B$. By the definition of $\sim^\calE$ and the fact that $\sim^f$ is a $\Sigma$-congruence, to show that $\sim^\calE \ \subseteq \ \sim^f$ it suffices to show for each syntactic $|\Sigma|$-equation $[\vbar \vdash s \doteq t : S]$ of sort $S$ in context $\vbar \equiv v_1 : S_1, \ldots, v_n : S_n$ in $|\calE|$ that $s^{A}(a_1, \ldots, a_n) \sim_S^f t^{A}(a_1, \ldots, a_n)$, i.e.~that $f_S\left(s^{A}(a_1, \ldots, a_n)\right) = f_S\left(t^{A}(a_1, \ldots, a_n)\right)$, for all $a_i \in \left|A_{S_i}\right|$ ($1 \leq i \leq n$), where we write $s^{A} := \left[\vbar \vdash s : S\right]^{A}$ \eqref{disc_to_disc_2}, and similarly for $t$ and $B$. Indeed, by \eqref{synt_to_nat_square} and the assumption that $B$ is a $\calT$-algebra, we have   
\[ f_S\left(s^{A}(a_1, \ldots, a_n)\right) = s^{B}\left(f_{S_1}(a_1), \ldots, f_{S_n}(a_n)\right) = t^{B}\left(f_{S_1}(a_1), \ldots, f_{S_n}(a_n)\right) = f_S\left(t^{A}(a_1, \ldots, a_n)\right), \] as desired.           
\end{proof}

\noindent We now have the following theorem, which generalizes \cite[Theorem 3.13]{Battenfeld_comp}, where $\V = \CGTop$ and $\calT$ is single-sorted, as well as \cite[Proposition 1.5]{Porst_cc}, where $\V$ is an arbitrary cartesian closed topological category over $\Set$ (as in this subsection), but $\calT$ is ordinary and single-sorted. 

\begin{theo_sub}
\label{cc_free_alg_thm}
Let $\calT = (\Sigma, \calE)$ be a $\V$-enriched $\calS$-sorted equational theory, and let $F^\Sigma : \V^\calS \to \Sigma\Alg$ be the left adjoint of $U^\Sigma : \Sigma\Alg \to \V^\calS$ \eqref{taut_sig_thm}. Let $X$ be an object of $\V^\calS$, and let $\sim^\calE$ be the $\Sigma$-congruence on the free $\Sigma$-algebra $F^\Sigma X$ generated by $\calE$ \eqref{cong_gen_2}. Then the $\Sigma$-algebra $\left.F^\Sigma X\middle/{\sim^\calE}\right.$ \eqref{quot_alg_eq} is the free $\calT$-algebra $F^\calT X$ on $X$.   
\end{theo_sub}

\begin{proof}
That $\left.F^\Sigma X\middle/{\sim^\calE}\right.$ is indeed a $\calT$-algebra holds by Lemma \ref{cong_gen_lem}. We have a $\V^\calS$-morphism $\left.\eta_{X}^\calT : X \to U^\calT\left(F^\Sigma X\middle/{\sim^\calE}\right)\right.$ defined as the composite
\[ \left.X \xrightarrow{\eta^\Sigma_X} U^\Sigma F^\Sigma X \xrightarrow{U^\Sigma q \ = \ q} U^\Sigma\left(F^\Sigma X\middle/{\sim^\calE}\right) = U^\calT\left(F^\Sigma X\middle/{\sim^\calE}\right). \right.\]
Now let $f : X \to U^\calT A = A$ be a morphism from $X$ to the underlying $\calS$-sorted carrier object of a $\calT$-algebra $A$. We must show that there is a unique $\Sigma$-algebra morphism $\left.f^\sharp : \left(F^\Sigma X\middle/{\sim^\calE}\right) \to A\right.$ such that $f^\sharp \circ \eta^\calT_X = f$. Since $A$ is (in particular) a $\Sigma$-algebra, we first obtain a unique $\Sigma$-algebra morphism $f^* : F^\Sigma X \to A$ such that $f^* \circ \eta^\Sigma_X = f$. By Lemma \ref{cong_gen_lem} we have that $\sim^\calE$ is contained in the kernel $\Sigma$-congruence $\sim^{f^*}$ of $f^*$. From Lemma \ref{kernel_lem} we then obtain the following factorization of $f^*$ in $\Sigma\Alg$:
\[ \left.f^* = \left(F^\Sigma X \xrightarrow{q} \left(F^\Sigma X\middle/{\sim^\calE}\right) \xrightarrow{f^\sharp} A\right). \right.\]
We then have \[ f^\sharp \circ \eta^\calT_X = f^\sharp \circ q \circ \eta^\Sigma_X = f^* \circ \eta^\Sigma_X = f, \] as desired. For the uniqueness of $f^\sharp$, let $\left.h : \left(F^\Sigma X\middle/{\sim^\calE}\right) \to A\right.$ satisfy $h \circ \eta_X^\calT = f$. Then we have \[ h \circ q \circ \eta^\Sigma_X = h \circ \eta^\calT_X = f = f^* \circ \eta^\Sigma_X = f^\sharp \circ q \circ \eta^\Sigma_X, \] so that $h \circ q = f^\sharp \circ q$ and hence $h = f^\sharp$ by the surjectivity of each $q_S$ ($S \in \calS$).   
\end{proof}

\noindent The following theorem now follows immediately from Theorems \ref{free_sig_alg_cc} and \ref{cc_free_alg_thm}.

\begin{theo_sub}
\label{free_alg_cc}
Let $\calT = (\Sigma, \calE)$ be a $\V$-enriched $\calS$-sorted equational theory, let $X$ be an object of $\V^\calS$, and let $\sim^\calE$ be the $\Sigma$-congruence on the free $\Sigma$-algebra $F^\Sigma X$ generated by $\calE$. Then the $\calS$-sorted carrier object $U^\calT F^\calT X$ of the free $\calT$-algebra $F^\calT X$ on $X$ is 
\[ \left. U^\calT F^\calT X = \left(\bigvee_n \Omega^n(X)\right) \middle/ {\sim^\calE} = \left(\left(\bigvee_n \Omega^n(X)_S\right)\middle/{\sim^\calE_S}\right)_{S \in \calS}, \right. \]
where the $\calS$-sorted objects $\Omega^n(X)$ of $\V$ ($n \geq 0$) in the fibre over the $\calS$-sorted set $\Term_{|\Sigma|}(|X|)$ are defined as follows \eqref{free_sig_alg_cc}:
\begin{enumerate}[leftmargin=*]
\item $\Omega^0(X)$ is the final $\V^\calS$-structure on $\Term_{|\Sigma|}(|X|)$ induced by the $\calS$-sorted inclusion function $\eta = \eta_{|X|}^{|\Sigma|} : |X| \to \Term_{|\Sigma|}(|X|)$. 

\item Given $\Omega^n(X) \in \Fib\left(\Term_{|\Sigma|}(|X|)\right)$ ($n \geq 0$), we first define, for each sort $S \in \calS$, an element $\Omega^n(X, S) \in \Fib\left(\Term_{|\Sigma|}(|X|)_S\right)$ as the final $\V$-structure on the set $\Term_{|\Sigma|}(|X|)_S$ induced by the functions
\[ \widehat{\sigma} : \left|P \times \Omega^n(X)_{S_1} \times \ldots \times \Omega^n(X)_{S_m}\right| = |P| \times \left|\Omega^n(X)_{S_1}\right| \times \ldots \times \left|\Omega^n(X)_{S_m}\right| \to \Term_{|\Sigma|}(|X|)_S \]
\[ (p, t_1, \ldots, t_m) \mapsto \sigma_p(t_1, \ldots, t_m) \]
 for all $\sigma \in \Sigma$ of type $((S_1, \ldots, S_m), S, P)$. We thus obtain an $\calS$-sorted object $\left(\Omega^n(X, S)\right)_{S \in \calS}$ of $\V$ in the fibre over the $\calS$-sorted set $\Term_{|\Sigma|}(|X|)$, and we then define \[ \Omega^{n+1}(X) := \Omega^n(X) \vee \left(\Omega^n(X, S)\right)_{S \in \calS} \] in the fibre over the $\calS$-sorted set $\Term_{|\Sigma|}(|X|)$. \qed   
\end{enumerate}     
\end{theo_sub}

\noindent The following corollary is the single-sorted version of Theorem \ref{free_alg_cc}.

\begin{cor_sub}
\label{free_alg_cc_SS}
Let $\calT = (\Sigma, \calE)$ be a $\V$-enriched single-sorted equational theory, let $X$ be an object of $\V$, and let $\sim^\calE$ be the $\Sigma$-congruence on the free $\Sigma$-algebra $F^\Sigma X$ generated by $\calE$. Then the carrier object $U^\calT F^\calT X$ of the free $\calT$-algebra $F^\calT X$ on $X$ is  
\[ \left. U^\calT F^\calT X = \left(\bigvee_n \Omega^n(X)\right) \middle/ {\sim^\calE}, \right. \]
where the objects $\Omega^n(X)$ of $\V$ ($n \geq 0$) in the fibre over the set $\Term_{|\Sigma|}(|X|)$ are defined as follows \eqref{free_sig_alg_cc_SS}:
\begin{enumerate}[leftmargin=*]
\item $\Omega^0(X)$ is the final $\V$-structure on the set $\Term_{|\Sigma|}(|X|)$ induced by the inclusion function $\eta = \eta_{|X|}^{|\Sigma|} : |X| \to \Term_{|\Sigma|}(|X|)$. 

\item Given $\Omega^n(X) \in \Fib\left(\Term_{|\Sigma|}(|X|)\right)$ ($n \geq 0$), we first define $\Omega^n(X, \Sigma)$ to be the final $\V$-structure on the set $\Term_{|\Sigma|}(|X|)$ induced by the functions
\[ \widehat{\sigma} : |P| \times \left|\Omega^n(X)\right|^m \to \Term_{|\Sigma|}(|X|) \] 
\[ (p, t_1, \ldots, t_m) \mapsto \sigma_p(t_1, \ldots, t_m) \]
for all $\sigma \in \Sigma$ of arity $m \geq 0$ and parameter $P$. We then define \[ \Omega^{n+1}(X) := \Omega^n(X) \vee \Omega^n(X, \Sigma) \] in the fibre over the set $\Term_{|\Sigma|}(|X|)$.  \qed 
\end{enumerate}     
\end{cor_sub}

\section{Examples of enriched multi-sorted equational theories}
\label{examples}

In this section, we provide several examples of $\V$-enriched multi-sorted equational theories, whose free algebras can now be explicitly described (and computed) using Theorem \ref{explicit_free_fibre_prop_2} and Theorem \ref{free_alg_cc} (if $\V$ is cartesian closed). Unless otherwise stated, $\V$ is any topological category over $\Set$ that satisfies our main assumptions (stated at the beginning of \S\ref{first_section}), and $\calS$ is an arbitrary set of sorts. By Theorem \ref{equiv_theory_prop}, we can and shall regard a $\V$-enriched $\calS$-sorted equational theory as given by a $\V$-enriched $\calS$-sorted signature $\Sigma$ and a classical $\calS$-sorted equational theory over the underlying classical $\calS$-sorted signature $|\Sigma|$ \eqref{ord_sig}.

\begin{ex}[\textbf{$\V$-enriched equational theories determined by classical equational theories}]
\label{trad_ex}
Every classical $\calS$-sorted equational theory $\calT = (\Sigma, \calE)$ \eqref{trad_theory} can be regarded as an ordinary $\V$-enriched $\calS$-sorted equational theory $\calT^*$ by simply regarding each operation symbol in $\Sigma$ as being ordinary (i.e.~as having parameter $I$, the unit object of $\V$); $\calT^*$-algebras can then be described as $\calT$-algebras ``in $\V$''. Examples of classical single-sorted equational theories include the equational theories for semigroups, lattices, (commutative) monoids, (abelian) groups, and (commutative and/or unital) rings. Given a small category $\A$, the category of functors $\A \to \V$ can also be described as the category of algebras for the ordinary $\V$-enriched $\ob\A$-sorted equational theory determined by a certain classical $\ob\A$-sorted equational theory (e.g.~see \cite[\S 3]{Lambek_emb}, and also cf.~Example \ref{Vfunctor_ex} below).    
\end{ex}

\begin{ex}[\textbf{$H$-monoids in $\V = \Top$}]
\label{h_space}
In the following example, due to Rory Lucyshyn-Wright, we let $\V = \Top$ be the category of topological spaces, equipped with the cartesian monoidal structure (see Examples \ref{examples_ex} and \ref{top_ex}). For the purposes of this example, we define an \emph{$H$-monoid}\footnote{Here we are adopting the terminology of the nLab \cite{nlab:h-space}, but note that in this example we moreover require that the binary operation is associative up to a \emph{specified} homotopy, and that the unit element is a unit for the binary operation up to \emph{specified} homotopies. $H$-monoids were referred to as \emph{homotopy associative $H$-spaces} in \cite{Stasheff_H-spaces} and \cite{Sugawara_grouplike}.} to be a topological space $A$ equipped with a continuous binary operation $m^A : A \times A \to A$, an element $e^A \in |A|$ (i.e.~a continuous function $e^A : 1 \to A$), a homotopy $\alpha^A : [0, 1] \times A^3 \to A$ between 
\[ A \times A \times A \xrightarrow{1_A \times m^A} A \times A \xrightarrow{m^A} A \ \ \ \text{   and   } \ \ \ A \times A \times A \xrightarrow{m^A \times 1_A} A \times A \xrightarrow{m^A} A, \] a homotopy $\beta^A : [0, 1] \times A \to A$ between 
\[ A \xrightarrow{\left\langle1_A, \ !_A\right\rangle} A \times 1 \xrightarrow{\left\langle 1_A, e^A\right\rangle} A \times A \xrightarrow{m^A} A \ \ \ \text{   and   } \ \ \ 1_A : A \to A, \] and a homotopy $\gamma^A : [0, 1] \times A \to A$ between 
\[ A \xrightarrow{\left\langle!_A, 1_A\right\rangle} 1 \times A \xrightarrow{\left\langle e^A, 1_A\right\rangle} A \times A \xrightarrow{m^A} A \ \ \ \text{   and   } \ \ \ 1_A : A \to A. \] 
Let $\Sigma$ be the $\Top$-enriched single-sorted signature \eqref{signature} with an ordinary operation symbol $m$ of arity $2$, an ordinary operation symbol $e$ of arity $0$, an operation symbol $\alpha$ of arity $3$ and parameter $[0, 1]$, and operation symbols $\beta$ and $\gamma$ of arity $1$ and parameter $[0, 1]$. We define a $\Top$-enriched single-sorted equational theory $\calT = (\Sigma, \calE)$ by letting $\calE$ consist of the following syntactic $|\Sigma|$-equations in context \eqref{syntax_para}: 
\[ \left[v_1, v_2, v_3 \vdash \alpha_0\left(v_1, v_2, v_3\right) \doteq m\left(v_1, m\left(v_2, v_3\right)\right)\right] \ \ \text{ and } \ \ \left[v_1, v_2, v_3 \vdash \alpha_1\left(v_1, v_2, v_3\right) \doteq m\left(m\left(v_1, v_2\right), v_3\right)\right], \]
\[ \left[v \vdash \beta_0(v) \doteq m(v, e)\right] \ \ \text{ and } \ \ \left[v \vdash \beta_1(v) \doteq v\right], \]
\[ \left[v \vdash \gamma_0(v) \doteq m(e, v)\right] \ \ \text{ and } \ \ \left[v \vdash \gamma_1(v) \doteq v\right]. \]
A $\calT$-algebra is then precisely an $H$-monoid in the sense defined above.
\end{ex}

\begin{ex}[\textbf{Homotopy weakenings of classical equational theories}]
More generally\footnote{We recover the example of $H$-monoids from Example \ref{h_space} by taking $\calT$ to be the classical (single-sorted) equational theory for monoids.}, let $\calT = (\Sigma, \calE)$ be any classical $\calS$-sorted equational theory, and (with $\V = \Top$) let us construct what we might call its \emph{homotopy weakening} $\calT_h = \left(\Sigma_h, \calE_h\right)$, which will be a $\Top$-enriched $\calS$-sorted equational theory, with the property that $\calT_h$-algebras can be described as ``$\calT$-algebras'' in $\V = \Top$ where the equations of $\calT$ only hold up to specified homotopies. 

For each syntactic $\Sigma$-equation $\sfe \in \calE$ of sort $S$ in context $\vbar \equiv v_1 : S_1, \ldots, v_n : S_n$, we consider an operation symbol $\alpha^\sfe$ of type $((S_1, \ldots, S_n), S, [0, 1])$. We then define $\Sigma_h := \Sigma \cup \left\{\alpha^\sfe \mid \sfe \in \calE\right\}$, where we regard each $\sigma \in \Sigma$ as an ordinary operation symbol in $\Sigma_h$ with the same input and output sorts. Given a syntactic $\Sigma$-equation $\sfe \equiv [\vbar \vdash s \doteq t : S]$ in $\calE$ (with $\vbar$ as above), we consider the syntactic $\left|\Sigma_h\right|$-equations 
\begin{equation}
\label{hw_eq}
\left[\vbar \vdash \alpha_0^\sfe\left(v_1, \ldots, v_n\right) \doteq s : S\right] \ \ \text{ and } \ \ \left[\vbar \vdash \alpha_1^\sfe\left(v_1, \ldots, v_n\right) \doteq t : S\right],
\end{equation} 
and we let $\calE_h$ consist of precisely the syntactic $\left|\Sigma_h\right|$-equations in \eqref{hw_eq} for all $\sfe \in \calE$.

A $\calT_h$-algebra $A$ is then an $\calS$-sorted topological space $A = \left(A_S\right)_{S \in \calS}$ equipped with, for each $\sigma \in \Sigma$ with input sorts $(S_1, \ldots, S_n)$ and output sort $S$, a continuous function $\sigma^A : A_{S_1} \times \ldots \times A_{S_n} \to A_S$, and, for each syntactic $\Sigma$-equation $\sfe \equiv \left[\vbar \vdash s \doteq t : S\right]$ in $\calE$ as above, a homotopy $\left(\alpha^\sfe\right)^A : [0, 1] \times A_{S_1} \times \ldots \times A_{S_n} \to A_S$ from the continuous function $[\vbar \vdash s : S]^A : A_{S_1} \times \ldots \times A_{S_n} \to A_S$ to the continuous function $[\vbar \vdash t : S]^A : A_{S_1} \times \ldots \times A_{S_n} \to A_S$ \eqref{disc_to_disc_2}. Thus, a $\calT_h$-algebra can indeed be described as a $\calT$-algebra in $\V = \Top$ in which each equation only holds up to a specified homotopy.     
\end{ex}

\noindent The following example assumes some basic familiarity with enriched category theory.  

\begin{ex}[\textbf{Covariant $\V$-presheaves}]
\label{Vfunctor_ex}
Let $\A$ be any small \emph{$\V$-category}, i.e.~a small category enriched in the symmetric monoidal category $\V = (\V, \tensor, I)$. Under the (weak) assumption that the topological functor $|-| : \V \to \Set$ is represented by the unit object $I$ (equivalently, that the unit object $I$ is the discrete object on a singleton set), the small $\V$-category $\A$ is precisely a small \emph{ordinary} category $\A_0$ for which each hom-set $\A_0(A, B)$ ($A, B \in \ob\A_0$) carries the structure of an object $\A(A, B)$ of $\V$, and the composition morphisms 
\[ |\A(A, B) \tensor \A(B, C)| = \A_0(A, B) \times \A_0(B, C) \to \A_0(A, C) = |\A(A, C)| \]
\[ (f, g) \mapsto g \circ f \] ($A, B, C \in \ob\A_0$) of $\A_0$ all lift to $\V$-morphisms $\A(A, B) \tensor \A(B, C) \to \A(A, C)$.

A \emph{(covariant) $\V$-presheaf} $F$ on $\A$ (e.g.~see \cite[Definition 8.4]{formal_relative_monads})\footnote{Note that the definition given in (e.g.)~\cite[Definition 8.4]{formal_relative_monads} reduces to that given here because $|-| : \V \to \Set$ is faithful.} consists of an $\ob\A$-sorted object $(FA)_{A \in \ob\A}$ of $\V$ together with, for each ordered pair $(A, B)$ of objects of $\A$, a $\V$-morphism
\[ F_{AB} : \A(A, B) \tensor FA \to FB, \] so that the following equalities hold for all $A, B, C \in \ob\A$, all $f \in |\A(A, B)|$ and $g \in |\A(B, C)|$, and all $x \in |FA|$:
\[ F_{AA}\left(1_A, x\right) = x \]
\[ F_{AC}\left(g \circ f, x\right) = F_{BC}\left(g, F_{AB}(f, x)\right). \] When $\V$ is symmetric monoidal closed, so that $\V$ itself may be regarded as a $\V$-category, $\V$-presheaves on $\A$ are equivalently $\V$-functors $\A \to \V$.

Given $\V$-presheaves $F$ and $G$ on $\A$, a \emph{$\V$-natural transformation} $\phi : F \Longrightarrow G$ is an $\ob\A$-sorted morphism $\phi = \left(\phi_A : FA \to GA\right)_{A \in \ob\A}$ of $\V$ such that the following square commutes for all $A, B \in \ob\A$:
\begin{equation}
\label{psh_eq}
\begin{tikzcd}
	{\A(A, B) \tensor FA} &&& {\A(A, B) \tensor GA} \\
	\\
	{FB} &&& {GB}.
	\arrow["{1 \tensor \phi_A}", from=1-1, to=1-4]
	\arrow["F_{AB}"', from=1-1, to=3-1]
	\arrow["G_{AB}", from=1-4, to=3-4]
	\arrow["{\phi_B}"', from=3-1, to=3-4]
\end{tikzcd}
\end{equation}
When $\V$ is symmetric monoidal closed, a $\V$-natural transformation between $\V$-presheaves is precisely a $\V$-natural transformation between the corresponding $\V$-valued $\V$-functors. We have an ordinary category $\V\Psh(\A)$ of $\V$-presheaves on $\A$ and $\V$-natural transformations between them, which is isomorphic to the ordinary category $\V\Cat(\A, \V)$ of $\V$-functors $\A \to \V$ and $\V$-natural transformations between them when $\V$ is symmetric monoidal closed.   

We shall define a $\V$-enriched $\ob\A$-sorted equational theory $\calT_\A = \left(\Sigma_\A, \calE_\A\right)$ such that $\calT_\A\Alg = \V\Psh(\A)$. For each ordered pair $(A, B)$ of objects of $\A$, we consider an operation symbol $\sigma_{AB}$ of type $(A, B, \A(A, B))$, where the parameter object is the hom-object $\A(A, B)$ of the $\V$-category $\A$. The $\V$-enriched $\ob\A$-sorted signature $\Sigma_\A$ then consists of precisely the operation symbols $\sigma_{AB}$ for all $(A, B) \in \ob\A \times \ob\A$. For all $A, B, C \in \ob\A$, all $f \in |\A(A, B)|$, and all $g \in |\A(B, C)|$, we let $\calE_\A$ consist of precisely the following syntactic $\left|\Sigma_\A\right|$-equations in context: 
\[ \left[v : A \vdash \left(\sigma_{AA}\right)_{1_A}(v) = v\right] \]
\[ \left[v : A \vdash \left(\sigma_{BC}\right)_g\left(\left(\sigma_{AB}\right)_f(v)\right) = \left(\sigma_{AC}\right)_{gf}(v)\right]. \] 
Setting $\calT_\A := \left(\Sigma_\A, \calE_\A\right)$, w immediately have $\calT_\A\Alg = \V\Psh(\A)$.    
\end{ex}

\begin{ex}[\textbf{Relational algebraic theories}]
\label{rel_th_ex}
Let $\V = \T\Mod$ be the category of models for a relational Horn theory $\T$ without equality (see Examples \ref{examples_ex} and \ref{rel_ex}). Then a certain subclass of the \emph{relational algebraic theories} of \cite{Monadsrelational} can be equivalently described as $\T\Mod$-enriched single-sorted equational theories. However, treating this example in adequate detail is beyond the scope of the present paper. We intend to study the connection between the present paper and \cite{Monadsrelational} more closely in future work.   
\end{ex}

\section{The connection with $\V$-enriched monads and algebraic theories when $\V$ is symmetric monoidal closed}
\label{smc_section}

We suppose throughout this section that the symmetric monoidal structure of the topological category $\V = (\V, \tensor, I)$ over $\Set$ is \emph{closed}, with internal hom $[-, -]$, so that $\V$ itself may be regarded as a $\V$-category. We also suppose that the faithful (topological) functor $|-| : \V \to \Set$ is represented by the unit object $I$ of $\V$. Under these assumptions, we shall exposit the connection between the $\V$-enriched $\calS$-sorted equational theories of \S\ref{theories_section} and the \emph{(diagrammatic) presentations} of certain $\V$-enriched monads and algebraic theories studied in \cite{Pres, EP, Struct, Ax}. We assume that the reader has some basic familiarity with enriched category theory (in particular, with the basic theory of weighted colimits, which we mainly employ in \ref{eleuth_para} and Theorem \ref{pi_category_thm}); see \cite{Kelly}.  

\begin{para}[\textbf{$\V^\calS$ as a $\V$-category}]
\label{Vcat_para}
The category $\V^\calS$ may also be regarded as a $\V$-category: given $\calS$-sorted objects $X$ and $Y$ of $\V$, the hom-object $\V^\calS(X, Y)$ in $\V$ is the product $\prod_{S \in \calS} \left[X_S, Y_S\right]$. The $\V$-category $\V^\calS$ is also \emph{tensored} \cite[\S 3.7]{Kelly}, with tensors in $\V^\calS$ being formed \emph{pointwise} \cite[\S 3.3]{Kelly}: given an object $V$ of $\V$ and an object $X$ of $\V^\calS$, the tensor $V \tensor X$ in $\V^\calS$ is given by
\[ V \tensor X = \left(V \tensor X_S\right)_{S \in \calS}. \]
\end{para}

\begin{para}[\textbf{The subcategory of arities $\N_\calS \hookrightarrow \V^\calS$}]
\label{N_S_para}
For each natural number $n \in \N$, we write $n \cdot I = \coprod_n I$ for the $n^{\mathsf{th}}$ copower of the unit object $I$. We then write $\N_\V \hookrightarrow \V$ for the full sub-$\V$-category consisting of the objects $n \cdot I$ ($n \in \N$). For each $n \in \N$ and each $V \in \ob\V$, note that $[n \cdot I, V] \cong V^n$, the $n^{\mathsf{th}}$ power of the object $V$. We also write $\N_\V^\calS$ for the full sub-$\V$-category $\prod_\calS \N_\V \hookrightarrow \prod_\calS \V = \V^\calS$. Finally, we write $\N_\calS$ for the full sub-$\V$-category of $\N_\V^\calS$ (and hence of $\V^\calS$) consisting of the objects $\left(n_S \cdot I\right)_{S \in \calS}$ of $\N_\V^\calS$ such that $n_S = 0$ for all but finitely many $S \in \calS$. We shall denote an object of $\N_\calS$ as $J = \left(n_{S_1}, \ldots, n_{S_m}\right)$, where $S_1, \ldots, S_m$ are the finitely many sorts with $n_{S_i} \neq 0$ ($1 \leq i \leq m$). Note that $\N_\calS \hookrightarrow \V^\calS$ is closed under (conical) finite coproducts, which are formed pointwise in $\V^\calS$. 

We first claim that the full sub-$\V$-category $\N_\calS \hookrightarrow \V^\calS$ is \emph{dense} (in the enriched sense; see \cite[Chapter 5]{Kelly}). First, for each $S \in \calS$, the $\V$-category $\V^\calS$ contains the ($\V$-valued) \emph{representable} corresponding to $S$, which is the $\calS$-sorted object $\calS(S, -)$ of $\V$ with $\calS(S, -)_S = I = 1 \cdot I$ and $\calS(S, -)_{T} = 0 = 0 \cdot I$ for all $T \neq S \in \calS$. So $\N_\calS$ clearly contains the representables, and then since the full sub-$\V$-category of $\V^\calS$ consisting of the representables is dense (by e.g.~\cite[Proposition 5.16]{Kelly}), we deduce from \cite[Theorem 5.13]{Kelly} that $\N_\calS \hookrightarrow \V^\calS$ is dense.

Thus, the full sub-$\V$-category $\N_\calS \hookrightarrow \V^\calS$ is a small \emph{subcategory of arities} in the sense of \cite[Definition 3.1]{Pres}.
\end{para}

By a \emph{$\V$-category over $\V^\calS$} we mean an object of the (strict) slice category $\VCAT/\V^\calS$, i.e.~a $\V$-category $\A$ equipped with a $\V$-functor $U : \A \to \V^\calS$.

\begin{para}[\textbf{Free-form $\N_\calS$-signatures and their algebras}]
\label{EP_sig_para}
We now recall some relevant notions from \cite{EP} (in view of \cite[Remark 5.18]{EP}), specialized to the small subcategory of arities $\N_\calS \hookrightarrow \V^\calS$ \eqref{N_S_para}. A \emph{free-form $\N_\calS$-signature} \cite[Definition 5.1]{EP} is a set $\scrS$ of operation symbols equipped with an assignment to each operation symbol $\sigma \in \scrS$ of an \emph{(input) arity} $J = J_\sigma \in \ob\N_\calS$ and a \emph{parameter (object)} $P = P_\sigma \in \ob\V^\calS$. An \emph{$\scrS$-algebra} $A$ \cite[Definition 5.2]{EP} is an object $A$ of $\V^\calS$ equipped with, for each $\sigma \in \scrS$ of arity $J$ and parameter $P$, a $\V^\calS$-morphism
\[ \sigma^A : \V^\calS\left(J, A\right) \tensor P \to A. \] With $J = \left(n_{S_1}, \ldots, n_{S_m}\right)$ and $n_i := n_{S_i}$ ($1 \leq i \leq m$), we have $\V^\calS(J, A) \cong A_{S_1}^{n_1} \times \ldots \times A_{S_m}^{n_m}$. Writing $P = \left(P_S\right)_{S \in \calS}$, we then have
\begin{equation}
\label{EP_sig_iso}
\V^\calS(J, A) \tensor P \cong \left(P_S \tensor \left(A_{S_1}^{n_1} \times \ldots \times A_{S_m}^{n_m}\right)\right)_{S \in \calS}.
\end{equation}
Then $\sigma^A : \V^\calS\left(J, A\right) \tensor P \to A$ is completely determined by, for each sort $S \in \calS$, a $\V$-morphism
\begin{equation}
\label{EP_sig_eq}
\sigma^A_S : P_S \tensor \left(A_{S_1}^{n_1} \times \ldots \times A_{S_m}^{n_m}\right) \to A_S.
\end{equation} 
Given $\scrS$-algebras $A$ and $B$, a \emph{morphism of $\scrS$-algebras $f : A \to B$} \cite[Definition 5.2]{EP} is a morphism $f : A \to B$ of $\V^\calS$ that makes the following diagram commute for each $\sigma \in \scrS$ of arity $J$ and parameter $P$:
\[\begin{tikzcd}
	{\V^\calS(J, A) \tensor P} &&& {\V^\calS(J, B) \tensor P} \\
	\\
	{A} &&& {B}.
	\arrow["{\V^\calS(J, f) \tensor 1_P}", from=1-1, to=1-4]
	\arrow["{\sigma^A}"', from=1-1, to=3-1]
	\arrow["{\sigma^B}", from=1-4, to=3-4]
	\arrow["{f}"', from=3-1, to=3-4]
\end{tikzcd}\] 
We then have an ordinary category $\scrS\Alg_0$ of $\scrS$-algebras, which (in view of \cite[Definition 5.2]{EP}) underlies a $\V$-category $\scrS\Alg$ in which each hom-object $\scrS\Alg(A, B)$ is defined as the \emph{pairwise equalizer} \cite[2.1]{Commutants} of the following $\scrS$-indexed family of parallel pairs in $\V$ ($\sigma \in \scrS$):
\[ \V^\calS(A, B) \xrightarrow{\V^\calS\left(J_\sigma, -\right)_{AB} \tensor P_\sigma} \V^\calS\left(\V^\calS\left(J_\sigma, A\right) \tensor P_\sigma, \V^\calS\left(J_\sigma, B\right) \tensor P_\sigma\right) \xrightarrow{\V^\calS\left(1, \sigma^B\right)} \V^\calS\left(\V^\calS\left(J_\sigma, A\right) \tensor P_\sigma, B\right), \]
\[ \V^\calS(A, B) \xrightarrow{\V^\calS\left(\sigma^A, 1\right)} \V^\calS\left(\V^\calS\left(J_\sigma, A\right) \tensor P_\sigma, B\right). \]
Composition in $\scrS\Alg$ is then defined in the unique manner that allows the resulting subobjects $U^\scrS_{AB} : \scrS\Alg(A, B) \rightarrowtail \V^\calS(A, B)$ ($A, B \in \scrS\Alg$) to serve as the structural morphisms of a faithful $\V$-functor $U^\scrS : \scrS\Alg \to \V^\calS$ that sends each $\scrS$-algebra $A$ to its underlying $\calS$-sorted carrier object $A \in \ob\V^\calS$. We may thus regard $\scrS\Alg$ as a $\V$-category over $\V^\calS$. 
\end{para}

\begin{para}[\textbf{Diagrammatic $\N_\calS$-presentations and their algebras}]
\label{EP_eq_para}
Given a free-form $\N_\calS$-signature $\scrS$, a \emph{natural $\scrS$-operation}\footnote{In \cite[Definition 5.10]{EP} the concept of \emph{$\V$-natural $\scrS$-operation} is defined, but $\V$-naturality reduces to ordinary naturality in the present context because $\V(I, -) : \V \to \Set$ is faithful.} \cite[Definition 5.10]{EP} is a natural transformation $\omega : \V^\calS\left(J, U^\scrS-\right) \tensor P \Longrightarrow U^\scrS : \scrS\Alg \to \V^\calS$ with specified \emph{(input) arity} $J \in \ob\N_\calS$ and \emph{parameter (object)} $P \in \ob\V^\calS$. Explicitly, a natural $\scrS$-operation $\omega$ consists of $\calS$-sorted $\V$-morphisms $\omega_A : \V^\calS(J, A) \tensor P \to A$ ($A \in \ob\scrS\Alg$) such that the following diagram commutes for each morphism $f : A \to B$ of $\scrS\Alg$:
\[\begin{tikzcd}
	{\V^\calS(J, A) \tensor P} &&& {\V^\calS(J, B) \tensor P} \\
	\\
	{A} &&& {B}.
	\arrow["{\V^\calS(J, f) \tensor 1_{\calP}}", from=1-1, to=1-4]
	\arrow["{\omega_A}"', from=1-1, to=3-1]
	\arrow["{\omega_B}", from=1-4, to=3-4]
	\arrow["{f}"', from=3-1, to=3-4]
\end{tikzcd}\]  
A \emph{diagrammatic $\scrS$-equation} \cite[Definition 5.10]{EP}, written $\omega \doteq \nu$, is a pair of natural $\scrS$-operations $\omega, \nu : \V^\calS\left(J, U^\scrS-\right) \tensor P \Longrightarrow U^\scrS$ with the same arity $J$ and parameter $P$. An $\scrS$-algebra $A$ \emph{satisfies} a diagrammatic $\scrS$-equation $\omega \doteq \nu$ if $\omega_A = \nu_A : \V^\calS(J, A) \tensor P \to A$. Finally, a \emph{diagrammatic $\N_\calS$-presentation} \cite[Definition 5.12]{EP} is a pair $\scrT = (\scrS, \scrE)$ consisting of a free-form $\N_\calS$-signature $\scrS$ and a set $\scrE$ of diagrammatic $\scrS$-equations. We write $\scrT\Alg$ for the full sub-$\V$-category of $\scrS\Alg$ consisting of the \emph{$\scrT$-algebras}, i.e.~the $\scrS$-algebras that satisfy all of the diagrammatic $\scrS$-equations in $\scrE$. The forgetful $\V$-functor $U^\scrS : \scrS\Alg \to \V^\calS$ restricts to a forgetful $\V$-functor $U^\scrT : \scrT\Alg \to \V^\calS$, and we may thus regard $\scrT\Alg$ as a $\V$-category over $\V^\calS$.     
\end{para} 

\noindent In the next two propositions, we show that $\V$-enriched $\calS$-sorted equational theories (in the sense of Definition \ref{eqn_theory}) have the same ``expressive power" as diagrammatic $\N_\calS$-presentations. This will allow us to ultimately show in Theorem \ref{mnd_pres_thm} (see also Remark \ref{amenable_rmk}) that $\V$-enriched $\calS$-sorted equational theories have the same ``expressive power" as certain $\V$-enriched monads on $\V^\calS$ (as well as the \emph{$\N_\calS$-theories} mentioned in the Introduction and defined in \ref{Jary_para}). 

Recall from Remarks \ref{enriched_rmk} and \ref{enriched_rmk_2} that in the context of the present section (where $\V$ is symmetric monoidal closed), the ordinary category $\Sigma\Alg$ of $\Sigma$-algebras for a $\V$-enriched $\calS$-sorted signature $\Sigma$ underlies a $\V$-category $\Sigma\Alg$ over $\V^\calS$, while the ordinary category $\calT\Alg$ of $\calT$-algebras for a $\V$-enriched $\calS$-sorted equational theory $\calT$ underlies a $\V$-category $\calT\Alg$ over $\V^\calS$.

\begin{prop}
\label{equiv_presentations_prop1}
\
\begin{enumerate}[leftmargin=*]
\item For each $\V$-enriched $\calS$-sorted signature $\Sigma$, there is a free-form $\N_\calS$-signature $\scrS_\Sigma$ such that $\Sigma\Alg \cong \scrS_\Sigma\Alg$ in $\VCAT/\V^\calS$. 

\item For each algebraic $\Sigma$-equation $\omega \doteq \nu$, there is a diagrammatic $\scrS_\Sigma$-equation $\bar{\omega} \doteq \bar{\nu}$ such that a $\Sigma$-algebra $A$ satisfies $\omega \doteq \nu$ iff the corresponding $\scrS_\Sigma$-algebra satisfies $\bar{\omega} \doteq \bar{\nu}$. 

\item For each $\V$-enriched $\calS$-sorted equational theory $\calT$ (in the sense of Definition \ref{eqn_theory}), there is a diagrammatic $\N_\calS$-presentation $\scrT_\calT$ such that $\calT\Alg \cong \scrT_\calT\Alg$ in $\VCAT/\V^\calS$.
\end{enumerate} 
\end{prop}

\begin{proof}
(3) readily follows from (1) and (2). For (1), we define the free-form $\N_\calS$-signature $\scrS_\Sigma$ as follows. For each operation symbol $\sigma \in \Sigma$ of type $\left(\Sbar = (S_1, \ldots, S_m), S, P\right)$ and each $1 \leq i \leq m$, let $1 \leq n_i \leq m$ be the number of occurrences of $S_i$ in $(S_1, \ldots, S_m)$. We write $J_{\Sbar}$ for the object of $\N_\calS$ with $\left(J_{\Sbar}\right)_{S_i} := n_i \cdot I$ for each $1 \leq i \leq m$ and $\left(J_{\Sbar}\right)_{T} := 0 \cdot I$ otherwise, and we write $\calP_{P, S}$ for the object of $\V^\calS$ with $\left(\calP_{P, S}\right)_{S} := P$ and $\left(\calP_{P, S}\right)_{T} := 0$ otherwise. We then declare that $\scrS_\Sigma$ has, for each operation symbol $\sigma \in \Sigma$ of type $\left(\Sbar, S, P\right)$ as above, an operation symbol $\bar{\sigma}$ with arity $J_{\Sbar}$ and parameter $\calP_{P, S}$. In view of the discussion in \ref{EP_sig_para}, it then readily follows that we have an isomorphism of ordinary categories $\Sigma\Alg \cong \scrS_\Sigma\Alg_0$ in $\mathsf{CAT}/\V^\calS$ (which is therefore identity-on-homs). Specifically, given a $\Sigma$-algebra $A$, the corresponding $\scrS_\Sigma$-algebra $\widehat{A}$ has the same underlying $\calS$-sorted object $A$ of $\V$, and for each $\sigma \in \Sigma$ as above, the $\calS$-sorted $\V$-morphism
\[ \bar{\sigma}^{\widehat{A}} : \V^\calS\left(J_{\Sbar}, A\right) \tensor \calP_{P, S} \to A \] is (by the definition of $\calP_{P, S}$ and the fact that $0 \tensor X \cong 0$ for each object $X$ of $\V$) completely determined by its $S$-component, which we define to be
\[ \left(\bar{\sigma}^{\widehat{A}}\right)_S := \sigma^A : P \tensor \left(A_{S_1} \times \ldots \times A_{S_m}\right) \to A_S. \] The desired isomorphism of $\V$-categories $\Sigma\Alg \cong \scrS_\Sigma\Alg$ in $\VCAT/\V^\calS$ then follows from the fact that the underlying ordinary categories over $\V^\calS$ are isomorphic, and the fact that the $\V$-category structure of $\Sigma\Alg$ defined in Remark \ref{enriched_rmk} is readily seen to be the transport of the $\V$-category structure of $\scrS_\Sigma\Alg$ across this (identity-on-homs) isomorphism.

To prove (2), let $\omega \doteq \nu$ be an algebraic $\Sigma$-equation of type $\left(\Sbar = (S_1, \ldots, S_m), S, P\right)$. So $\omega, \nu : P \tensor U^{\Sbar} \Longrightarrow U^S : \Sigma\Alg \to \V^\calS$ are algebraic $\Sigma$-operations, and thus for each $\Sigma$-algebra $A$ we have parallel $\V$-morphisms
\[ \omega_A, \nu_A : P \tensor \left(A_{S_1} \times \ldots \times A_{S_m}\right) \rightrightarrows A_S. \] We now define a diagrammatic $\scrS_\Sigma$-equation $\bar{\omega} \doteq \bar{\nu}$ with arity $J_{\Sbar}$ and parameter $\calP_{P, S}$. For each $\scrS_\Sigma$-algebra $A$, with corresponding $\Sigma$-algebra $\overline{A}$, we must define $\calS$-sorted $\V$-morphisms
\[ \bar{\omega}_A, \bar{\nu}_A : \V^\calS\left(J_{\Sbar}, A\right) \tensor \calP_{P, S} \rightrightarrows A. \] By the definition of $\calP_{P, S}$, these $\calS$-sorted $\V$-morphisms will be completely determined by their $S$-components, which we define to be
\[ \left(\bar{\omega}_A\right)_S := \omega_{\overline{A}}, \left(\bar{\nu}_A\right)_S := \nu_{\overline{A}} : P \tensor \left(A_{S_1} \times \ldots \times A_{S_m}\right) \rightrightarrows A_S. \] Then $\bar{\omega}$ and $\bar{\nu}$ are evidently algebraic $\scrS_\Sigma$-operations. A $\Sigma$-algebra $A$ satisfies $\omega \doteq \nu$ iff $\omega_A = \nu_A$, iff $\omega_{\overline{\widehat{A}}} = \nu_{\overline{\widehat{A}}}$, iff $\left(\bar{\omega}_{\widehat{A}}\right)_S = \left(\bar{\nu}_{\widehat{A}}\right)_S$, iff $\bar{\omega}_{\widehat{A}} = \bar{\nu}_{\widehat{A}}$, iff the corresponding $\scrS_\Sigma$-algebra $\widehat{A}$ satisfies the diagrammatic $\scrS_\Sigma$-equation $\bar{\omega} \doteq \bar{\nu}$, as desired. 
\end{proof}

\begin{prop}
\label{equiv_presentations_prop2}
\
\begin{enumerate}[leftmargin=*]
\item For each free-form $\N_\calS$-signature $\scrS$, there is a $\V$-enriched $\calS$-sorted signature $\Sigma_\scrS$ such that $\scrS\Alg \cong \Sigma_\scrS\Alg$ in $\VCAT/\V^\calS$. 

\item For each diagrammatic $\scrS$-equation $\sfe$, there is a set $\calE_\sfe$ of algebraic $\Sigma_\scrS$-equations such that an $\scrS$-algebra satisfies $\sfe$ iff the corresponding $\Sigma_\scrS$-algebra satisfies each algebraic $\Sigma_\scrS$-equation in $\calE_\sfe$. 

\item For each diagrammatic $\N_\calS$-presentation $\scrT$, there is a $\V$-enriched $\calS$-sorted equational theory $\calT_\scrT$ (in the sense of Definition \ref{eqn_theory}) such that $\scrT\Alg \cong \calT_\scrT\Alg$ in $\VCAT/\V^\calS$.
\end{enumerate}  
\end{prop}

\begin{proof}
(3) readily follows from (1) and (2). To prove (1), we define the $\V$-enriched $\calS$-sorted signature $\Sigma_\scrS$ as follows. For each operation symbol $\sigma \in \scrS$ of arity $J = \left(n_{S_1}, \ldots, n_{S_m}\right)$ and parameter $P = \left(P_S\right)_{S \in \calS}$, we declare that $\Sigma_\scrS$ will have, for each sort $S \in \calS$, an operation symbol $\sigma_S$ with input sorts $(S_1, \ldots, S_1, \ldots, S_m, \ldots, S_m)$ (with $n_i := n_{S_i}$ occurrences of $S_i$ for each $1 \leq i \leq m$), output sort $S$, and parameter $P_S$. In view of the discussion in \ref{EP_sig_para}, it readily follows that we have an isomorphism of ordinary categories $\scrS\Alg_0 \cong \Sigma_\scrS\Alg$ in $\mathsf{CAT}/\V^\calS$ (which is therefore identity-on-homs). Specifically, given an $\scrS$-algebra $A$, the corresponding $\Sigma_\scrS$-algebra $\widehat{A}$ has the same underlying $\calS$-sorted object of $\V$, and for each $\sigma \in \scrS$ as above and each $S \in \calS$, we have (with the notation of \eqref{EP_sig_eq})
\[ \sigma_S^{\widehat{A}} := \sigma^A_S : P_S \tensor \left(A_{S_1}^{n_1} \times \ldots \times A_{S_m}^{n_m}\right) \to A_S. \]
The desired isomorphism of $\V$-categories $\scrS\Alg \cong \Sigma_\scrS\Alg$ in $\VCAT/\V^\calS$ then follows from the fact that the underlying ordinary categories over $\V^\calS$ are isomorphic, and the fact that the $\V$-category structure of $\Sigma_\scrS\Alg$ (see Remark \ref{enriched_rmk}) is readily seen to be the transport of the $\V$-category structure of $\scrS\Alg$ across this (identity-on-homs) isomorphism.

To prove (2), let $\sfe \equiv (\omega \doteq \nu)$ be a diagrammatic $\scrS$-equation of arity $J$ (as above) and parameter $P$. So $\omega, \nu : \V^\calS\left(J, U^\scrS-\right) \tensor P \Longrightarrow U^\scrS : \scrS\Alg \to \V^\calS$ are natural $\scrS$-operations, and thus for each $\scrS$-algebra $A$ we have (in view of \eqref{EP_sig_iso}) parallel $\calS$-sorted $\V$-morphisms \[ \omega_A, \nu_A : \left(P_S \tensor \left(A_{S_1}^{n_1} \times \ldots \times A_{S_m}^{n_m}\right)\right)_{S \in \calS} \rightrightarrows \left(A_S\right)_{S \in \calS}. \]
For each sort $S \in \calS$, we now define an algebraic $\Sigma_\scrS$-equation $\omega^S \doteq \nu^S$ with input sorts $(S_1, \ldots, S_1, \ldots, S_m, \ldots, S_m)$ (with $n_i$ occurrences of $S_i$ for each $1 \leq i \leq m$), output sort $S$, and parameter $P_S$. For each $\Sigma_\scrS$-algebra $A$, with corresponding $\scrS$-algebra $\overline{A}$, we define
\[ \omega^S_A := \left(\omega_{\overline{A}}\right)_S, \nu^S_A := \left(\nu_{\overline{A}}\right)_S : P_S \tensor \left(A_{S_1}^{n_1} \times \ldots \times A_{S_m}^{n_m}\right) \rightrightarrows A_S. \] Then $\omega^S, \nu^S$ are evidently algebraic $\Sigma_\scrS$-operations. An $\scrS$-algebra $A$ satisfies $\omega \doteq \nu$ iff $\omega_A = \nu_A$, iff $\left(\omega_A\right)_S =  \left(\nu_A\right)_S$ for each sort $S \in \calS$, iff $\left(\omega_{\overline{\widehat{A}}}\right)_S =  \left(\nu_{\overline{\widehat{A}}}\right)_S$ for each sort $S \in \calS$, iff $\omega^S_{\widehat{A}} = \nu^S_{\widehat{A}}$ for each sort $S \in \calS$, iff the corresponding $\Sigma_\scrS$-algebra $\widehat{A}$ satisfies the algebraic $\Sigma_\scrS$-equation $\omega^S \doteq \nu^S$ for each sort $S \in \calS$. We therefore let $\calE_\sfe$ be the set of algebraic $\Sigma_\scrS$-equations $\calE_\sfe := \left\{\omega^S \doteq \nu^S \mid S \in \calS\right\}$.      
\end{proof}

\begin{para}
\label{eleuth_para}
We write $j : \N_\calS \hookrightarrow \V^\calS$ for the inclusion and $\Phi_{\N_\calS}$ for the class of all small weights $\V^\calS(j-, X) : \N_\calS^\op \to \V$ for $X \in \ob\V^\calS$. We say that the small subcategory of arities $\N_\calS \hookrightarrow \V^\calS$ is \emph{eleutheric} \cite[Definition 3.3]{Pres} if each representable $\V$-functor $\V^\calS(J, -) : \V^\calS \to \V$ ($J \in \ob\N_\calS$) preserves weighted $\Phi_{\N_\calS}$-colimits, or equivalently preserves ($\V$-enriched) left Kan extensions along $j$.\footnote{\cite[Definition 3.3]{Pres} also requires that $\V^\calS$ have weighted $\Phi_{\N_\calS}$-colimits, but this is true in the present context because $\N_\calS$ is small and $\V^\calS$ is cocomplete (as a $\V$-category), since $\V$ is cocomplete \eqref{top_para}.}
\end{para}

\noindent Given a small $\V$-category $\A$, we write $[\A, \V]$ for the $\V$-functor $\V$-category. The following result essentially generalizes \cite[Example 7.5.5]{EAT}, which treats the case where $\calS$ is a singleton.

\begin{theo}
\label{pi_category_thm}
Suppose that $(\V, \tensor, I) = (\V, \times, 1)$, so that $\V$ is cartesian closed. Then the small subcategory of arities $\N_\calS \hookrightarrow \V^\calS$ is eleutheric.
\end{theo}  

\begin{proof}
Let $J = \left(n_{S_1}, \ldots, n_{S_m}\right) \in \ob\N_\calS$, let $X \in \ob\V^\calS$, and let us show that $\V^\calS(J, -) : \V^\calS \to \V$ preserves $\V^\calS(j-, X)$-weighted colimits. Writing $n_i := n_{S_i}$ ($1 \leq i \leq m$), from the definitions of $J$ and $\V^\calS$ we readily have
\begin{equation}
\label{eleuth_eq0}
\V^\calS(J, V) \cong V_{S_1}^{n_1} \times \ldots \times V_{S_m}^{n_m}
\end{equation}
$\V$-naturally in $V \in \V^\calS$.
Let $D : \N_\calS \to \V^\calS$ be a $\V$-functor, and let us write $D_S := \pi_S \circ D : \N_\calS \to \V$ for each $S \in \calS$, where $\pi_S : \V^\calS \to \V$ is the product projection $\V$-functor. We need to show that 
\begin{equation}
\label{eleuth_eq}
\V^\calS\left(J, \V^\calS(j-, X) \ast D\right) \cong \V^\calS(j-, X) \ast \V^\calS(J, D-)
\end{equation} 
in $\V$. Since weighted colimits are formed pointwise in $\V^\calS$ \cite[\S 3.3]{Kelly}, we have
\begin{equation}
\label{eleuth_eq1}
\V^\calS(j-, X) \ast D = \left(\V^\calS(j-, X) \ast D_S\right)_{S \in \calS}.
\end{equation} 
Now the inclusion $j : \N_\calS \hookrightarrow \V^\calS$ preserves (conical) finite coproducts by \ref{N_S_para}, so that $\V^\calS(j-, X) : \N_\calS^\op \to \V$ preserves (conical) finite products. The weighted-colimit-taking $\V$-functor \[ \V^\calS(j-, X) \ast (-) : \left[\N_\calS, \V\right] \to \V \] is the left Kan extension of $\V^\calS(j-, X) : \N_\calS^\op \to \V$ along the Yoneda embedding $\y : \N_\calS^\op \to \left[\N_\calS, \V\right]$ by \cite[(3.9) and (4.31)]{Kelly}. Because $\V$ is cartesian closed, we then deduce from \cite[Theorem 1.5, Proposition 2.1, and Example 3.1]{Borceux_Kan} that $\V^\calS(j-, X) \ast (-) : \left[\N_\calS, \V\right] \to \V$ preserves conical finite products. For each $1 \leq i \leq m$, writing $D_{S_i}^{n_i}$ for the conical $n_i^{\mathsf{th}}$-power of $D_{S_i}$ in the $\V$-functor $\V$-category $\left[\N_\calS, \V\right]$, we therefore have
\begin{equation}
\label{eleuth_eq2}
\V^\calS(j-, X) \ast D_{S_i}^{n_i} \cong \left(\V^\calS(j-, X) \ast D_{S_i}\right)^{n_i}.
\end{equation}
From \eqref{eleuth_eq0} we also readily obtain
\begin{equation}
\label{eleuth_eq3}
\V^\calS(J, D-) \cong D_{S_1}^{n_1} \times \ldots \times D_{S_m}^{n_m} : \N_\calS \to \V.
\end{equation} 
Using successively \eqref{eleuth_eq1}, \eqref{eleuth_eq0}, \eqref{eleuth_eq2}, the fact that $\V^\calS(j-, X) \ast (-) : \left[\N_\calS, \V\right] \to \V$ preserves conical finite products, and \eqref{eleuth_eq3}, we therefore have 
\begin{align*}
&\ \ \ \ \V^\calS\left(J, \V^\calS(j-, X) \ast D\right) \\	
&= \V^\calS\left(J, \left(\V^\calS(j-, X) \ast D_S\right)_{S \in \calS}\right) \\
&\cong \left(\V^\calS(j-, X) \ast D_{S_1}\right)^{n_1} \times \ldots \times \left(\V^\calS(j-, X) \ast D_{S_m}\right)^{n_m} \\
&\cong \left(\V^\calS(j-, X) \ast D_{S_1}^{n_1}\right) \times \ldots \times \left(\V^\calS(j-, X) \ast D_{S_m}^{n_m}\right) \\
&\cong \V^\calS(j-, X) \ast \left(D_{S_1}^{n_1} \times \ldots \times D_{S_m}^{n_m}\right) \\
&\cong \V^\calS(j-, X) \ast \V^\calS(J, D-), 
\end{align*} 
thus establishing the desired isomorphism \eqref{eleuth_eq}.       
\end{proof} 

\begin{para}
\label{Jary_para}
We say that a $\V$-monad $\T = (T, \eta, \mu)$ on $\V^\calS$ is \emph{$\N_\calS$-ary} \cite[Definition 4.1]{Pres} if the $\V$-endofunctor $T : \V^\calS \to \V^\calS$ preserves weighted $\Phi_{\N_\calS}$-colimits, or equivalently if it preserves ($\V$-enriched) left Kan extensions along $j : \N_\calS \hookrightarrow \V^\calS$. We say that a $\V$-monad $\T$ on $\V^\calS$ is \emph{$\N_\calS$-nervous} if it satisfies the (somewhat technical) conditions of \cite[Definition 4.9]{Struct}. If $(\V, \tensor, I) = (\V, \times, 1)$, so that $\V$ is cartesian closed, then the subcategory of arities $\N_\calS \hookrightarrow \V^\calS$ is eleutheric by Theorem \ref{pi_category_thm}, and we then deduce from \cite[Corollary 5.1.14]{Struct} that the $\N_\calS$-ary $\V$-monads on $\V^\calS$ coincide with the $\N_\calS$-nervous $\V$-monads on $\V^\calS$.
Given a $\V$-monad $\T$ on $\V^\calS$, we write $\T\Alg$ for the $\V$-category of $\T$-algebras, which may be regarded as a $\V$-category over $\V^\calS$ by way of the forgetful $\V$-functor $U^\T : \T\Alg \to \V^\calS$. 

Recall from the Introduction that an \emph{$\N_\calS$-theory} \cite[Definition 3.1]{Struct} is a $\V$-category $\mathfrak{T}$ equipped with an identity-on-objects $\V$-functor $\tau : \N_\calS^\op \to \mathfrak{T}$ satisfying the condition that each $\mathfrak{T}(J, \tau-) : \N_\calS^\op \to \V$ ($J \in \ob\N_\calS$) is a \emph{nerve} for the inclusion $j : \N_\calS \hookrightarrow \V^\calS$, meaning that $\mathfrak{T}(J, \tau-) \cong \V^\calS(j-, X)$ for some $X \in \ob\V^\calS$. We write $\Th_{\N_\calS}\left(\V^\calS\right)$ for the (ordinary) category of $\N_\calS$-theories. A \emph{(concrete) $\frakT$-algebra} is an object $A$ of $\V^\calS$ equipped with a $\V$-functor $M : \frakT \to \V$ satisfying $M \circ \tau = \V^\calS(j-, A) : \N_\calS^\op \to \V$. There is a $\V$-category $\frakT\Alg$ of $\frakT$-algebras and a forgetful $\V$-functor $U^\frakT : \frakT\Alg \to \V^\calS$ (see \cite[Definition 3.2]{Struct}).

Since the symmetric monoidal closed category $\V$ is topological over $\Set$, it is a \emph{locally bounded closed category} (in the sense of \cite[\S 6.1]{Kelly} and \cite[Definition 5.1]{locbd}) by \cite[Proposition 5.13(1)]{locbd}, and $\V^\calS$ is a \emph{$\V$-sketchable $\V$-category} in the sense of \cite[5.3.2]{Struct}. The small subcategory of arities $\N_\calS \hookrightarrow \V^\calS$ is therefore \emph{strongly amenable} (in the sense of \cite[Definition 3.12]{Struct}) by \cite[Theorem 5.3.12]{Struct}. We then deduce from \cite[Theorem 4.13]{Struct} that there is an equivalence $\Th_{\N_\calS}\left(\V^\calS\right) \simeq \Mnd_{\N_\calS}\left(\V^\calS\right)$ between the category of $\N_\calS$-theories and the category of $\N_\calS$-nervous $\V$-monads on $\V^\calS$, with the property that the $\V$-category of algebras for an $\N_\calS$-theory is isomorphic (over $\V^\calS$) to the $\V$-category of algebras for the corresponding $\N_\calS$-nervous $\V$-monad. 

The $\V$-category $\V^\calS$ is also a \emph{locally bounded $\V$-category} (in the sense of \cite[Definition 4.25]{locbd}) by \cite[Proposition 4.29]{locbd}. By \cite[Proposition 6.1.14]{Pres}, the small subcategory of arities $\N_\calS \hookrightarrow \V^\calS$ is therefore \emph{bounded} in the sense of \cite[Definition 6.1.12]{Pres}. If $(\V, \tensor, I) = (\V, \times, 1)$, so that $\V$ is cartesian closed, then the subcategory of arities $\N_\calS \hookrightarrow \V^\calS$ is also eleutheric by Theorem \ref{pi_category_thm}.  
\end{para} 

\begin{theo}
\label{mnd_pres_thm}
Suppose that $(\V, \tensor, I) = (\V, \times, 1)$, so that $\V$ is cartesian closed.
\begin{enumerate}[leftmargin=*]
\item For each $\V$-enriched $\calS$-sorted equational theory $\calT$, there is an $\N_\calS$-ary $\V$-monad $\T_\calT$ on $\V^\calS$ such that $\calT\Alg \cong \T_\calT\Alg$ in $\VCAT/\V^\calS$.

\item For each $\N_\calS$-ary $\V$-monad $\T$ on $\V^\calS$, there is a $\V$-enriched $\calS$-sorted equational theory $\calT_\T$ such that $\T\Alg \cong \calT_\T\Alg$ in $\VCAT/\V^\calS$.

\item For each $\N_\calS$-ary $\V$-monad $\T$ on $\V^\calS$, there is an $\N_\calS$-theory $\frakT_\T$ such that $\T\Alg \cong \frakT_\T\Alg$ in $\VCAT/\V^\calS$. For each $\N_\calS$-theory $\frakT$, there is an $\N_\calS$-ary $\V$-monad $\T_\frakT$ on $\V^\calS$ such that $\frakT\Alg \cong \T_\frakT\Alg$ in $\VCAT/\V^\calS$. 
\end{enumerate}
\end{theo}

\begin{proof}
(3) immediately follows from the discussion in \ref{Jary_para}. For (1), there is by Proposition \ref{equiv_presentations_prop1} a diagrammatic $\N_\calS$-presentation $\scrT_\calT$ such that $\calT\Alg \cong \scrT_\calT\Alg$ in $\VCAT/\V^\calS$. Since the subcategory of arities $\N_\calS \hookrightarrow \V^\calS$ is bounded and eleutheric \eqref{Jary_para}, by \cite[Theorem 5.20]{EP} there is an $\N_\calS$-ary $\V$-monad $\T_\calT$ on $\V^\calS$ such that $\T_\calT\Alg \cong \scrT_\calT\Alg \cong \calT\Alg$ in $\VCAT/\V^\calS$, as desired. For (2), again because $\N_\calS$ is bounded and eleutheric, by \cite[Proposition 5.24]{EP} there is a diagrammatic $\N_\calS$-presentation $\scrT$ such that $\T\Alg \cong \scrT\Alg$ in $\VCAT/\V^\calS$. Then from Proposition \ref{equiv_presentations_prop2} we obtain a $\V$-enriched $\calS$-sorted equational theory $\calT_\T$ such that $\T\Alg \cong \scrT\Alg \cong \calT_\T\Alg$ in $\VCAT/\V^\calS$, as desired.
\end{proof}

\begin{rmk}
\label{amenable_rmk}
Results in forthcoming work \cite{Ax} will ultimately enable us to omit the assumption in Theorem \ref{mnd_pres_thm} that the symmetric monoidal structure of $\V$ is cartesian, provided that we replace $\N_\calS$-ary $\V$-monads in Theorem \ref{mnd_pres_thm} with the $\N_\calS$-nervous $\V$-monads of \ref{Jary_para} (which generally do not coincide with $\N_\calS$-ary $\V$-monads when the subcategory of arities $\N_\calS \hookrightarrow \V^\calS$ is not eleutheric). 
\end{rmk}

\section{Appendix: An inventory of some prominent topological categories over $\Set$}
\label{appendix}

In this Appendix, expanding on Example \ref{examples_ex}, we present a (far from exhaustive) list of examples of some prominent topological categories over $\Set$. We point out which examples are cartesian closed, since our strongest results apply to such examples (see \S\ref{sig_alg_subsection_cc} and \S\ref{cong_section}). For each example, we explicitly describe the final lifts of structured sinks (and structured surjections\footnote{By a \emph{structured surjection} in a concrete category $\V$ over $\Set$, we mean a surjective function $g : |V| \to S$ with $V \in \ob\V$ and $S$ a set. The final lift of a structured surjection is a quotient morphism \eqref{top_para}.}), as well as the complete lattice structure of the fibres, since these constructions feature prominently in some of our main results, such as Theorems \ref{explicit_free_fibre_prop_2} and \ref{free_alg_cc}. 

\begin{ex}[\textbf{Topological spaces}]
\label{top_ex}
Perhaps the canonical example of a topological category over $\Set$ is the category $\Top$ of topological spaces and continuous functions, equipped with the forgetful functor $|-| : \Top \to \Set$ that sends a topological space to its underlying set. It is well known that $\Top$ is not cartesian closed, but it is symmetric monoidal closed when equipped with the canonical symmetric monoidal closed structure of \ref{sym_mon_para}. 
\begin{itemize}[leftmargin=*]
\item \textbf{Final lifts of structured sinks:} A sink $(g_i : X_i \to X)_{i \in I}$ in $\Top$ is final iff $X$ has the \emph{final topology} induced by the $g_i$ ($i \in I$), meaning that a subset $U \subseteq |X|$ is open in $X$ iff $g_i^{-1}(U)$ is open in $X_i$ for all $i \in I$. The final lift of a structured sink $(g_i : |X_i| \to S)_{i \in I}$ equips the set $S$ with the final topology induced by the $g_i$ ($i \in I$).

\item \textbf{Complete lattice structure of fibres:} The fibre $\Fib(S) = \Fib_\Top(S)$ over a set $S$ can be identified with the partially ordered set of topologies on $S$, ordered by reverse inclusion. Given a family $\left(\calO_i\right)_{i \in I}$ of topologies on $S$, the infimum $\bigwedge_{i \in I} \calO_i$ in $\Fib(S)$ is the topology on $S$ generated by $\bigcup_{i \in I} \calO_i$, while the supremum $\bigvee_{i \in I} \calO_i$ in $\Fib(S)$ is the intersection $\bigcap_{i \in I} \calO_i$. 
\end{itemize} 
\end{ex}

\begin{ex}[\textbf{$\calC$-generated topological spaces}] 
\label{topc_ex}
The following example comes from \cite{EscardoCCC}. Let $\calC$ be a fixed class of topological spaces, called the \emph{generating spaces}. A topological space $X$ is \emph{$\calC$-generated} if $X$ has the final topology induced by all continuous functions into $X$ from spaces in $\calC$. We write $\Top_\calC \hookrightarrow \Top$ for the full subcategory of $\Top$ consisting of the $\calC$-generated spaces, which is (\emph{concretely}) coreflective. Given a topological space $X$, its $\calC$-generated coreflection $\calC X$ is the set $|X|$ equipped with the final topology induced by all continuous functions into $X$ from spaces in $\calC$, and the counit morphism is the identity function $\calC X \to X$. Since $\Top$ is topological over $\Set$, it follows from \cite[Theorem 21.33]{AHS} that $\Top_\calC$ is topological over $\Set$ when equipped with the restricted forgetful functor $|-| : \Top_\calC \to \Set$. When the class $\calC$ of generating spaces is \emph{productive} in the sense of \cite[Definition 3.5]{EscardoCCC}, the category $\Top_\calC$ is cartesian closed by \cite[Theorem 3.6]{EscardoCCC}. By \cite[Definition 3.3]{EscardoCCC}, examples of $\Top_\calC$ with $\calC$ productive include the categories of compactly generated spaces, core compactly generated spaces, locally compactly generated spaces, and sequentially generated spaces. 
\begin{itemize}[leftmargin=*]
\item \textbf{Final lifts of structured sinks:} The final lift of a structured sink $(g_i : |X_i| \to S)_{i \in I}$ in $\Top_\calC$ equips the set $S$ with the final topology induced by the $g_i$ ($i \in I$), which is $\calC$-generated because each $X_i$ ($i \in I$) is $\calC$-generated. 

\item \textbf{Complete lattice structure of fibres:} The fibre $\Fib(S) = \Fib_{\Top_\calC}(S)$ over a set $S$ can be identified with the partially ordered set of $\calC$-generated topologies on $S$, ordered by reverse inclusion. Given a family $\left(\calO_i\right)_{i \in I}$ of $\calC$-generated topologies on $S$, the infimum $\bigwedge_{i \in I} \calO_i$ in $\Fib(S)$ is the $\calC$-generated coreflection of the topology on $S$ generated by $\bigcup_{i \in I} \calO_i$, while the supremum $\bigvee_{i \in I} \calO_i$ in $\Fib(S)$ is the intersection $\bigcap_{i \in I} \calO_i$. 
\end{itemize} 
\end{ex}  

\begin{ex}[\textbf{Measurable spaces}]
\label{meas_ex}
The category $\Meas$ of measurable spaces and measurable functions is topological over $\Set$ when equipped with the forgetful functor $|-| : \Meas \to \Set$ that sends a measurable space to its underlying set (see e.g.~\cite[\S 2.1]{Sato}). We can describe final lifts of structured sinks and the complete lattice structure of fibres exactly as for $\Top$ in Example \ref{top_ex} above, by everywhere replacing ``topology'' with ``$\sigma$-algebra'' and ``open'' with ``measurable''. It is known that $\Meas$ is not cartesian closed, but it is symmetric monoidal closed when equipped with the canonical symmetric monoidal closed structure of \ref{sym_mon_para} (which is specifically unpacked for $\Meas$ in \cite[\S 2]{Sato}). \end{ex}

\begin{ex}[\textbf{Relational Horn theories without equality}] 
\label{rel_ex}
Given a \emph{relational Horn theory $\T$ without equality} (in the precise sense of \cite[Definition 3.5]{Extensivity}), the category $\T\Mod$ of $\T$-models and their morphisms is topological over $\Set$. To explain the details of this example (essentially using the notations and terminology of \cite{Monadsrelational, Extensivity}), a \emph{(finitary) relational signature} is a set $\Pi$ of \emph{relation symbols} equipped with an assignment to each relation symbol $R \in \Pi$ of a natural number \emph{arity} $\arity(R) \geq 0$.\footnote{One can also consider more general (infinitary) relational signatures, where the arities of relation symbols are permitted to be (possibly infinite) cardinal numbers; see e.g.~\cite{Rosickyconcrete}.} Given a set $S$, a \emph{$\Pi$-edge in $S$} is an element of the set-theoretic coproduct $\coprod_{R \in \Pi} S^{\arity(R)}$, and is thus a tuple $\left(R, \left(s_1, \ldots, s_{\arity(R)}\right)\right)$ consisting of a relation symbol $R \in \Pi$ and a tuple $\left(s_1, \ldots, s_{\arity(R)}\right) \in S^{\arity(R)}$. We now fix a regular cardinal $\lambda$ and a set $\Var$ of variables of cardinality $\lambda$. A \emph{$\Pi$-formula} is an expression of the form $R v_1 \ldots v_{\arity(R)}$ for a relation symbol $R \in \Pi$ and variables $v_1, \ldots, v_{\arity(R)} \in \Var$ (and thus can be identified with a $\Pi$-edge in the set $\Var$). Given a set $\sfE$ of $\Pi$-edges in a set $S$, a $\Pi$-formula $Rv_1 \ldots v_{\arity(R)}$, and a function $\kappa : \Var \to S$ (also called a \emph{valuation (in $S$)}), we say that $\sfE$ \emph{satisfies $Rv_1 \ldots v_{\arity(R)}$ with respect to $\kappa$} and write $(\sfE, \kappa) \models Rv_1 \ldots v_{\arity(R)}$ if $\sfE$ contains the $\Pi$-edge $\left(R, \left(\kappa(v_1), \ldots, \kappa\left(v_{\arity(R)}\right)\right)\right)$. 

A \emph{$\lambda$-ary Horn formula without equality over $\Pi$} is an expression of the form $\Phi \Longrightarrow \psi$, where $\Phi$ is a set of $\Pi$-formulas of cardinality $< \lambda$ and $\psi$ is a $\Pi$-formula. A set $\sfE$ of $\Pi$-edges in a set $S$ \emph{satisfies $\Phi \Longrightarrow \psi$} provided that, for each valuation $\kappa : \Var \to S$, we have $(\sfE, \kappa) \models \psi$ if $(\sfE, \kappa) \models \varphi$ for all $\varphi \in \Phi$. A \emph{$\lambda$-ary relational Horn theory without equality (over $\Pi$)} is a set $\T$ of $\lambda$-ary Horn formulas without equality over $\Pi$ (called the \emph{axioms} of $\T$). Given a set $S$, a \emph{$\T$-model structure on $S$} is a set $\sfE$ of $\Pi$-edges in $S$ that satisfies each axiom of $\T$, and a \emph{$\T$-model $X$} is a set $|X|$ equipped with a $\T$-model structure $\sfE(X)$ on $|X|$. For a $\T$-model $X$, we may also write $X \models Rx_1 \ldots x_{\arity(R)}$ to mean that $\sfE(X)$ contains the $\Pi$-edge $\left(R, \left(x_1, \ldots, x_{\arity(R)}\right)\right)$. Given $\T$-models $X$ and $Y$, a \emph{morphism of $\T$-models} $f : X \to Y$ is a function $f : |X| \to |Y|$ that sends $\Pi$-edges in $X$ to $\Pi$-edges in $Y$, i.e.~such that $X \models Rx_1 \ldots x_{\arity(R)}$ implies $Y \models Rf\left(x_1\right) \ldots f\left(x_{\arity(R)}\right)$. We write $\T\Mod$ for the category of $\T$-models and their morphisms, which is equipped with the forgetful functor $|-| : \T\Mod \to \Set$ that sends a $\T$-model to its underlying set. By a \emph{relational Horn theory without equality}, we mean a $\lambda$-ary relational Horn theory without equality for some regular cardinal $\lambda$.   

Given a relational Horn theory $\T$ without equality (over a relational signature $\Pi$), the category $\T\Mod$ is topological over $\Set$ (as was first shown by Rosick\'{y} \cite[Proposition 5.1]{Rosickyconcrete}). $\T\Mod$ is not cartesian closed in general, but it is cartesian closed under suitable conditions on $\T$ (see \cite[Theorem 6.15]{Exp_relational}). The canonical symmetric monoidal closed structure on $\T\Mod$ \eqref{sym_mon_para} was previously considered in \cite[Corollary 3.13]{Monadsrelational} (see also \cite[3.12]{Exp_relational}). 
\begin{itemize}[leftmargin=*]
\item \textbf{Final lifts of structured sinks:} A sink $(g_i : X_i \to X)_{i \in I}$ in $\T\Mod$ is final iff $X$ has the \emph{final $\T$-model structure} induced by the $g_i$ ($i \in I$), meaning that $\sfE(X)$ is the smallest $\T$-model structure on $|X|$ for which each $g_i : X_i \to X$ a morphism of $\T$-models. In particular, the final lift of a structured sink $(g_i : |X_i| \to S)_{i \in I}$ equips the set $S$ with the final $\T$-model structure induced by the $g_i$ ($i \in I$).

\item \textbf{Complete lattice structure of fibres:} The fibre $\Fib(S) = \Fib_{\T\Mod}(S)$ over a set $S$ can be identified with the partially ordered set of all $\T$-model structures on $S$, ordered by inclusion. Given a family $\left(\sfE_i\right)_{i \in I}$ of $\T$-model structures on $S$, the infimum $\bigwedge_{i \in I} \sfE_i$ in $\Fib(S)$ is the intersection $\bigcap_{i \in I} \sfE_i$, while the supremum $\bigvee_{i \in I} \sfE_i$ is the smallest $\T$-model structure on $S$ that contains each $\sfE_i$ ($i \in I$).  
\end{itemize} 
Three prominent examples of $\T\Mod$ for a relational Horn theory $\T$ without equality are the category $\mathsf{Rel}$ of sets equipped with a binary relation, the category $\Preord$ of preordered sets and monotone functions, and the category $\PMet$ of (extended) pseudo-metric spaces and non-expansive mappings. For a more comprehensive list of examples, see \cite[Example 3.5]{Monadsrelational} or \cite[Example 3.7]{Extensivity}.
\end{ex}

\begin{ex}[\textbf{Quasispaces, a.k.a.~concrete sheaves, on concrete sites}]
\label{quasisp_ex}
The following example originates from \cite{Dubucquasitopoi}: the category of \emph{quasispaces} (or \emph{concrete sheaves}) on a \emph{concrete site} is topological over $\Set$ (such a category is also known as a \emph{concrete quasitopos}). First, a category $\bbC$ (not necessarily small) is \emph{well-pointed} if it has a terminal object $1$ such that $\bbC(1, -) : \bbC \to \Set$ is faithful, thus exhibiting $\bbC$ as a concrete category over $\Set$. A \emph{concrete site} $(\bbC, \calJ)$ is then a well-pointed category $\bbC$ equipped with a Grothendieck topology $\calJ$ whose covering families are all \emph{jointly surjective}.\footnote{Given a concrete category $\V$ over $\Set$, we say that a sink $\left(g_i : V_i \to V\right)_{i \in I}$ in $\V$ is \emph{jointly surjective} if the underlying sink $\left(g_i : |V_i| \to |V|\right)_{i \in I}$ in $\Set$ is jointly surjective.} Given a concrete site $(\bbC, \calJ)$ and a set $S$, a \emph{plot (in $S$)}\footnote{This terminology comes from \cite{Baezsmooth}.} is a function $p : |C| \to S$ for some object $C$ of $\bbC$. A class $\calQ$ of plots in $S$ is \emph{admissible} if it satisfies the following conditions:
\begin{itemize}[leftmargin=*]
\item $\calQ$ contains all constant plots; i.e.~for each object $C$ of $\bbC$, every constant plot $|C| \to S$ belongs to $\calQ$. 
\item $\calQ$ is closed under precomposition with morphisms of $\bbC$; i.e.~for each morphism $h : C \to D$ of $\bbC$ and each plot $p : |D| \to S$ in $\calQ$, the plot $p \circ |h| : |C| \to S$ belongs to $\calQ$.  
\item $\calQ$ satisfies the \emph{sheaf condition} (or \emph{gluing condition}): for each object $C$ of $\bbC$ and each covering family $\left(h_j : C_j \to C\right)_{j \in J} \in \calJ(C)$, a plot $p : |C| \to S$ belongs to $\calQ$ if each plot $p \circ \left|h_j\right| : \left|C_j\right| \to S$ ($j \in J$) belongs to $\calQ$. 
\end{itemize}
A \emph{quasispace $X$} (\emph{on} $(\bbC, \calJ)$) is a set $|X|$ equipped with an admissible class of plots $\calQ_X$ in $|X|$. 
Given quasispaces $X$ and $Y$, a \emph{morphism of quasispaces} $f : X \to Y$ is a function $f : |X| \to |Y|$ that sends plots to plots, i.e.~for each object $C$ of $\bbC$ and each plot $p : |C| \to |X|$ in $\calQ_X$, the plot $f \circ p : |C| \to |Y|$ belongs to $\calQ_Y$. We write $\QuasiSp(\bbC, \calJ)$ for the category of quasispaces on $(\bbC, \calJ)$ and their morphisms, which is equipped with the forgetful functor $|-| : \QuasiSp(\bbC, \calJ) \to \Set$ that sends a quasispace to its underlying set. The concrete category $\QuasiSp(\bbC, \calJ)$ is topological over $\Set$ (see \cite[Theorem 1.7]{Dubucquasitopoi}), as we now recall and explain. Categories of quasispaces on concrete sites are always cartesian closed (even locally cartesian closed).
  
\begin{itemize}[leftmargin=*]
\item \textbf{Final lifts of structured sinks:} First, we have the following explicit characterizations of (jointly surjective) final sinks and quotient morphisms in $\QuasiSp(\bbC, \calJ)$.
\end{itemize}
\begin{prop}
\label{final_sinks_quasispaces}
Let $(\bbC, \calJ)$ be a concrete site.
\begin{enumerate}[leftmargin=*]
\item A sink $\left(g_i : X_i \to X\right)_{i \in I}$ in $\QuasiSp(\bbC, \calJ)$ is final iff $\calQ_X$ satisfies the following condition: a plot $p : |C| \to |X|$ ($C \in \ob\bbC$) belongs to $\calQ_X$ iff there is a covering family $\left(h_j : C_j \to C\right)_{j \in J} \in \calJ(C)$ such that for each $j \in J$, either $p \circ \left|h_j\right| : \left|C_j\right| \to |X|$ is constant, or there are some $i_j \in I$ and $p_{j} : \left|C_j\right| \to \left|X_{i_j}\right|$ in $\calQ_{X_{i_j}}$ such that $p \circ \left|h_j\right| = g_{i_j} \circ p_{j}$. The final lift of a structured sink $\left(g_i : \left|X_i\right| \to S\right)_{i \in I}$ equips the set $S$ with the admissible class of plots defined by this condition. 

\item A jointly surjective sink $\left(g_i : X_i \to X\right)_{i \in I}$ in $\QuasiSp(\bbC, \calJ)$ is final iff $\calQ_X$ satisfies the following condition: a plot $p : |C| \to |X|$ ($C \in \ob\bbC$) belongs to $\calQ_X$ iff there is a covering family $\left(h_j : C_j \to C\right)_{j \in J} \in \calJ(C)$ such that for each $j \in J$, there are some $i_j \in I$ and $p_{j} : \left|C_j\right| \to \left|X_{i_j}\right|$ in $\calQ_{X_{i_j}}$ such that $p \circ \left|h_j\right| = g_{i_j} \circ p_{j}$. The final lift of a jointly surjective structured sink $\left(g_i : \left|X_i\right| \to S\right)_{i \in I}$ equips the set $S$ with the admissible class of plots defined by this condition. 

\item In particular, a surjective morphism $g : X \to Y$ of $\QuasiSp(\bbC, \calJ)$ is final (i.e.~is a quotient morphism) iff $\calQ_Y$ satisfies the following condition: a plot $p : |C| \to |Y|$ ($C \in \ob\bbC$) belongs to $\calQ_Y$ iff there is a covering family $\left(h_j : C_j \to C\right)_{j \in J} \in \calJ(C)$ such that for each $j \in J$, there is some $p_{j} : \left|C_j\right| \to \left|X\right|$ in $\calQ_{X}$ such that $p \circ \left|h_j\right| = g \circ p_{j}$. The final lift of a structured surjection $g : |X| \to S$ equips the set $S$ with the admissible class of plots defined by this condition. 
\end{enumerate}
\end{prop}

\begin{proof}
A slight extension of the proof of \cite[Proposition 4.12]{Dubuc_espanol} yields (a), from which (b) follows (but see also \cite[Proposition 1.6]{Dubucquasitopoi}). 
\end{proof}

\begin{itemize}[leftmargin=*]
\item \textbf{Complete lattice structure of fibres:} The fibre $\Fib(S) = \Fib_{\QuasiSp(\bbC, \calJ)}(S)$ over a set $S$ can be identified with the partially ordered class of all admissible classes of plots in $S$, ordered by inclusion. Given a family $\left(\calQ_i\right)_{i \in I}$ of admissible classes of plots in $S$, the infimum $\bigwedge_{i \in I} \calQ_i$ in $\Fib(S)$ is the intersection $\bigcap_{i \in I} \calQ_i$, while the supremum $\bigvee_{i \in I} \calQ_i$ in $\Fib(S)$ is the smallest admissible class of plots in $S$ that contains the union $\bigcup_{i \in I} \calQ_i$. Suprema in $\Fib(S)$ have the following more explicit description obtained from Proposition \ref{final_sinks_quasispaces} (in view of \ref{fibre_para}).
\end{itemize}

\begin{prop}
\label{suprema_concrete_sheaf}
Let $\left(\calQ_i\right)_{i \in I}$ be a family of admissible classes of plots in a set $S$. When $I \neq \varnothing$, the supremum $\bigvee_{i \in I} \calQ_i$ in $\Fib(S)$ has the following explicit description: a plot $p : |C| \to S$ ($C \in \ob\bbC$) belongs to $\bigvee_{i \in I} \calQ_i$ iff there is a covering family $\left(h_j : C_j \to C\right)_{j \in J} \in \calJ(C)$ such that for each $j \in J$, there are some $i_j \in I$ and $p_{j} : \left|C_j\right| \to S$ in $\calQ_{i_j}$ such that $p \circ \left|h_j\right| = p_{j}$. 

When $I = \varnothing$, the bottom element $\bigvee \varnothing$ of $\Fib(S)$ (i.e.~the discrete quasispace on $S$) consists of precisely the \emph{locally constant} plots in $S$, i.e.~a plot $p : |C| \to S$ ($C \in \ob\bbC$) belongs to $\bigvee \varnothing$ iff there is a covering family $\left(h_j : C_j \to C\right)_{j \in J} \in \calJ(C)$ such that each $p \circ \left|h_j\right| : \left|C_j\right| \to S$ ($j \in J$) is constant. \qed  
\end{prop}

Prominent examples of categories of quasispaces on concrete sites (from \cite{Dubucquasitopoi}) include the categories of convergence spaces, subsequential spaces, bornological sets, pseudotopological spaces, and quasitopological spaces. Other examples (from \cite{Baezsmooth}) include the categories of diffeological spaces, Chen spaces, and (abstract) simplicial complexes, while \cite{Heunenprobability} provides the example of quasi-Borel spaces. Some further examples (such as the example of \emph{quantum sets}) are considered in \cite{Concrete_cats_recursion}, while the example of \emph{$C$-spaces} is studied in \cite{Escardo_Xu}.
\end{ex} 

\begin{ex}[\textbf{Abstract simplicial complexes}]
\label{simp_ex}
Although the category of (abstract) simplicial complexes can be presented as the category of quasispaces on a certain concrete site (see e.g.~\cite[Proposition 4.18]{Baezsmooth}), it is more usually presented as follows. 
Given a set $S$, a \emph{simplex (in $S$)} is a non-empty finite subset of $S$. A set $\mathscr{S}$ of simplices in $S$ is \emph{admissible} if it contains all singleton simplices and is downward closed (i.e.~if $s, s'$ are simplices in $S$ with $s' \subseteq s$ and $s \in \scrS$, then $s' \in \scrS$). A(n) \emph{(abstract) simplicial complex} $X$ is a set $|X|$ equipped with an admissible set of simplices $\scrS_X$ in $|X|$. Given simplicial complexes $X$ and $Y$, a \emph{morphism of simplicial complexes $f : X \to Y$} is a function $f : |X| \to |Y|$ that sends simplices to simplices, i.e.~for each $s \in \scrS_X$ we have $f[s] \in \scrS_Y$. We write $\Simp$ for the category of simplicial complexes and their morphisms, which is equipped with the forgetful functor $|-| : \Simp \to \Set$ that sends a simplicial complex to its underlying set. The category $\Simp$ is topological over $\Set$, and it is cartesian closed (even locally cartesian closed).
\begin{itemize}[leftmargin=*]  
\item \textbf{Final lifts of structured sinks:} A sink $(g_i : X_i \to X)_{i \in I}$ in $\Simp$ is final iff $\scrS_X$ satisfies the following condition: a simplex $s \subseteq |X|$ belongs to $\scrS_X$ iff $s$ is a singleton or $s \subseteq g_i\left[s_i\right]$ for some $i \in I$ and simplex $s_i \in \scrS_{X_i}$. 
The final lift of a structured sink $(g_i : |X_i| \to S)_{i \in I}$ equips the set $S$ with the admissible set of simplices defined by this condition. A jointly surjective sink $(g_i : X_i \to X)_{i \in I}$ in $\Simp$ is final iff $\scrS_X$ satisfies the following condition: a simplex $s \subseteq |X|$ belongs to $\scrS_X$ iff $s \subseteq g_i\left[s_i\right]$ for some $i \in I$ and simplex $s_i \in \scrS_{X_i}$. 
The final lift of a jointly surjective structured sink $(g_i : |X_i| \to S)_{i \in I}$ equips the set $S$ with the admissible set of simplices defined by this condition. In particular, a surjective morphism $g : X \to Y$ of $\Simp$ is final (i.e.~is a quotient morphism) iff $\scrS_Y$ satisfies the following condition: a simplex $s \subseteq |Y|$ belongs to $\scrS_Y$ iff $s \subseteq g\left[s'\right]$ for some simplex $s' \in \scrS_{X}$. 
The final lift of a structured surjection $g : |X| \to S$ equips the set $S$ with the admissible set of simplices defined by this condition. 

\item \textbf{Complete lattice structure of fibres:} The fibre $\Fib(S) = \Fib_{\Simp}(S)$ over a set $S$ can be identified with the partially ordered set of all admissible sets of simplices in $S$, ordered by inclusion. Given a family $\left(\scrS_i\right)_{i \in I}$ of admissible sets of simplices in $S$, the infimum $\bigwedge_{i \in I} \scrS_i$ in $\Fib(S)$ is the intersection $\bigcap_{i \in I} \scrS_i$. For $I \neq \varnothing$, the supremum $\bigvee_{i \in I} \scrS_i$ is the union $\bigcup_{i \in I} \scrS_i$, while the bottom element $\bigvee \varnothing$ of $\Fib(S)$ (i.e.~the discrete simplicial complex on $S$) consists of precisely the singleton simplices in $S$. 
\end{itemize}
\end{ex}

\begin{ex}[\textbf{Bornological sets}]
\label{born_ex}
Similarly, although the category of bornological sets can be presented as the category of quasispaces on a certain concrete site, it is more usually presented as follows. Given a set $S$, a \emph{bornology} on $S$ is a downward-closed set of subsets of $S$ that is closed under finite union and contains all singletons. A \emph{bornological set} $X$ is a set $|X|$ equipped with a bornology $\calB(X)$ on $|X|$; we refer to the elements of $\calB(X)$ as \emph{bounded subsets} of $X$. Given bornological sets $X$ and $Y$, a \emph{morphism of bornological sets $f : X \to Y$} is a function $f : |X| \to |Y|$ that sends bounded subsets to bounded subsets, i.e.~for each $U \in \calB(X)$ we have $f[U] \in \calB(Y)$. We write $\Born$ for the category of bornological sets and their morphisms, which is equipped with the forgetful functor $|-| : \Born \to \Set$ that sends a bornological set to its underlying set. The category $\Born$ is topological over $\Set$, and it is cartesian closed (even locally cartesian closed).
\begin{itemize}[leftmargin=*]
\item \textbf{Final lifts of structured sinks:} A sink $(g_i : X_i \to X)_{i \in I}$ in $\Born$ is final iff $\calB(X)$ satisfies the following condition: a subset $U \subseteq |X|$ belongs to $\calB(X)$ iff $U$ is a subset of a finite union of the $g_i\left[U_i\right]$ ($i \in I, U_i \in \calB\left(X_i\right)$) with a finite subset of $|X|$. 
The final lift of a structured sink $(g_i : |X_i| \to S)_{i \in I}$ equips the set $S$ with the bornology defined by this condition. In particular, a single morphism $g : X \to Y$ of $\Born$ is final iff $\calB(Y)$ satisfies the following condition: a subset $V \subseteq |Y|$ belongs to $\calB(Y)$ iff there are $U \in \calB(X)$ and a finite subset $\{y_1, \ldots, y_m\} \subseteq |Y|$ such that $V \subseteq g[U] \cup \{y_1, \ldots, y_m\}$ (because $\calB(X)$ is closed under finite union and unions are preserved under direct image). The final lift of a structured function $g : |X| \to S$ equips the set $S$ with the bornology defined by this condition. 

A jointly surjective sink $(g_i : X_i \to X)_{i \in I}$ in $\Born$ is final iff $\calB(X)$ satisfies the following condition: a subset $U \subseteq |X|$ belongs to $\calB(X)$ iff $U$ is a subset of a finite union of the $g_i\left[U_i\right]$ ($i \in I, U_i \in \calB\left(X_i\right)$). 
The final lift of a jointly surjective structured sink $(g_i : |X_i| \to S)_{i \in I}$ equips the set $S$ with the bornology defined by this condition. In particular, a surjective morphism $g : X \to Y$ of $\Born$ is final (i.e.~is a quotient morphism) iff $\calB(Y)$ satisfies the following condition: a subset $V \subseteq |Y|$ belongs to $\calB(Y)$ iff $V \subseteq g[U]$ for some $U \in \calB(X)$. 
The final lift of a structured surjection $g : |X| \to S$ equips the set $S$ with the bornology defined by this condition.  

\item \textbf{Complete lattice structure of fibres:} The fibre $\Fib(S) = \Fib_{\Born}(S)$ over a set $S$ can be identified with the partially ordered set of all bornologies on $S$, ordered by inclusion. Given a family $\left(\calB_i\right)_{i \in I}$ of bornologies on $S$, the infimum $\bigwedge_{i \in I} \calB_i$ in $\Fib(S)$ is the intersection $\bigcap_{i \in I} \calB_i$. For $I \neq \varnothing$, the supremum $\bigvee_{i \in I} \calB_i$ in $\Fib(S)$ is the bornology consisting of all subsets of finite unions of elements of $\bigcup_{i \in I} \calB_i$, while the bottom element $\bigvee \varnothing$ of $\Fib(S)$ (i.e.~the discrete bornological set on $S$) consists of precisely the finite subsets of $S$. 
\end{itemize}
\end{ex}

\bibliographystyle{amsalpha}
\bibliography{mybib}

\end{document}